\documentclass[final,a4paper,10pt]{article}
\usepackage{etex}
\usepackage{xparse}
\usepackage{amsmath,amssymb, amsthm}
\usepackage{nicefrac}
\usepackage[english]{babel}
\usepackage[a4paper,margin=2cm,footskip=.5cm]{geometry}
\usepackage{cite}
\usepackage[latin1]{inputenc}
\usepackage{enumitem}
\usepackage{latexsym}
\usepackage{mathtools}
\usepackage{mathrsfs}
\usepackage{textcomp}
\usepackage[pdftex]{graphicx}
\usepackage[font=small]{caption}
\usepackage{subfig}
\usepackage{empheq}
\usepackage{mdframed}
\usepackage{cite}
\usepackage{showkeys}
\usepackage{chngcntr}
\usepackage{algorithm}
\usepackage{algorithmicx}
\usepackage{amsthm}
\usepackage{algpseudocode}
\usepackage{stmaryrd}
\usepackage[usenames,dvipsnames,svgnames,table]{xcolor}
\usepackage{tikz}
\usetikzlibrary{backgrounds,shapes,arrows}
\usepackage{pgfplots,pgfplotstable}
\usepackage{epstopdf}

\pgfplotsset{compat=newest,
  width=0.33\textwidth,
  height=0.55\textwidth,
  grid style={dotted},
  ticklabel style = {font=\tiny},
  x label style={at={(axis description cs:0.5,-0.035)},anchor=north},
  y label style={at={(axis description cs:-0.155,.5)},anchor=south},
  title style={at={(axis description cs:0.5,1.06)},anchor=north},
  legend pos=south west,
  legend cell align=left,
  legend style = {inner xsep=0.8pt,inner ysep=0.4pt,font=\tiny,at={(axis description
      cs:0.015,0.006)}}
}

\usepackage{hyperref}
\definecolor{myblue}{rgb}{0,0,0.6}
\hypersetup{colorlinks=true, linkcolor=myblue,citecolor=myblue,filecolor=myblue,urlcolor=myblue}
\newcommand{\rhc}[1]{{\color{black}#1}}
\newcommand{\esc}[1]{{\color{black}{#1}}}
\newcommand{\cs}[1]{{\color{black}{#1}}}
\newcommand{\rev}[1]{{\color{black}{#1}}}

\usepackage{soul}

\newcommand{\wn}{k}
\newcommand{\uinc}{u^{\mathrm{inc}}}
\newcommand{\uscat}{u^{\mathrm{scat}}}
\newcommand{\ualt}{u^{\mathrm{alt}}}
\newcommand{\prms}{\Cp} 

\newcommand{\refs}[1]{\wh{#1}}
\newcommand{\Bsr}{\refs{\Bs}}
\newcommand{\rr}{\refs{\rho}}
\newcommand{\dtn}{\Op{DtN}}
\newcommand{\Derv}{\Op{D}}
\newcommand{\Dervxt}{\Op{D}_{\xrf}}
\newcommand{\xrf}{\refs{\Bx}}
\newcommand{\iu}{\rm i}
\newtheorem{theorem}{Theorem}[section]
\newtheorem{lemma}[theorem]{Lemma}
\newtheorem{proposition}[theorem]{Proposition}
\newtheorem{corollary}[theorem]{Corollary}
\newtheorem{assumption}[theorem]{Assumption}
\theoremstyle{definition}
\newtheorem{definition}[theorem]{Definition}
\newtheorem{remark}[theorem]{Remark}

\numberwithin{equation}{section}
\numberwithin{figure}{section}
\numberwithin{table}{section}
\newcommand{\eps}{\varepsilon}
\newcommand{\bsnu}{{\boldsymbol{\nu}}}
\newcommand{\bsmu}{{\boldsymbol{\mu}}}
\newcommand{\bschi}{{\boldsymbol{\chi}}}
\newcommand{\bsb}{{\boldsymbol{b}}}
\newcommand{\bse}{{\boldsymbol{e}}}
\newcommand{\pts}{{\mathrm{pts}}}
\newcommand{\be}{\begin{equation}}
  \newcommand{\ee}{\end{equation}}
\newcommand{\cA} {{\mathcal A}}

\newcommand{\cD} {{\mathcal D}}

\newcommand{\cF}{{\mathcal F}}

\newcommand{\cP}{{\mathcal P}}

\newcommand{\Vx}{{\mathbf{x}}}
\newcommand{\Vy}{{\mathbf{y}}}
\newcommand{\Vz}{{\mathbf{z}}}

\newcommand{\VA}{{\mathbf{A}}}

\newcommand{\VD}{{\mathbf{D}}}

\newcommand{\VF}{{\mathbf{F}}}

\newcommand{\VH}{{\mathbf{H}}}
\newcommand{\VI}{{\mathbf{I}}}

\newcommand{\VY}{{\mathbf{Y}}}

\usepackage{upgreek}

\newcommand{\Vzeta}      {\boldsymbol{\upzeta}}

\newcommand{\Vrho}       {\boldsymbol{\uprho}}

\newcommand{\BC}{{\boldsymbol{C}}}

\newcommand{\Bb}{{\boldsymbol{b}}}

\newcommand{\Bd}{{\boldsymbol{d}}}

\newcommand{\Bh}{{\boldsymbol{h}}}

\newcommand{\Bn}{{\boldsymbol{n}}}

\newcommand{\Bq}{{\boldsymbol{q}}}

\newcommand{\Bs}{{\boldsymbol{s}}}

\newcommand{\Bx}{{\boldsymbol{x}}}
\newcommand{\By}{{\boldsymbol{y}}}

\newcommand{\Bbeta}      {\boldsymbol{\beta}}

\newcommand{\nablabf}{\boldsymbol{\nabla}}

\newcommand{\Psibf}{\boldsymbol{\Psi}}

\newcommand{\Phibf}{\boldsymbol{\Phi}}

\newcommand{\Ca}{{\cal A}}

\newcommand{\Ch}{{\cal H}}

\newcommand{\Co}{{\cal O}}
\newcommand{\Cp}{{\cal P}}

\newcommand{\Cr}{{\cal R}}

\newcommand{\bbC}{\mathbb{C}}

\newcommand{\bbE}{\mathbb{E}}

\newcommand{\bbN}{\mathbb{N}}

\newcommand{\bbR}{\mathbb{R}}
\newcommand{\bbS}{\mathbb{S}}

\newcommand{\Fu}{\mathfrak{u}}

\renewcommand{\Re}{\operatorname{Re}}             
\renewcommand{\Im}{\operatorname{Im}}             

\newcommand*{\Op}[1]{\mathsf{#1}} 

\providecommand*{\wt}[1]{\widetilde{#1}}
\providecommand*{\wh}[1]{\widehat{#1}}

\providecommand*{\N}[1]{\left\|{#1}\right\|} 

\newcommand*{\rst}[2]{\left.\mathstrut#1\right|_{#2}}

\providecommand{\Supp}{\operatorname{supp}}                            
\providecommand{\supp}{\Supp}

\newcommand{\defaultdomain}{\Omega}
\newcommand{\defaultboundary}{\partial\defaultdomain}

\providecommand*{\Lp}[2][\defaultdomain]{L^{#2}({#1})}

\newcommand*{\NLp}[3][\defaultdomain]{\N{#2}_{\Lp[#1]{#3}}}

\newcommand*{\Ltwo}[1][\defaultdomain]{\Lp[#1]{2}}

\newcommand*{\NLtwo}[2][\defaultdomain]{\NLp[#1]{#2}{2}}

\NewDocumentCommand{\Hm}{O{\defaultdomain}md<>}{%
  \IfValueTF{#3}{H^{#2}_{#3}(#1)}{H^{#2}(#1)}}            

\newcommand*{\bHm}[3][\defaultdomain]{H_{#3}^{#2}({#1})}

\newcommand*{\NHm}[3][\defaultdomain]{{\N{#2}}_{\Hm[{#1}]{#3}}}

\newcommand*{\Hone}[1][\defaultdomain]{\Hm[#1]{1}}

\newcommand*{\NHone}[2][\defaultdomain]{{\N{#2}}_{\Hone[{#1}]}}

\newcommand{\hlb}{\frac{1}{2}}
\NewDocumentCommand{\Hh}{O{\defaultboundary}d<>}{%
  \IfValueTF{#2}{\bHm[#1]{\hlb}{#2}}{\Hm[#1]{\hlb}}}

\NewDocumentCommand{\Hmh}{O{\defaultboundary}d<>}{%
  \IfValueTF{#2}{\bHm[#1]{-\hlb}{#2}}{\Hm[#1]{-\hlb}}}

\newcommand{\contspcsymb}{C}
\NewDocumentCommand{\Contm}{sd<>O{\defaultdomain}m}{%
  \IfBooleanTF{#1}{\renewcommand{\contspcsymb}{\BC}}{\renewcommand{\contspcsymb}{C}}%
  \IfValueTF{#2}{\contspcsymb_{#2}^{#4}({#3})}{\contspcsymb^{#4}({#3})}}

\NewDocumentCommand{\pwContm}{sd<>O{\defaultdomain}m}{%
  \IfBooleanTF{#1}{\renewcommand{\contspcsymb}{\BC}}{\renewcommand{\contspcsymb}{C}}%
  \IfValueTF{#2}{\contspcsymb_{\mathrm{pw},#2}^{#4}({#3})}%
  {\contspcsymb_{\mathrm{pw}}^{#4}({#3})}}            

\DeclareDocumentCommand{\itg}{s O{\defaultdomain} o m O{\mathrm{d}\Bx}}{%
  \IfBooleanTF{#1}%
  {\IfNoValueTF{#3}{\int\nolimits_{#2}{#4}\,{#5}}{\int\nolimits_{#2}^{#3}{#4}\,{#5}}}%
  {\IfNoValueTF{#3}{\int\limits_{#2}{#4}\,{#5}}{\int\limits_{#2}^{#3}{#4}\,{#5}}}}%

\newcommand*{\cintv}[1]{[{#1}]}

\newcommand*{\cointv}[1]{[{#1}[}

\newcommand{\beq}{\begin{equation}}
  \newcommand{\eeq}{\end{equation}}
 \newcommand{\beqs}{\begin{equation*}}
   \newcommand{\eeqs}{\end{equation*}}
\newcommand{\rd}{{\rm d}}
\newcommand{\ri}{{\rm i}}
\newcommand{\re}{{\rm e}}

\newcommand{\Rea}{\mathbb{R}}
\newcommand{\bpf}{\begin{proof}}
\newcommand{\epf}{\end{proof}}

\newcommand{\Csolt}{C_{{\rm sol,} 2}}
\newcommand{\Csolo}{C_{{\rm sol,} 1}}

\usepackage{bbm}

\NewDocumentCommand{\NHmk}{O{\Omega}mD<>{1}}{\N{#2}^2_{H^{#3}_\wn (#1)}}
\newcommand{\unitx}{\Bq}
\newcommand{\unitxh}{\unitx}
\newcommand{\rhffp}{\widehat{u}^{\rm scat}_\infty}
\NewDocumentCommand{\rhffparg}{O{\rC}D<>{\unitxh}}{\widehat{u}^{\rm scat}_\infty(#1;#2)}
\NewDocumentCommand{\rhffph}{O{\rC}D<>{\unitxh}}{\widehat{u}^{\rm scat}_{\infty, \rm PML}(#1;#2)_h}
\NewDocumentCommand{\rhuhpml}{O{\rC}D<>{\cdot}}{\widehat{u}_{\rm PML}(#1;#2)_h}

\NewDocumentCommand{\rbffparg}{O{\rC}D<>{\unitxh}}{\breve{u}^{\rm scat}_\infty(#1;#2)}
\NewDocumentCommand{\rbffph}{O{\rC}D<>{\unitxh}}{\breve{u}^{\rm scat}_{\infty, \rm PML}(#1;#2)_h}
\NewDocumentCommand{\rbuhpml}{O{\rC}D<>{\cdot}}{\breve{u}_{\rm PML}(#1;#2)_h}
\title{Frequency-Explicit Shape Holomorphy \\
in Uncertainty Quantification for Acoustic Scattering}
\author{R.~Hiptmair\thanks{SAM, D-MATH, ETH Zurich, Switzerland,
    email: hiptmair@sam.math.ethz.ch} ,
  Ch.~Schwab\thanks{SAM, D-MATH, ETH Zurich, Switzerland,
    email: christoph.schwab@sam.math.ethz.ch} ,
  and E.~A.~Spence%
  \thanks{Department of Mathematical Sciences, University of Bath, UK, email: e.a.spence@bath.ac.uk}
}

\allowdisplaybreaks

\begin{document}
\maketitle

\begin{abstract}
  We consider frequency-domain acoustic scattering at a homogeneous star-shaped penetrable
  obstacle, whose shape is uncertain and modelled via a radial spectral parameterization
  with random coefficients. Using recent results on the stability of Helmholtz
  transmission problems with piecewise constant coefficients from $[${\sc A.~Moiola and
    E.~A. Spence}, {\em Acoustic transmission problems: wavenumber-explicit bounds and
    resonance-free regions}, Mathematical Models and Methods in Applied Sciences, 29
  (2019), pp.~317--354$]$ we obtain frequency-explicit statements on the holomorphic
  dependence of the scattered field and the far-field pattern on the stochastic shape
  parameters. This paves the way for applying general results on the efficient
  construction of high-dimensional surrogate models. We also take into account the effect of
  domain truncation by means of perfectly matched layers (PML).  
  In addition,
  spatial regularity estimates which are explicit in terms of the wavenumber $\wn$ 
  permit us to quantify the impact of
  finite-element Galerkin discretization using high-order Lagrangian finite-element
  spaces.
\end{abstract}


\section{Introduction}
\label{sec:intro}

%
\subsection{Scattering Transmission Problem}
\label{ss:stp}
We consider frequency-domain acoustic scattering by a homogeneous penetrable scatterer
occupying an open, bounded set $D\subset\bbR^{d}$, $d=2,3$, with Lipschitz boundary
$\partial D$, which is embedded in a homogeneous background medium occupying
$\bbR^d\backslash \overline{D}$.  This can be modeled mathematically by a
transmission problem for the Helmholtz equation with the Sommerfeld radiation conditions
at infinity\footnote{Vectors in Euclidean space are denoted by bold roman symbols,
  $\Bx\in \bbR^3$, and $|\Bx|$ stands for its Euclidean norm.}:
\begin{subequations}
  \label{eq:htp}
  \begin{align}
    \label{eq:htp1}
    \big(-\wn^{-2}\Delta  -n(\Bx)\big)u(\Bx) & = 0 &&
    \hspace{-2em}\text{in}\;\bbR^{d},\\
    \label{eq:htp2}
    |\Bx|^{\frac{d-1}{2}}
    \left(
      \wn^{-1}\frac{\partial}{\partial |\Bx|}-\iu
    \right)(u-\uinc)(\Bx) & \to 0 &&
    \hspace{-2em}\text{as}\;|\Bx|\to\infty, \text{ uniformly in $\Bx/|\Bx|$.}
  \end{align}
\end{subequations}
Here, $\wn>0$ is the \emph{wavenumber}, which is proportional to the angular frequency,
and $n=n(\Bx)$ is a spatially varying, but piecewise constant, 
\emph{index of refraction}, 
for which we assume
\begin{gather}
  \label{eq:n}
  n(\Bx) =
  \begin{cases}
    n_{i} > 0 & \text{for } \Bx\in D ,\\
    1 & \text{for }\Bx\in\bbR^{d}\setminus\overline{D}.
  \end{cases}
\end{gather}
In the acoustic modeling context, the solution $u$ of \eqref{eq:htp} gives the complex
amplitude of the sound pressure of the so-called \emph{scattered wave}, see
\cite[Sect~2.1]{COK13} for more detail.

Excitation is provided by an incident wave $\uinc$ that satisfies
$-\Delta \uinc - \wn^{2}\uinc=0$ in $\bbR^{d}$, the prime example being a plane wave
$\uinc(\Bx) = \exp(\ri \wn \Bd\cdot\Bx)$, propagating in direction $\Bd\in\bbR^{d}$,
$|\Bd|=1$. 
Existence and uniqueness of a solution $u\in H^{1}_{\mathrm{loc}}(\bbR^{d})$ of
\eqref{eq:htp} is well-known, see, e.g., \cite[Lemma 2.2]{MS19}.

In this article we focus on the model problem \eqref{eq:htp}--\eqref{eq:n} for the sake of
simplicity. Our analysis could also be extended to settings with more general piecewise
constant coefficient functions in the zero-order and second-order terms of the Helmholtz
equation \eqref{eq:htp1}, which would not require fundamentally new ideas.

\subsection{Uncertainty Quantification (UQ) via Polynomial Surrogate Models}
\label{ss:uqpoly}

The shape of the scatterer $D$ is assumed to be uncertain, which we take into account by
the customary approach to deterministic UQ through stochastic parameterization: the
uncertainty in $D$ is captured by introducing a dependence of $D$ on (possibly infinitely
many) real-valued random variables with known distributions, \cs{see, e.g., \cite{CCS15}
  for this approach and examples.}  
The key idea is to relegate those to 
\cs{real-valued, deterministic parameters, 
  and to endow the (possibly infinite, but countable) cartesian product of
  the parameter domains with probability measures. 
  We refer to
  \cite{CCS15,ChCChS14,JerezChSZech2017,AJSZ20_2734} 
  and the references therein for
  discussions and realizations of this idea. 
  In the present paper, 
we arrive at a family
of transmission problems depending on infinitely many parameters.  
Holomorphic dependence of solutions on the shape of the domain in which the 
PDE is set is available for a wide range of elliptic and parabolic PDEs, see
\cite{HSS15,JerezChSZech2017,CSZ18,HS21_2779,AJSZ20_2734} and the references there.
The results in these references were not explict in e.g. the wavenumber.
\emph{Frequency-explicit} holomorphy of solutions of time-harmonic, acoustic scattering
was developed for certain forward and problems acoustic scattering problems 
recently in \cite{DLM22,KUS24,GKN25,GKS20}.
The solution $u$ of the presently considered
\emph{transmission problem} and any derived quantity 
of interest will then become functions of the parameters, alike.
}

\cs{
Next we build sparse polynomial surrogate models of those functions 
on the parameter domain. 
This can be done accurately and efficiently using suitable spaces of multi-variate
polynomials in the parameters, provided that $u$ and the quantities of interest are
analytic/holomorphic\footnote{We use ``analytic'' and ``holomorphic'' as synonyms,
  \emph{cf.} \cite[Definition~5.1]{MUJ86}.} functions of the parameters with a
sufficiently large domain of analyticity in the extension of the parameter space 
into the complex domain; see \cite{CCS15} and the works cited there.
}
\subsection{Simplest Case: Size Uncertainty Quantification}
\label{ss:scuq}
We first consider the case that only the size of scatterer $D$ 
and not its shape is random, and model it by setting
\begin{gather}
  \label{eq:do}
  D = D(\omega) := (1+ \tfrac12 Y(\omega))D_{0}\;, \quad \omega \in \Omega,
\end{gather}
where
\begin{itemize}
\item $D_{0}\subset\bbR^{d}$ is a bounded Lipschitz domain 
  (the ``{nominal} scatterer''), and 
\item $Y:\Omega\to [-1,1]$ is a random variable on some probability space
  $\Omega$, whose details are not important at this stage.
\end{itemize}
We follow the policy outlined in the previous section and replace $Y(\omega)$ 
with a single real parameter $y\in \cintv{-1,1}$. 
This yields the \emph{parametric domain model}
\begin{gather}
  \label{eq:dy}
  D(y) := (1+\tfrac12 y)D_{0}\;,\quad -1\leq y \leq 1 \;.
\end{gather}
Simple scaling arguments show that, 
if $u=u(y)\in H^1_{\rm loc}(\Rea^d)$ solves the transmission problem \eqref{eq:htp}, 
then $\wh{u} = \wh{u}(y)\in H^1_{\rm loc}(\Rea^d)$
defined as
\begin{gather}
  \label{eq:pbu}
  \wh{u}(y;\refs{\Bx}) := u((1+\tfrac12 y)\refs{\Bx})\;,\quad \refs{\Bx}\in\Rea^{d}\;,
\end{gather}
solves \eqref{eq:htp} with 
(i) $D(y)$ replaced with $D_{0}$, 
(ii) a modified wave number
$\wh{\wn} = \wh{\wn}(y) := (1+\tfrac12 y)k$, and 
(iii) the exciting field
$\wh{u}_{\mathrm{inc}}(\refs{\Bx}) := u_{\mathrm{inc}}((1+\tfrac12 y)\refs{\Bx})$. 
We
realize that, up to an affine transformation, the dependence of $\wh{u}$ on the scaling
parameter $y$ is the same as its dependence on the wave number $\wn$.

The dependence of the solution of the Helmholtz transmission problem on the wave number
$\wn$ is a classic topic of study in scattering theory.  One key result is when $n_{i}>1$
and $D_{0}$ is smooth with strictly-positive curvature the norm of the solution operator
grows superalgebraically through an increasing sequence of wavenumbers \cite{PoVo:99}
\footnote{Strictly speaking, \cite{PoVo:99} proves the existence of a sequence of
  resonances exponentially close to the real axis, but then the ``resonances to
  quasimodes'' result of \cite{St:00} implies super-algebraic growth through a sequence of
  real $k_j$s.}; i.e., in the scaled setting above, there exist sequences of real wavenumbers
$( \wh{k}_{j})_{j\in\bbN}$, $\wh{k}_{j}\to\infty$ for $j\to\infty$, such that, given
$R, N>0$, there exists $C_N>0$ such that\footnote{We write $B_{R}$ for the open ball
  around $0$ in $\bbR^{d}$.}
\begin{gather}
  \label{eq:kj0}
  \N{\wh{u}}_{H^{1}(B_R)} \geq C_N (\wh{k}_j)^N 
  \quad\text{ for all }
  j\in\bbN.
\end{gather}
At least when $D_0$ is a ball, this growth is exponential, i.e., 
there exist $C, \gamma>0$ such that 
\begin{equation}
  \label{eq:kj}
  \N{\wh{u}}_{H^{1}(B_R)} \geq C \exp\big(\gamma \wh{k}_{j}\big)
  \quad\text{ for all }
  j\in\bbN;
\end{equation}
see \cite{Cap12, CLP12, AC16}. 
To explain this growth, recall that \emph{resonances} of
the Helmholtz transmission problem are poles of the meromorphic extension
$\wh{k}\in\bbC\mapsto \wh{u}(\wh{k})$. When $n_{i}>1$ and $D_{0}$ is smooth with
strictly-positive curvature there are resonances superalgebraically close to the real axis
\cite[Theorem 1.1]{PoVo:99}, corresponding physically to ``whispering-gallery modes"
created by total internal reflection of rays hitting $\partial D_0$ from $D_0$.  The
$\wh{k}_j$, often called \emph{quasi-resonances}, can then be thought of as the real parts
of these resonances close to the real axis; 
see Figure~\ref{fig:qr} (a) for a numerical illustration. 
We highlight also that the
density of the near-real-axis resonances, and hence also the $\wh{k_j}$, 
increases as $j$ increases \cite[Theorem 1.3]{CaPoVo:01}.

\begin{figure}[h]
  \subfloat[$n_{i} = 3$]{%
    \includegraphics[width=0.5\textwidth]{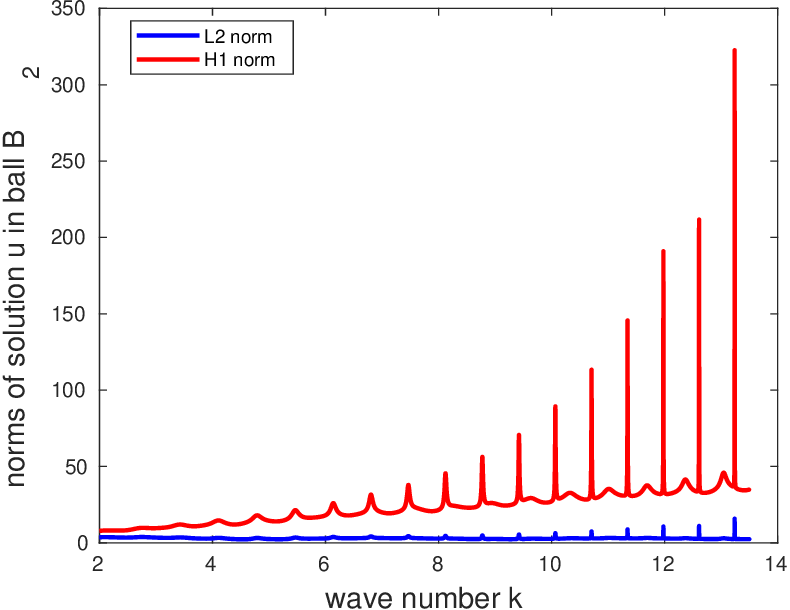}
  }
  \subfloat[$n_{i} = \frac{1}{3}$]{%
    \includegraphics[width=0.5\textwidth]{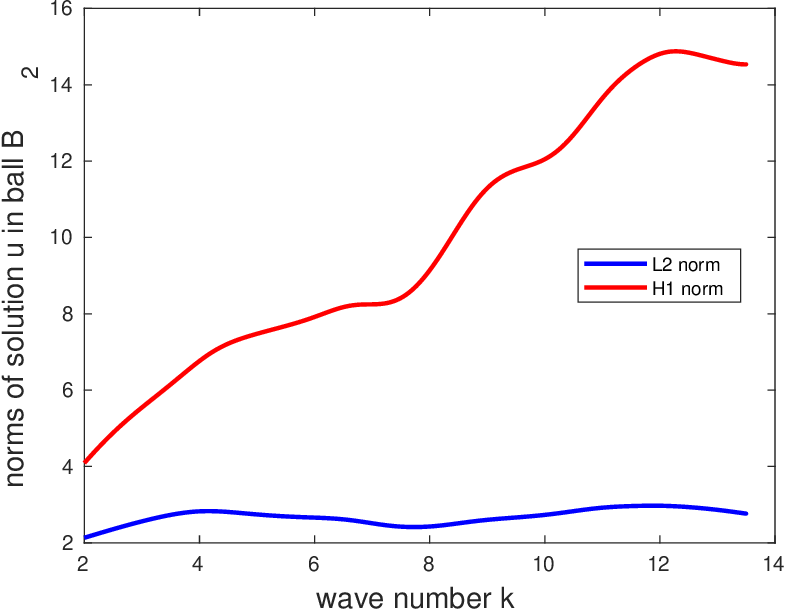}
  }
  
  \caption{{Dependence of norms ${\N{u}_{L^{2}(B_{2})}}$ and
      ${\N{u}_{H^{1}(B_{2})}}$ of the solution $u$ of the transmission problem
      \eqref{eq:htp} on $\wn$, when $d=2$, $D$ is the unit disk, and
      $u_{\mathrm{inc}}(\Bx)=\exp(\iu \wn x_{1})$. For $n_{i}=3$ quasi-resonances manifest
      themselves as spikes of the graph $k\mapsto\N{u}_{H^{1}(B_{2})}$, for $n_{i}=\frac13$
      such spikes are conspicuously absent. The norms were computed by means of a Fourier
      spectral method with a number of modes large enough to render any discretization error
      negligible. MATLAB codes in \url{https://github.com/hiptmair/ScatteringQuasiResonances}.}}
  \label{fig:qr}
\end{figure}

From the relationship $\wh{k}(y) = (1+\frac12 y)\wn$ we conclude that the mapping
$y\in\cintv{-1,1}\mapsto \N{\wh{u}}_{H^{1}(D_{0})}$ may feature spikes, which will become
more numerous, steeper, and higher with increasing $\wn$. Therefore, although the function
$y\mapsto \wh{u}(y)$ is analytic, already for moderate $\wn$ its accurate
polynomial approximation may require very high degrees. 
Putting it bluntly, polynomial surrogate
modeling is doomed in the presence of quasi-resonances and at this point no viable
alternative is available, which forces us to avoid scattering problems beset with
quasi-resonances.

\begin{remark}
  \label{rem:ratfun}
  A promising approach to obtain efficient surrogate models for the map $\wh{k}\mapsto
  \wh{u}(k)$ may rely on \emph{rational functions}, see \cite{BNP20}. So far this is
  confined to a single parameter and extension to many parameters as required by general
  shape UQ methods is not clear yet. 
\end{remark}

\subsection{Scattering problems without quasi-resonances} 
\label{ss:nonres}

Whereas the Helmholtz transmission problem with $n_i>1$ can suffer from
quasi-resonances, the problem with $n_i<1$ does not.  Indeed, the condition $n_i<1$
rules out total internal reflection of rays hitting $\partial D_0$ from
$D_0$. Furthermore \cite{MS19} showed that when $n_i<1$ and $D$ is star-shaped
Lipschitz,
given $k_0, R>0$, for all $k\geq k_0$, $k_{0} > 0$ sufficiently large, 
the outgoing solution of 
\beq\label{eq:SCHelm}
{(k^{-2}\Delta +n ) u = f} \quad\text{ with } \quad f\in L^2_{\rm comp}(B_R)
\eeq
satisfies 
\begin{gather}
  \label{eq:ntcest}
  \N{u}_{H^{1}_{\wn}(B_{R})} \leq C\wn \NLtwo[B_{R}]{f} \;,
\end{gather}
where $\N{u}^2_{H^1_\wn (B_R)}:= \wn^{-2}\N{\nabla u}^2_{L^2(B_R)} + \N{u}^2_{L^2(B_R)}$
and $C$ depends only on $n_i$, $d$, and $R$; i.e.,
\begin{quote}
  \emph{the $k$-explicit bound \eqref{eq:ntcest} holds uniformly on the set of star-shaped scatterers.}
\end{quote}
This observation motivated the present paper, because star-shapedness is a simple
geometric property, whose persistence under random shape perturbation seems rather
natural.
%
The bound \eqref{eq:ntcest} implies that the poles of the meromorphic function
$\wn\in \bbC\mapsto u(k)$ are located below $\left\{z\in\bbC:\,\Im z = -\nu\right\}$ for
some $\nu>0$, i.e., there is a resonance-free strip beneath the real axis, 
which means that no quasi-resonances are present, see Figure~\ref{fig:qr} (b).  
This
resonance-free-strip result was proved for smooth $D$ with strictly-positive curvature in
\cite{CaPoVo:99}, and more-detailed information about the location of the resonances in
this case is given in \cite{Ga:18}. 

%


We now explore the implications of the results from \cite{MS19} in the simple
size-UQ setting of Section~\ref{ss:scuq}.  
If $n_{i}\leq 1$ and $D_{0}$ is
star-shaped with respect to $0$, then $y\in\bbC\mapsto \wh{u}(y)$ is meromorphic on
$\left\{z\in \bbC:\,\Re z > \rhc{-2} \right\}$. However, since $\wh{\wn}(y)=(1+\frac12 y)k$,
the distance of the poles of $y\mapsto \wh{u}(y)$ to the real axis shrinks like $O(\wn^{-1})$ as $\wn\to\infty$. 
This means that the extent of the domain of analyticity of the function
$y\in\bbC\mapsto \wh{u}(y)$ in the direction of the imaginary axis will decrease like
$O(\wn^{-1})$ for $\wn\to\infty$.  Therefore, even in non-resonant situations,
$\wn$-uniformly accurate polynomial surrogate models for $y\mapsto \wh{u}(y)$ will require
increasingly larger degree for growing $\wn$.  
This growth with respect to $\wn$ can
be avoided only if the size of admissible shape perturbations (controlled by the range of the
parameter $y$ in the size-UQ model) decreases as $O(\wn^{-1})$ for $\wn\to\infty$. 
In other words,
\begin{quote}
  \emph{shape perturbations must be confined to the scale of the wavelength (or smaller)
    to permit provably $\wn$-robust polynomial surrogate modelling.}
\end{quote}

\begin{remark}
  \label{rem:gks}
  We highlight that this condition of $O(\wn^{-1})$ perturbations is encountered in
  the UQ of the Helmholtz equation with variable coefficients, i.e.,
  \begin{equation*}
    k^{-2} \nabla_\Bx\cdot\big( A(\Bx,\By) \nabla_\Bx u (\Bx,\By) \big)  + n(\Bx,\By) u(\Bx,\By) = {f}(\Bx) 
  \end{equation*}
  with $u$ satisfying the Sommerfeld radiation condition and $A(\Bx,\By)$ and
  $n(\Bx,\By)$ perturbations of some $A_0(\Bx)$ and $n_0(\Bx)$, respectively.  
  When the
  problem with $A_0$ and $n_0$ is nontrapping, given $k_0>0$ there exists $C>0$ such
  that for all $k\geq k_0$ the map $\By\mapsto u(\By)$ is analytic for all $\By$ such
  that \cite[Theorem~1.1]{SpWu:22}
  \begin{equation}\label{eq:SpWu1}
    k \max \big\{ \|A-A_0\|_{L^\infty} , \|n-n_0\|_{L^\infty}\big\} \leq C \;,
  \end{equation}
  with this bound sharp through an unbounded sequence of {wavenumbers $k$} 
  \cite[Theorem 1.4]{SpWu:22}.  
  In \cite[Theorem~4.2]{GKS20}
\rev{and \cite[Theorem 3.1]{GKN25},} 
  the condition \eqref{eq:SpWu1} is also
  shown to be a sufficient condition for analyticity (in the form of the relevant bounds
  on derivatives of $u$ with respect to $\By$) \rev{for particular classes of nontrapping $A_0,n_0$.} 
  \rev{In
  \cite[Section~4]{KUS24} $O(k^{-1})$ bounds on shape perturbations are also identified as
  necessary for reliable Bayesian shape inversion.}
\end{remark}


\subsection{Layout of the paper}
\label{ss:outline}

The foundations for the current work were laid in \cite{HSS15}, which dealt with shape UQ
for the Helmholtz transmission problem \eqref{eq:htp}, \eqref{eq:n} in a
non-$\wn$-explicit way
under the assumption that the wave number was smaller than any possible quasi-resonance.
Here we reuse a tool of \cite{HSS15}, the parameterization of the shape of $D$ by a
\emph{radial displacement function} $r\in C^{1}(\bbS^{d-1})$ defined on the unit sphere
$\bbS^{d-1}$ in $\bbR^{d}$:
\begin{gather}
  \label{eq:Dr}
  D=D(r) := \left\{ \Bx=\rho \Bsr \in\bbR^{d}:\; 0\leq \rho < 1+ r(\Bsr),\,
    \Bsr \in\bbS^{d-1} \right\}\;.
\end{gather}
If $\N{r}_{C^{0}(\bbS^{d-1})}<1$, then any $D(r)$ is star-shaped with respect to $0$ and
if, in addition, $n_{i}\leq 1$, then the requirements of \cite{MS19} for non-resonant
scattering are satisfied immediately.

The shape parameterization \eqref{eq:Dr} is examined in Section~\ref{sec:param},
where we extend it to a mapping of $D(r)$ to the unit sphere $B_{1}$ to
convert \eqref{eq:htp} into a variational problem with $r$-dependent coefficients. This is
the gist of the popular \emph{domain mapping method} for shape UQ, first proposed in
\cite{XIT06} and then used by many other authors, e.g.\ in \cite{HPS15,CNT15a,CCS15}. 
Yet, the latter
works rely on volume diffeomorphisms for parameterization of shapes, which do not offer
a natural avenue to the preservation of star-shapedness.

In Section~\ref{sec:MS} we review the results of \cite{MS19}, apply them to the concrete
transmission problem \eqref{eq:htp}-\eqref{eq:n} and state more precisely the estimate
\eqref{eq:ntcest} and generalizations of it, providing \emph{completely $k$-explicit}
expressions for the constants. Then, in Section~\ref{sec:FreDepHol}, we admit
$\bbC$-valued radial displacement functions $r$ in \eqref{eq:Dr}; thanks to the domain
mapping approach the resulting transmission problem remains well defined and a
perturbation argument yields $\wn$-explicit stability estimates. These estimates give
detailed $\wn$-explicit information about \emph{shape holomorphy}, more precisely, about
the domain of analyticity of $r\mapsto \wh{u}(r)$, $\wh{u}(r)$ the solution of
\eqref{eq:htp}, with $D=D(r)$ given by \eqref{eq:Dr}, pulled back to the unit
sphere. Such ``shape holomorphy'' results have been established in recent years for a
host of parametric PDEs models
\cite{CNT15a,AJSZ20_2734,HS21_2779,CSZ18,JerezChSZech2017}. In these works, the
size of the holomorphy domain of parametric solutions and operators was not made explicit
in terms of problem parameters such as wavenumber, or (in \cite{CSZ18}) the Reynolds
number.

Next, in Section~\ref{sec:ExpBdsFEM} we investigate the impact of a high-order
finite-element discretization of the transmission problem. We give $\wn$-explicit
estimates of how the finite-element discretization error affects the estimates of
statistical moments.

In Section~\ref{sec:faf} we present the so-called far-field pattern as a representative of
relevant output functionals that depend on $\wh{u}(r)$ and inherit its holomorphy.

In Section~\ref{sec:ParHolIntQuad} we introduce a second-level affine parameterization of
$r$ through a sequence of uniformly-distributed random variables.  
Then, in Section~\ref{sec:AplShUQ}, we appeal to the theory of \cite{CCS15,Z18_2760,ZS20_2485}, 
and leverage the available information on shape
holomorphy to predict the rate of convergence of Smolyak-type, high-dimensional quadrature
when applied to extract the statistical mean of the far-field-pattern random field.

\subsection*{List of notations}

\newcommand{\rC}{\mathfrak{r}}
\newcommand{\rCS}{\mathfrak{R}}

\begin{tabular}[c]{r@{\quad$\hat{=}$\quad}p{0.7\linewidth}}
  $d$ & spatial dimension $\in \{2,3\}$ \\
  $\wn$ & positive wavenumber \\
  $r=r(\refs{\Bs})$ & radial displacement function $\in C^{1}(\mathbb{S}^{d-1})$, see
  \eqref{eq:Dr} \\
  $u=u(r;\refs{\Bx})$ & solution of scattering problem, see \eqref{eq:vf} \\
  $\Phibf = \Phibf(r;\wh{\Bx})$ & transformation from reference domain, see \eqref{eq:Phi}
  \\
  $\chi=\chi(\rho)$ & cut-off function defined in \eqref{chi} \\
  $\refs{u}=\refs{u}(r;\refs{\Bx})$ & solution of transformed scattering problem
  \eqref{eq:vft} \\
  $\refs{\VA} = \refs{\VA}(r;\refs{\Bx})$ & diffusion coefficient tensor of transformed
  variational problem, see \eqref{eq:At} \\
  $\refs{n} = \refs{n}(r;\refs{\Bx})$ & refractive index of transformed
  variational problem, see \eqref{eq:nt} \\
  $\N{\cdot}_{H^1_\wn (B_2)}$, $\N{\cdot}_{H^2_\wn (B_2)}$ & $\wn$-scaled Sobolev norms,
  see \eqref{eq:weighted_norms} \\
  $\rC = \rC(\refs{\Bs})$ & \emph{complex-valued} radial displacement function
  $\in C^{1}(\bbS^{d-1},\bbC)$ \\
  $\rCS$ & set of admissible \emph{complex-valued} radial displacement functions, see
  \eqref{eq:rCS} \\
  $\refs{a} = \refs{a}(\rC;\refs{u},\refs{v})$ & bilinear form of transformed variational
  problem, see \eqref{eq:ahatdef} \\
  $\Ca = \Ca(k;C_{\Re},C_{\Im})$ & domain of analyticity of $\rC\mapsto \wh{u}(\rC)$, see
  \eqref{eq:Anset}\\
  $\refs{u}_{\rm PML} = \refs{u}_{\rm PML} (\rC;\refs{\Bx}) $ & solution of PML-truncated
  transformed scattering problem \\
  $\Omega_{\mathrm{tr}}$ & PML-truncated computational domain \\ 
  $\wt{\sigma} = \wt{\sigma}(\refs(\rho))$ & PML control function, see
  \eqref{eq_sigma_prop} \\
  $p\in\mathbb{N}$ & polynomial degree/approximation order of finite element space, see
  Assumption~\ref{ass:spaces} \\
  $V_{h}$ & finite element space  $\subset H^{1}(\Omega_{\mathrm{tr}})$ \\
  $\refs{u}_{\rm PML}(\rC;\cdot)_h\in V_{h}$ & Galerkin FE solution with PML truncation \\
  $\cA_p  = \cA_{p}(k;C_{\Re},C_{\Im})$ & set of admissible complex-valued radial
  displacement function for FEM, see \eqref{eq:Hh_gamma} \\
  $\refs{u}^{\rm scat}_\infty = \refs{u}^{\rm scat}_\infty(\rC;\unitxh)  $ & far-field
  pattern, see \eqref{eq:faft} \\
  $\refs{u}^{\rm scat}_{\infty, \rm PML}(\rC;\unitxh)_h  $ & approximate far-field pattern
  based on FEM and PML, see \eqref{eq:faft2} \\
  $r_{j}=r_{j}(\Bx)$ & radial expansion functions, see \eqref{eq:rj} \\
  $\prms$ & set of shape parameter sequences, see \eqref{eq:prms} \\
  $\Ch = \Ch(C_{\Re},C_{\Im},k)$ & domain of analyticity of $\Vz \mapsto \breve{u}({\Vz}
  ;\cdot)$, see Corollary~\ref{cor:H}. 
\end{tabular}


\section{Variational Formulation of the Transmission Problem with Parametric Interface}
\label{sec:param}
The parameterization \eqref{eq:Dr}, apart from nicely fitting the assumptions of
the theory of \cite{MS19}, 
also paves the way for a detailed and explicit analysis, on which we now embark.
\subsection{A parametric family of diffeomorphisms}
\label{sec:pi}

Recall the parameterization of the shape of the scatterer $D$ by means of a continuously
differentiable radial displacement function $r\in C^{1}(\bbS^{d-1},\bbR)$ defined on the
$d$-dimensional unit sphere $\bbS^{d-1}\subset\bbR^{d}$
\begin{gather*}
  \tag{\ref{eq:Dr}}
  D=D(r) := \left\{ \Bx=\rho{\Bsr}\in\bbR^{d}:\; 0\leq \rho < 1+ r({\Bsr}),\,
    {\Bsr}\in\bbS^{d-1} \right\}\;.
\end{gather*}
It goes without saying that $r$ has to be bounded from below at least to render
\eqref{eq:Dr} meaningful. For our analysis we confine $r$ to a more narrow range of values:

\begin{assumption}[Bounds for radial displacement function]
  \label{ass:rbd}
  We admit only radial displacement functions belonging to the set 
  \begin{gather}
    \label{eq:R}
    \Cr := \left\{r\in C^{1}(\bbS^{d-1},\bbR),\,\N{r}_{C^{0}(\bbS^{d-1})}\leq \frac13 \right\},
  \end{gather}
  that is $|r({\Bsr})|\leq \frac13$ for all $\Bsr \in \bbS^{d-1}$. 
\end{assumption}

Note that if $r\in C^{m}(\bbS^{d-1},\bbR)$, then $D(r)$ is of class $C^{m}$.  
In particular, $D(r)$ is a bounded Lipschitz domain and the interface
$\Gamma(r):=\partial D(r)$ is contained in a spherical shell:
\begin{gather}
  \label{eq:gann}
  \text{Assumption~\ref{ass:rbd}}\quad\Rightarrow\quad
  {\frac23} \leq |\Bx| \leq
  {\frac43}\quad \forall \Bx \in \Gamma(r)\;.
\end{gather}
As in \cite[Section~3]{HSS15}, we shift parameter dependence from the domain $D(r)$ to the
coefficients of a transformed transmission problem, for which the scatterer occupies the
unit ball $\refs{D}:=\{\Bx\in\bbR^{d}:\,|\Bx|<1\}$. Taking the cue from \cite[Equation
(3.2)]{HSS15} or \cite[Section~3.2]{GHS24}, the transformation is effected by the
parameter-dependent mapping $\Phibf(r):\bbR^{2}\to\bbR^{2}$, $r\in \Cr$, given by
\fboxsep1ex
\begin{gather}
  \label{eq:Phi}
  \boxed{
    \Phibf(r;\refs{\Bx})\mid_{\refs{\Bx} = \rr\Bsr \in \bbR^{d}}
    = 
    \Bsr\bigl(\rr + \chi(\rr)r(\Bsr)\bigr)\;,\quad
    \Bsr\in\bbS^{d-1}\,,\; \rr\geq 0 \;,\quad \refs{\Bx}=\rr\Bsr \in \bbR^{d}}\;,
\end{gather}
\begin{minipage}[c]{0.5\textwidth}
  with a cut-off function\footnote{{We can choose
      $\chi(\rho) = f(g(\frac{\rho-1}{1-\lambda}))$ with
      $f(\xi) := \exp\left(1-\frac{1}{1-\xi^{2}}\right)$ for $|\xi|<1$, $f(\xi) := 0$ for
      $|\xi|\geq 1$, $g(\zeta) := \arctan(\frac32\zeta)/\arctan(\frac32)$, and
      $\lambda=\frac{1}{40}$. This function is displayed in the plot beside.}}
  $\chi \in C^{\infty}(\cointv{0,\infty})$, satisfying for some $0<\lambda\ll 1 $
  \begin{subequations}
    \label{chi}
    \begin{flalign}
      \label{chi:1}
      \text{(i)} && 0\leq \chi(\rho) & \leq 1\quad \forall\rho\geq 0, && \\
      \label{chi:2}
      \text{(ii)} && \chi(\rho)& =0\quad\text{for}\; \rho\leq \lambda\;\text{or}\; \rho \geq 2-\lambda,&& \\
      \label{chi:3}
      \text{(iii)} && |\chi'(\rho)|& \leq {2-\lambda} \quad \forall\rho\geq 0, && \\
      \label{chi:4}
      \text{(iv)} &&
      \chi(1)&=1. &&
    \end{flalign}
  \end{subequations}
\end{minipage}%
\begin{minipage}[c]{0.5\textwidth}
  \includegraphics[width=\textwidth,clip]{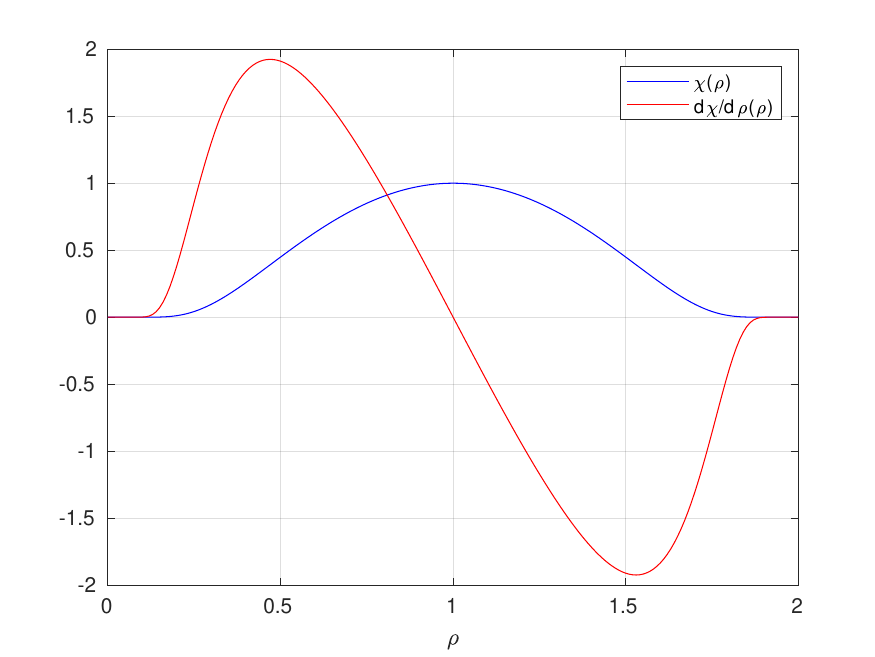}
\end{minipage}%

Property \eqref{chi:3} combined with \eqref{eq:gann} ensures that the radial dilation
effected by $\Phibf(r;\cdot)$, $\rho \mapsto \rho+ \chi(\rho)r(\refs{\Bsr})$, has a derivative
bounded below by $\frac13$ and, hence, is strictly monotone. We therefore have the
following lemma.

\begin{lemma}[Properties of mappings $\Phibf$]
  \label{lem:propphi}
  Under Assumption \ref{ass:rbd} 
  and \eqref{chi} the mappings $\Phibf(r;\cdot)$, $r\in\Cr$, 
  as defined in \eqref{eq:Phi} are diffeomorphisms of class $C^{m-1,1}$, 
  map the unit sphere $\partial \refs{D} = \bbS^{d-1}$ onto the interface $\Gamma(r)$, and 
  agree with the identity in the exterior of the spherical shell
  $\{\refs{\Bx}\in\bbR^{d}:\, \lambda < |\refs{\Bx}| < 2-\lambda\}$.
\end{lemma}

\subsection{\cs{Variational formulation under pullback via the domain mapping $\Phi$}}
\label{sec:vf}

Recall that for all $r\in\Cr$ the domain $D(r)$ is contained in the ball $B_{2}$;
therefore 
the transmission problem \eqref{eq:htp} can be written as a variational problem on
$B_2$ involving the Dirichlet-to-Neumann map on $\partial B_2$. 
%
Given $g\in H^{1/2}(\partial B_2)$, let $v$ be the solution of
\begin{equation}\label{eq:DtN1}
  (-\wn^{-2}\Delta - 1)v=0 \quad\text{ in } \mathbb{R}^d \setminus \overline{B_2} 
  \quad\text{ and }\quad  
  v=g \text{ on } \Gamma_2 := \partial B_{2}
\end{equation}
satisfying the Sommerfeld radiation condition
\beq\label{eq:src}
  |\Bx|^{\frac{d-1}{2}} \left( \wn^{-1}\frac{\partial}{\partial |\Bx|}-\iu \right)v(\Bx)
  \to 0 \quad \text{as}\;|\Bx|\to\infty, \text{uniformly in $\Bx/|\Bx|$}.
  \eeq
  Then, define the map $\dtn: H^{1/2}(\partial B_2)\to H^{-1/2}(\partial B_2)$ by
\begin{equation*}
  (\dtn g)(\Bx) := \wn^{-1} \nabla v(\Bx)\cdot \Bn(\Bx)\;,\quad \Bx\in \partial B_{2},
\end{equation*}
where on $\partial B_2$ we have $\Bn(\Bx) := R^{-1} \Bx$, i.e.,
$\Bn(\Bx)=\frac{\Bx}{|{\Bx}|}$ 
is the outward-pointing unit normal vector to $B_2$.

Through the $r$-dependent domain $D=D(r)$, the weak solution $u\in H^{1}(B_{2})$ of the
transmission problem obviously depends on the radial displacement function $r\in \Cr$, with $\Cr$ defined
in Assumption~\ref{ass:rbd}; we therefore write $u=u(r)$, regarding $u$ as a mapping
$\Cr\to H^{1}(B_{2})$, which turns out to be convenient in \S\ref{sec:ParHolIntQuad}.
Recall that, by Assumption~\ref{ass:rbd}, $D=D(r)$ is contained in the ball $B_{2}$;
therefore the variational formulation of \eqref{eq:htp} is:

\noindent\fbox{%
  \begin{minipage}{0.975\textwidth}
    Given $L\in (\Hone[B_{2}])^*$, 
       $r\in \Cr$, 
        and piecewise-constant refractive index $n=n(r,\cdot)$ defined by \eqref{eq:n} and \eqref{eq:Dr},  
        find $u{(r;\cdot)}\in\Hone[B_{2}]$
    such that, for all $v \in \Hone[B_{2}]$,
    \begin{equation}
      \label{eq:vf}
      \int_{B_2}\Big(\wn^{-2}\nablabf u({r};\Bx) \cdot\overline{\nablabf v}(\Bx)
      - n(r;\Bx)u({r};\Bx)\overline{v}(\Bx)\Big) \rd \Bx -
      \int\nolimits_{\partial B_{2}} \hspace{-1ex}\wn^{-1}\dtn(\rst{u({r;\cdot})}{\partial
        B_{2}})\overline{v}\,\mathrm{d}S  
      =  L(v).
    \end{equation}
  \end{minipage}}
\medskip

If 
\begin{equation*}
  L(v):= \wn^{-1}\int\nolimits_{\partial B_{2}} \bigl(
  \wn^{-1}  \nablabf \uinc\cdot
  \Bn - \dtn(\rst{\uinc}{\partial
    B_{2}})\bigr)\,\overline{v}\,\mathrm{d}S,
\end{equation*}
then the solution of \eqref{eq:vf} is the restriction to $B_2$ of the solution of \eqref{eq:htp}.

We repeat that the variational problem \eqref{eq:vf} depends on
$r\in \Cr$ via $D=D(r)$ and \eqref{eq:n}. This dependence becomes more
apparent in the transformed variational problem formed by pulling back \eqref{eq:vf} to
the nominal setting with interface $\partial \refs{D}=\bbS^{d-1}$ through the
diffeomorphism $\refs{\Bx}\mapsto\Phibf(r;\refs{\Bx})$. As derived in
\cite[Section~4.1]{HSS15}, this transformed variational problem reads:
\medskip

\noindent\fbox{%
  \begin{minipage}{0.975\textwidth}
Find $\refs{u}\in\Hone[B_{2}]$ such that
    \begin{multline}
      \label{eq:vft}
        \int\nolimits_{B_{2}}
        \Big( 
        \wn^{-2}\refs{\VA}(r;\refs{\Bx}) \refs{\nablabf}
        \refs{u}(r;\refs{\Bx}) \cdot \overline{\refs{\nablabf} \refs{v}}(\refs{\Bx}) -
        \refs{n}(r;\refs{\Bx})\refs{u}(r;
        \refs{\Bx})\overline{\refs{v}}(\refs{\Bx})\Big) \rd \refs{\Bx} - \\
        \int\nolimits_{\partial B_{2}}\wn^{-1}
        \dtn(\rst{\refs{u}(r;\cdot)}{\partial B_{2}})\overline{\refs{v}}\,\mathrm{d}S  =
        \refs{L}(\refs{v}) \quad \forall \refs{v}\in \Hone[B_{2}]\;,
    \end{multline}
  \end{minipage}}
\begin{subequations}
  \label{eq:coeffs}
  \begin{flalign}
    \label{eq:At}
    \text{where}&& \refs{\VA}(r;\refs{\Bx}) & :=
    \Dervxt\Phibf(r;\refs{\Bx})^{-1}\Dervxt\Phibf(r;\refs{\Bx})^{-\top}
    \det \Dervxt\Phibf(r;\refs{\Bx}) ,&& \\
    \label{eq:nt}
    && \refs{n}(r;\refs{\Bx}) & := \det \Dervxt\Phibf(r;\refs{\Bx})
    n_{0}(\refs{\Bx})\;,\quad n_{0}(\refs{\Bx}) :=
    \begin{cases}
      n_{i} & \text{for}\; |\refs{\Bx}| < 1 ,\\
      1 &  \text{for}\; |\refs{\Bx}| \geq 1 ,
    \end{cases}
    && \\
    && \refs{L}(\refs{v}) & := L(\Phibf(r)^{*}v)\;,\,\,\text{ where } \,\,
    (\Phibf(r)^{*}v)(\xrf) := v(\Phibf(r;\xrf)).
  \end{flalign}
\end{subequations}
We recall that, by \eqref{eq:Phi},
$\Phibf(r,\refs{\Bx}) := \refs{\Bs}(\refs{\rho}+\chi(\refs{\rho})r(\refs{\Bs}))$,
$\refs{\Bs}\in\bbS^{d-1}$, $\refs{\Bx}=\refs{\rho}\refs{\Bs}$.

To understand how the coefficients $\refs{\VA}$ and $\refs{n}$ depend on the function
$r\in \Cr$, we establish their precise expressions.  The following elementary lemma
  shows that the Jacobian $\Dervxt\Phibf(r,\refs{\Bx})$ of $\Phibf(r;\xrf)$ has triangular
  structure with only off-diagonal depending on derivatives of $r(\Bsr)$. 

\begin{lemma}[Jacobian of transformation mapping]\label{lem:Jacobian}
  Given $r\in \Cr$, in polar/spherical coordinates the Jacobian of the
  mapping $\xrf\mapsto \Phibf(r;\xrf)$ from \eqref{eq:Phi} at $\xrf=\rr\Bsr\in B_{2}$,
  $0\leq \rr \leq R$, $\Bsr\in\bbS^{d-1}$, is
  \begin{gather}
    \label{eq:DPhi}
    \Dervxt\Phibf(r;\xrf) =
    \begin{bmatrix}
      1+\chi'(\rr)r(\Bsr) & \dfrac{\chi(\rr)}{\rr} \dfrac{d r}{d \Bsr}(\Bsr) \\
      \mathbf{0} & \left( 1+\dfrac{\chi(\rr)}{\rr}r(\Bsr) \right)\,\VI_{d-1}
    \end{bmatrix} \in \bbR^{d\times d}.
  \end{gather}
\end{lemma}

\begin{proof}
  For fixed $r\in C^{1}(\bbS^{d-1},\bbR^{+})$ the mapping
  $\xrf\mapsto \Phibf(r;\xrf) := \Bsr\bigl(\rr + \chi(\rr)r(\Bsr)\bigr)$,
  $\xrf = \rr\Bsr\in B_{2}\subset\bbR^{d}$, from \eqref{eq:Phi} is of the form
  \begin{gather}
    \label{eq:ganmap}
    B_{2} \to \bbR^{d}\quad,\quad \Bx \mapsto \Psibf(\Bx) := 
    \psi\bigl(|\Bx|,\frac{\Bx}{|\Bx|}\bigr)\,\frac{\Bx}{|\Bx|} ,
  \end{gather}
  with a continuously differentiable function $\psi:(0,2]\times\bbS^{d-1}\to\bbR$
  with $\psi(r,\Bsr) = r$ for small $r$. Hence it suffices to compute the derivative
  of $\Psibf$ for $\Bx\in B_{2}\setminus\{0\}$.

Let $\Bh\in\bbR^{d}$ be a sufficiently-small perturbation vector. By Taylor
  expansion, as $\Bh\to 0$,
  \begin{align}
    \label{eq:nas}
    |\Bx+\Bh| & = |\Bx| + \frac{\Bx^{\top}\Bh}{|\Bx|} + O(|\Bh|^{2}) ,\\
    \label{eq:inas}
    \frac{\Bx+\Bh}{|\Bx+\Bh|} & = \frac{\Bx}{|\Bx|} + \left( \frac{\VI}{|\Bx|} -
      \frac{\Bx\Bx^{\top}}{|\Bx|^{3}} \right)\Bh + O(|\Bh|^{2}) .
  \end{align}%
  We write $\partial_{1}\psi$ and $\partial_{2}\psi$ for the partial derivatives of $\psi$; 
  note that $\partial_{1}\psi$ is a scalar, and $\partial_{2}\psi$ is a $1\times d$ matrix (i.e., a row vector). 
  Using \eqref{eq:nas} and \eqref{eq:inas} followed by Taylor expansion gives us
  \begin{align*}
    \Psibf(\Bx+\Bh) & =
    \begin{aligned}[t] & 
      \psi\left( |\Bx| + \frac{\Bx^{\top}\Bh}{|\Bx|} + O(|\Bh|^{2}), \frac{\Bx}{|\Bx|} +
        \left( \frac{{\VI}}{|\Bx|} - \frac{\Bx\Bx^{\top}}{|\Bx|^{3}} \right){\Bh} + O(|\Bh|^{2})
      \right)\cdot \\ & \quad \left(\frac{\Bx}{|\Bx|} + \left( \frac{{\VI}}{|\Bx|} -
          \frac{\Bx\Bx^{\top}}{|\Bx|^{3}} \right){\Bh} + O(|\Bh|^{2}) \right)
    \end{aligned} \\
    & =  \begin{aligned}[t] &
      \left(
        \psi\Big(|\Bx|,\frac{\Bx}{|\Bx|}\Big) +
        \partial_{1}\psi\Big(|\Bx|,\frac{\Bx}{|\Bx|}\Big)\frac{\Bx^{\top}\Bh}{|\Bx|} +
        \partial_{2}\psi\Big(|\Bx|,\frac{\Bx}{|\Bx|}\Big)\left(
          \frac{{\VI}}{|\Bx|}-\frac{\Bx\Bx^{\top}}{|\Bx|^{3}} \right)\Bh + O(|\Bh|^{2})
      \right)\cdot \\
      &\quad  \left(\frac{\Bx}{|\Bx|} + \left( \frac{{\VI}}{|\Bx|} -
          \frac{\Bx\Bx^{\top}}{|\Bx|^{3}} \right){\Bh} + O(|\Bh|^{2}) \right)
    \end{aligned} \\
    & = \Psibf(\Bx) + 
    \begin{aligned}[t]
      & \frac{\Bx}{|\Bx|}\,\partial_{1}\psi\Big(|\Bx|,\frac{\Bx}{|\Bx|}\Big)\frac{\Bx^{\top}\Bh}{|\Bx|}
      + \frac{\Bx}{|\Bx|}\, \partial_{2}\psi\Big(|\Bx|,\frac{\Bx}{|\Bx|}\Big)
      \frac{1}{|\Bx|}\left( \VI - \frac{\Bx\Bx^{\top}}{|\Bx|^{\esc{2}}} \right)\Bh  \\ & 
+ \frac{1}{|\Bx|}\psi\Big(|\Bx|,\frac{\Bx}{|\Bx|}\Big) \left(
        \VI-\frac{\Bx\Bx^{\top}}{|\Bx|^{2}}\right)\Bh + O(|\Bh|^{2}).
    \end{aligned}
  \end{align*}%
Let $\Op{P}_{r}(\Bx)$ and $\Op{P}_{\varphi}(\Bx)$ denote the orthogonal projections
  onto $\left<\Bx\right>$ and $\left<\Bx\right>^{\perp}$, i.e., 
  \begin{gather}
    \label{eq:op}
    \Op{P}_{r}(\Bx)\Bh := \frac{\Bx\Bx^{\top}}{|\Bx|^{2}}\Bh \quad\text{and}\quad
    \Op{P}_{\varphi}(\Bx)\Bh := \left( \VI - \frac{\Bx\Bx^{\top}}{|\Bx|^{2}} \right)\Bh.
  \end{gather}
    Abbreviating $\Bsr := {\Bx}/{|\Bx|}$ (i.e., the radial unit vector) and $\rho:=|\Bx|$, we therefore have 
  \begin{gather*}
    \Derv\Psibf(\Bx)\Bh =
    \Bsr\,\partial_{1}\psi(\rho,\Bsr)\Bsr^{\top}\Op{P}_{r}(\Bx)\Bh +
    \frac{\Bsr}{\rho}\, \partial_{2}\psi(\rho,\Bsr)
    \Op{P}_{\varphi}(\Bx)\Bh
    + \frac{1}{\rho}\,\psi(\rho,\Bsr) \Op{P}_{\varphi}(\Bx)\Bh .
  \end{gather*}
The result then follows since the projections  $\Op{P}_{r}(\Bx)$ and $\Op{P}_{\varphi}(\Bx)$ split the
  perturbation vector $\Bh$ into its radial and angular components and 
 $\psi(\rho,\Bsr) =
  \rho+\chi(\rho)r(\Bsr)$.
\end{proof}

Lemma \ref{lem:Jacobian} immediately implies that for $\xrf=\rr\Bsr\in B_{2}$
\begin{subequations}
  \label{eq:Dform}
  \begin{align}
    \label{eq:DPd}
    \det\Dervxt\Phibf(r;\xrf) & = \left( 1+\chi'(\rr)r(\Bsr)  \right)
    \left( 1+\tfrac{\chi(\rr)}{\rr}r(\Bsr) \right)^{d-1},\\
    \label{eq:DPi}
    \Dervxt\Phibf(r;\xrf)^{-1} & =
    \frac{1}{\det\Dervxt\Phibf(r;\xrf)}
    \begin{bmatrix}
      1+\dfrac{\chi(\rr)}{\rr}r(\Bsr) & - \dfrac{\chi(\rr)}{\rr} \dfrac{d r}{d
        \Bsr}(\Bsr) \\
      \mathbf{0} & \left( 1+\chi'(\rr)r(\Bsr) \right)\,\VI_{d-1}
    \end{bmatrix}\;.
  \end{align}
\end{subequations}
By \eqref{chi:3}, the bounds for $r(\Bsr)$, $-1/3 \leq r(\Bsr) \leq 1/3$ for 
$\Bsr\in\bbS^{d-1}$ (by Assumption~\ref{ass:rbd}), imply  
\begin{gather*}
  \dfrac53 \geq |1+\chi'(\rr)r(\Bsr)| \geq \dfrac{1}{3} \quad\text{and}\quad
  \dfrac53 \geq \Big|1+\frac{\chi(\rr)}{\rr}r(\Bsr)\Big| \geq \dfrac13 
\quad  \text{ for all }\Bsr \in\mathbb{S}^{d-1}, 0\leq \rr \leq 2;
\end{gather*}
which, in turn, imply
\begin{gather}
  \label{eq:bdd}
  (5/3)^{d} \geq |\det\Dervxt\Phibf(r,\xrf)| \geq 3^{-d}
  \quad
\text{ for all } r\in \Cr \text{ and } \xrf \in B_2,
\end{gather}
and, by estimating the Euclidean matrix norm via the Frobenius norm, 
for all $\refs{\Bx} \in B_{2}$, $r\in\mathcal{R}$
\begin{gather}
  \label{eq:Mest}
  \N{\Dervxt\Phibf(r,\xrf)}_{2}^{2} \leq \tfrac53 d + 2
  \N{\frac{dr}{ds}(\Bsr)}_{2}^{2}\;,\quad
  \N{\Dervxt\Phibf(r,\xrf)^{-1}}_{2}^{2} \leq {3^{d}}\left( \tfrac53 d + 2
    \N{\frac{dr}{ds}(\Bsr)}_{2}^{2} \right).
\end{gather}
The following result -- equivalence of norms of functions under the
transformation $\xrf\mapsto\Phibf(r;\xrf)$ -- is a consequence of
\eqref{eq:bdd} and \eqref{eq:Mest}. 
For a proof, see \cite[Lemma~3.4]{HSS15} and \cite[Appendix~E]{SCA16}.

\begin{lemma}[Transformation of norms]
  \label{lem:normtrf}
  For $r\in\Cr$ define the pullback $\Phibf(r)^{*}v$ as
  $(\Phibf(r)^{*}v)(\xrf) := v(\Phibf(r;\xrf))$ for a function $v:B_{2}\to \bbR$.
  \begin{enumerate}[label={(\roman{*})}]
  \item If ${v \in \Ltwo[B_{2}]}$, then ${ \Phibf(r)^{*}v \in \Ltwo[B_{2}]}$ and
    \begin{gather}
      \label{eq:l2trf}
      \left( \tfrac35 \right)^{\frac{d}{2}}\NLtwo[B_{2}]{v} \leq \NLtwo[B_{2}]{\Phibf(r)^{*}v} 
      \leq 3^{d/2} \NLtwo[B_{2}]{v}
      \;.
    \end{gather}
  \item If $v \in \Hone[B_{2}]$,  then $ \Phibf(r)^{*}v \in \Hone[B_{2}]$ and 
    \begin{align}
\nonumber
      \left( \tfrac35 \right)^{\frac{d}{2}} \frac{1}{3^{d}\left( \tfrac53 d + 2
          \N{\tfrac{dr}{ds}}_{C^{0}(\bbS^{d-1})}^{2} \right)}\NHone[B_{2}]{v} &\leq
      \NHone[B_{2}]{\Phibf(r)^{*}v} \\ &\leq  3^{d/2} 
      \left( \tfrac53 d + 2 \N{\tfrac{dr}{ds}}_{C^{0}(\bbS^{d-1})}^{2} \right) 
      \NHone[B_{2}]{v}\;.
              \label{eq:h1trf}
    \end{align}
  \item If $r\in C^{2}(\bbS^{d-1})\cap \Cr$ and $v \in \Hm[B_{2}]{2}$, then 
    $\Phibf(r)^{*}v \in \Hm[B_{2}]{2}$ and there exists $C_{2}>0$ depending on $d$ and
    $\N{r}_{C^{2}(\bbS^{d-1})}$ such that
    \begin{gather}
      \label{eq:h2trf}
      C_{2}^{-1}\NHm[B_{2}]{v}{2} \leq \NHm[B_{2}]{\Phibf(r)^{*}v}{2} \leq
      C_{2}\NHm[B_{2}]{v}{2}\; .
    \end{gather}
  \end{enumerate}
\end{lemma}

\section{$\wn$-Explicit Norm Bounds for the Solution of the Transmission Problem}
\label{sec:MS}

Recall from Section~\ref{ss:nonres} that we are particularly interested in understanding
how the domain of analyticity of $r\mapsto \wh{u}(r)$, with $\wh{u}(r)$ the solution of
\eqref{eq:vft}, behaves as the wavenumber $\wn$ becomes large. The key tools are the
results of \cite{MS19} concerning the stability of the transmission problem in
variational formulation \eqref{eq:vf}.  
These results are stated in the following $\wn$-scaled norms 
\begin{gather}
  \label{eq:weighted_norms}
  \N{u}^2_{H^1_\wn (B_2)}:= \wn^{-2}\N{\nabla u}^2_{L^2(B_2)} + \N{u}^2_{L^2(B_2)}\;,
  \quad 
  \N{L}_{(H^1_\wn (B_2))^*}:= \sup_{v\in H^1(B_2)\setminus\{0\}} \frac{|L(v)|}{\N{v}_{H^1_\wn (B_2)}}\;,
\end{gather}
and we record for later the definition that
\begin{gather}
  \label{eq:weighted_H2}
  \N{u}^2_{H^2_\wn (B_2)}:= \wn^{-4}|u|_{H^2(B_2)}^2+ \wn^{-2}\N{\nabla u}^2_{L^2(B_2)} + \N{u}^2_{L^2(B_2)}\;.
\end{gather}

As discussed in \S\ref{ss:scuq}-\ref{ss:nonres}, the behaviour of the norm of the solution
operator of the transmission problem with respect to $\wn$ depends on whether
$n_i<1$ or $n_i>1$; if $n_i<1$ the norm of the solution operator from $L^2(B_2)\to L^2(B_2)$ grows
like $\wn$ (see \eqref{eq:MS1} below); if $n_i>1$ the solution operator can grow
exponentially through $\wn_j$ for a discrete sequence $0<\wn_1<\wn_2<\ldots \infty$
\cite{PoVo:99}, \cite{St:00}, \cite[\S6]{MS19}, although the growth for ``most" $\wn$ is
at most algebraic in $\wn$ \cite{LaSpWu:21}.  We now recall bounds from \cite[Theorem
3.1]{MS19} on the solution when $n_i<1$.
%

\begin{theorem}[$\wn$-explicit bound on the solution of the 
    transmission problem from \cite{MS19}]
  \label{thm:MS}
  Let $D$ be a star-shaped Lipschitz domain, and let $0<n_{i}<1$. 
  
  (i) If $L(v):= \int_{B_2} f \, \overline{v} \, \rd S$ with $f \in L^2(B_2)$, 
  then the solution $u$ of the variational problem \eqref{eq:vf} satisfies
  \begin{equation}\label{eq:MS1}
    \N{u}_{H^1_\wn (B_2)} \leq {\Csolo(\wn, n_i)} \N{f}_{L^2(B_2)}\;,
  \end{equation}
  where
  \begin{equation}\label{eq:MS2}
    \Csolo(\wn, n_i) := 
    \frac{4\wn}{\sqrt{n_i}} \sqrt{
      1 + \frac{1}{n_i}\left(1+\frac{d-1}{4\wn}\right)^2
    }.    
  \end{equation}

  (ii) Given $L \in (H^1(B_2))^*$, the solution $u$ of the variational problem \eqref{eq:vf} satisfies
  \begin{equation}\label{eq:MS4}
    \N{u}_{H^1_\wn (B_2)} \leq \Csolt(\wn, n_i) \N{L}_{(H^1_\wn (B_2))^*},
  \end{equation}
  where 
  \begin{equation}\label{eq:MS5}
    \Csolt(\wn,n_{i}) := \frac{1}{n_i} \Big( 1 +2 \Csolo(\wn, n_i)\Big).
  \end{equation}
\end{theorem}
The key points of this result for the present study are that 
(i) as $\wn \to\infty$, $\Csolo$ and $\Csolt$ both grow like $\wn$, 
and (ii) the ($\wn$-explicit) constants $\Csolo$ and $\Csolt$ are independent of $D$; i.e., 
the bounds \eqref{eq:MS1} and \eqref{eq:MS4} hold uniformly across all 
star-shaped Lipschitz domains.

\bpf[References for the proof]
The bound \eqref{eq:MS1} is proved in \cite[Theorem 3.1]{MS19}; 
indeed, \eqref{eq:MS1} follows from \cite[Equation 3.2]{MS19} by {(i)} choosing $a_i=a_o=A_D=A_N=n_o=1$ 
and 
(ii)
by bounding $\wn^{-2}\|\nabla u\|^2_{L^2(B_2)} + n_i \|u\|_{L^2(B_2)}^2$ 
below by 
$n_i(\wn^{-2}\|\nabla u\|^2_{L^2(B_2)} + \|u\|_{L^2(B_2)}^2)$, 
and (iii) noting that, in the proof of \cite[Theorem 3.1]{MS19}, 
${\rm diam}(\Omega_i)$ can be replaced by $2$ 
(when going from \cite[Equation 5.3]{MS19} to \cite[Equation 5.4]{MS19}).
Having proved the bound \eqref{eq:MS1}, the bound \eqref{eq:MS4} follows by, e.g., \cite[Theorem 5.1]{GPS19}.
\epf

\begin{remark}[The Helmholtz exterior Dirichlet problem]
  The analogous $\wn$-explicit bounds on the solution of the Helmholtz exterior Dirichlet
  problem that hold uniformly for all Lipschitz star-shaped $D$ were proved in
  \cite[Corollary 3.9]{CM08} (see also \cite[Theorem 2.5]{GPS19}). The appropriate
  analogues of the results in the next section about $\wn$-explicit holomorphic dependence
  of the solution on the radial displacement function could therefore also be obtained for the
  Helmholtz exterior Dirichlet problem using these bounds in place of Theorem
  \ref{thm:MS}.
\end{remark}

\section{$\wn$-Explicit Holomorphic Dependence of the Solution on the Radial Displacement
  Function} 
\label{sec:FreDepHol}

For an approximation error analysis of polynomial surrogate models we invoke techniques
from complex analysis.  This requires extending the mapping $r\in\Cr\mapsto \wh{u}(r)$,
defined as the solution of the variational problem \eqref{eq:vft}, to complex-valued
radial displacement functions $r$, for which we adopt the notation $\rC$. 
The definition
of the set $\Cr$ from Assumption~\ref{ass:rbd} straightforwardly extends to $\bbC$-valued
$r$. 
Then we write
\begin{gather}
  \label{eq:rCS}
  \rCS:=\Big\{\rC\in C^{1}(\bbS^{d-1},\bbC):\,\N{\Re \rC}_{C^{0}(\bbS^{d-1})}\leq \frac13\Big\}\;.
\end{gather}
As explained in Section~\ref{ss:uqpoly} our key task is to establish the holomorphy of
$\rC\mapsto \hat{u}(\rC)\in H^{1}(B_{R})$ and identify its domain of analyticity and its
dependence on $\wn$.
From now on we assume that $\wn \geq 1/2$ to avoid degeneracy of 
the constants in the bounds of Theorem \ref{thm:MS}.

For $\rC\in\rCS$ let $\refs{u}(\rC)$ be the solution of the following variational
problem:~given $\refs{L}\in (H^1(B_2))^*$, 
\begin{equation} \label{eq:uhatzdef}
\mbox{find} \; \refs{u}(\rC) \in H^1(B_2) \; \mbox{such that} 
\quad 
\refs{a}(\rC;\refs{u},\refs{v}) = \widehat{L}(\refs{v}) 
\quad 
\mbox{for all}
\;\;
\refs{v}\in \Hone[B_{2}]
\;.
\end{equation}
Here, for $\rC \in \rCS$, $\refs{a}(\rC;\cdot , \cdot )$ denotes the sesquilinear form
from \eqref{eq:vft}, \eqref{eq:coeffs}, i.e.
\begin{gather}
  \label{eq:ahatdef}
  \begin{aligned}
    \refs{a}(\rC;\refs{u},\refs{v}) := & \itg[B_{2}]{\wn^{-2}\refs{\VA}(\rC;\refs{\Bx})
      \refs{\nablabf} \refs{u}(\rC;\refs{\Bx}) \cdot \overline{\refs{\nablabf}
        \refs{v}}(\refs{\Bx}) -\refs{n}(\rC;\refs{\Bx})
      \refs{u}(\rC;\refs{\Bx})\overline{\refs{v}}(\refs{\Bx})}[\mathrm{d}\refs{\Bx}] -
    \int\limits_{\partial B_{2}}\wn^{-1} \dtn(\rst{\refs{u}(\rC;\cdot)}{\partial
      B_{2}})\overline{\refs{v}}\,\mathrm{d}S,
  \end{aligned}
\end{gather}
with (now complex-valued) coefficients $\refs{\VA}(\rC,\cdot)$ and $\refs{n}(\rC,\cdot)$,
still given by the same formulae \eqref{eq:coeffs}.
The main result of this section is the following. 
\begin{theorem}[$\wn$-explicit condition for existence of $\refs{u}(\rC;\cdot)$ and related bounds]
  \label{thm:existence}\mbox{ } \newline
  (i) In the setting of \S\ref{sec:param}, 
      fix $n_i<1$, $d$, and, in addition, a constant $C_{\Re}\geq 2$.

  Then there exist $C_{\Im},C_1 >0, k_0 \geq \frac12$ 
  such that 
  for all $k\geq k_0$, 
  for any $\refs{L}\in (H^1_\wn (B_2))^*$, and 
  for any
  \begin{equation}
    \label{eq:holo_thm1}
    \rC \in \rCS \quad\text{with}\quad \N{\Re \rC}_{C^{1}(\bbS^{d-1})}\leq C_{\Re} \quad\text{and}\quad
    \wn  \N{\Im \rC}_{C^{1}(\bbS^{d-1})} \leq C_{\Im}, 
  \end{equation}
  the solution $\refs{u}(\rC;\cdot)$ of \eqref{eq:uhatzdef} 
  exists, is unique, and
  satisfies the bound
  \begin{equation}
    \label{eq:holo_thm2}
    \N{\refs{u}(\rC;\cdot)}_{H^1_\wn (B_2)} \leq C_1 \wn  \big\|\refs{L}\big\|_{(H^1_\wn
      (B_2))^*}\;.
  \end{equation}
  
  (ii) If, in addition,
  \begin{gather*}
    \rC\in C^2(\mathbb{S}^{d-1}) \quad\text{and}\quad \|\rC\|_{C^2(\mathbb{S}^{d-1})}\leq
    C_{S} \quad\text{for some}\quad C_{S}>0,
  \end{gather*}
  then there exists $C_2=C_{2}(C_{S})>0$ 
  such that, if
  $\refs{L}(\refs{v}):= \int_{B_2}f\,\overline{\refs{v}}$ 
  for $f\in L^2(B_2)$ and $k\geq k_0$, then
  \begin{equation}
    \label{eq:holo_thm2_H2}
    \N{\refs{u}(\rC;\cdot)}_{H^2_\wn (B_2)} \leq C_2 \wn  \big\|f\big\|_{L^2(B_2)}.
  \end{equation} 
\end{theorem}

The $\rC$-dependent coefficients $\wh{\VA}(\rC;\cdot)$ and $\wh{n}(\rC;\cdot)$ 
in the
sesquilinear form on the left-hand side of \eqref{eq:vft}/\eqref{eq:uhatzdef} 
are holomorphic on $\rCS$.  
Let $F(\rC, \refs{u}(\rC))$ be defined as the left-hand side
  of \eqref{eq:uhatzdef} minus the right-hand side, so that $F(\rC, \refs{u}(\rC))=0$.
The analytic implicit function theorem, e.g. \cite[Theorem~15.3]{DEI85},
applied to $F$ and Part (i) of Theorem~\ref{thm:existence} then implies the following
result.

\begin{corollary}[Domain of analyticity of $\rC\mapsto \wh{u}(\rC)$]
  \label{cor:A} {Fix $d$, $n_i<1$, and a constant $C_{\Re}\geq 2$. 
   Then there exist
    $C_{\Im},k_0 \geq\frac12$ 
   such that for all $k\geq k_0$, given $\refs{L}\in (H^1_k(B_2))^*$,}
  the mapping
  \begin{gather*}
    \rC\in C^{1}(\bbS^{d-1},\bbC) \mapsto \wh{u}(\rC)\in H^{1}(B_{2})\quad \text{defined
      through \eqref{eq:vft}/\eqref{eq:uhatzdef}}
  \end{gather*}
  is analytic on the set
  \begin{gather}
    \label{eq:Anset}
    \Ca = \Ca(k;C_{\Re},C_{\Im}) := \left\{\rC\in \rCS:\;
      \N{\Re \rC}_{C^{1}(\bbS^{d-1})}\leq {C_{\Re}} ,\;
      \wn \N{\Im \rC}_{C^{1}(\bbS^{d-1})} \leq C_{\Im}
    \right\}.
  \end{gather}
\end{corollary}

\begin{remark}\textbf{(Analyticity of $\rC\mapsto \wh{u}(\rC) \in H^2$)}\label{rem:H2}
By Part (ii) of Theorem \ref{thm:existence}, an analogous result holds for the mapping $\rC\in C^{2}(\bbS^{d-1},\bbC)\mapsto \wh{u}(\rC)\in H^{2}(B_{2})$ when $\refs{L}(\refs{v}):= \int_{B_2}f\,\overline{\refs{v}}$ for $f\in L^2(B_2)$.
Although we do not use this result in the rest of the paper, 
{additional smoothness of the target space of the 
analytic map $\rC\mapsto \wh{u}(\rC)$,
translates into higher Sobolev regularity of the (analytic continuation of) the 
parametric solution. 
The higher Sobolev regularity of the analytic continuation of the 
parametric solution (as opposed to mere boundedness of this continuation in the energy
norm of the PDE) is key in the convergence analysis of multilevel QMC methods; 
see, e.g., \cite[\S4.3.1]{DiGaLeSc:17} and the further discussion in \S\ref{sec:Concl}.
}
\end{remark}

To prove Theorem \ref{thm:existence},  
we first record some results about the sesquilinear form in \eqref{eq:uhatzdef} 
for (complex valued) $\rC \in \rCS$, as opposed to (real valued) $r\in \mathcal{R}$.
\begin{lemma}\label{lem:strong_elliptic}
Fix $d\in\{2,3\}$ and $C_{\Re}\geq 2$. 
Then there exist $c_1,c_{2}>0$ such that, if 
\begin{equation*}
    \rC \in \rCS \quad\text{with}\quad \N{\Re \rC}_{C^{1}(\bbS^{d-1})}\leq C_{\Re} 
    \quad\text{and}\quad
    \N{\Im \rC}_{C^{1}(\bbS^{d-1})} \leq c_1, 
\end{equation*}
then
\begin{equation}\label{eq:strong_elliptic}
\Re \big( \refs{\VA}(\rC; \refs{\Bx}) \xi, \xi\big)_2 
    \geq c_2 \|\xi\|^2_2 
\quad\text{ for all } \xi\in \Rea^d, \; \refs{\Bx} \in B_2 .
\end{equation}
\end{lemma}

By, e.g., \cite[Page 122]{Mc:00}, the condition \eqref{eq:strong_elliptic} 
ensures that the differential operator underlying 
the variational formulation \eqref{eq:uhatzdef} is strongly elliptic.
We note that \eqref{eq:strong_elliptic} is analogous to the 
condition \cite[Equation 5.16]{HSS15} in \cite[Assumption 5.11]{HSS15}.

 \bpf[Proof of Lemma \ref{lem:strong_elliptic}] 
  By the definition of
  $ \refs{\VA}(\rC; \refs{\Bx})$ \eqref{eq:At},
  \begin{equation*}
    \big( \refs{\VA}(\rC; \refs{\Bx}) \xi, \xi\big)_2 = 
    (\det\Dervxt\Phibf(\rC;\xrf))
    \|\Dervxt\Phibf( \rC;\refs{\Bx})^{-\top}\xi\|_2^2;
  \end{equation*}
  thus it is sufficient to bound 
  $\|\Dervxt\Phibf( \rC;\refs{\Bx})^\top
  \|_{2}$ from above and 
  $\Re\big( \det\Dervxt\Phibf(\rC;\xrf)\big)$ from below.
  The bound on $\|\Dervxt\Phibf( \rC;\refs{\Bx})^\top
  \|_{2}$ from above follows by replacing $r$ by $\rC$ in the expression \eqref{eq:DPhi} for $\Dervxt\Phibf( \rC;\refs{\Bx})$ and then 
  recognising that the bound   
  \eqref{eq:Mest} on $\|\Dervxt\Phibf( \rC;\refs{\Bx})
  \|_{2}$ holds also for $\|\Dervxt\Phibf( \rC;\refs{\Bx})^\top
  \|_{2}$.
    Finally,  by the expression \eqref{eq:DPd} (with $r$ replaced by $\rC$) and the bound \eqref{eq:bdd}, 
  $\Re( \det\Dervxt\Phibf(\rC;\xrf))$ is continuous with respect to $\Im \rC$ and bounded below by $3^{-d}$ when $\Im \rC=0$. Therefore there exist $c_1', c_2'>0$ such that 
  \beqs
\Re( \det\Dervxt\Phibf(\rC;\xrf))\geq c_2'\quad \text{ when }\N{\Im \rC}_{C^{1}(\bbS^{d-1})} \leq c_1';
\eeqs
the result then follows.
\epf

  We also need upper bounds for the transformed coefficient functions, which are
  supplied by the next lemma.
  
  \begin{lemma}
    \begin{enumerate}
    \item[(i)] There exists ${C>0}$ depending only on ${d \in \{2,3\}}$ such that
    \begin{gather}
      \label{eq:continuity_cond1}
      \sup_{\refs{\Bx}\in B_{2}}\big\| \refs{\VA}(\rC;\refs{\Bx})\big\|_2 \leq C \left(1+
        \| \rC\|_{C^1(\mathbb{S}^{d-1})}\right)^{2}\quad
      \text{and}\quad
      \sup_{\refs{\Bx}\in B_{2}}\big\| \refs{n}(\rC;\refs{\Bx})\big\|_2 \leq \left(1+2\left\|\rC \right\|_{C^{0}(\mathbb{S}^{d-1})}\right)^{d}. 
    \end{gather}
    \item[(ii)] Fix $d\in\{2,3\}$ and $c_3'>0$ 
      and assume that
      \begin{equation}\label{eq:continuity_cond2}
        \rC \in \rCS \quad\text{ with } \quad \| \rC\|_{C^2(\mathbb{S}^{d-1})}\leq c_3'\;.
      \end{equation}
      Then there exists $c_4' > 0$ such that
      \begin{equation*}
        \sup_{\refs{\Bx}}\big\| \refs{\nablabf}
        \refs{\VA}(\rC;\refs{\Bx})\big\|_2\leq c_4' \;\mbox{ for all }\; \xrf\in B_2\;.
      \end{equation*}
    \end{enumerate}
    \label{lem:continuity}
  \end{lemma}

Lemma \ref{lem:strong_elliptic} and Part (i) of Lemma \ref{lem:continuity} have the following immediate corollary.

\begin{corollary}[Continuity and G\aa rding inequality for the sesquilinear form]
  \label{cor:sesqui}
  Let $\refs{a}(\rC;\cdot,\cdot)$ denote the sesquilinear form on the left-hand
  side of \eqref{eq:uhatzdef}. 
  We fix $d\in\{2,3\}$ and $C_{\Re}>2$. 
  Then there exists $c_{1}>0$ such that if
\begin{equation*}
    \rC \in \rCS \quad\text{with}\quad \N{\Re \rC}_{C^{1}(\bbS^{d-1})}\leq C_{\Re} \quad\text{and}\quad
      \N{\Im \rC}_{C^{1}(\bbS^{d-1})} \leq c_1, 
\end{equation*}
then there exist ($k$-independent) constants $c_{2},c_{5},c_{6}>0$ with 
\begin{equation*}
|\refs{a}(\rC;\refs{u},\refs{v})| 
\leq 
c_5 \big\| \refs{u}\big\|_{H^1_k(B_2)} \big\| \refs{v}\big\|_{H^1_k(B_2)}
\quad\forall \refs{u}, \refs{v}\in H^1(B_2)
\end{equation*}
and
\begin{equation*}
\Re \big(\refs{a}(\rC;\refs{v},\refs{v})\big) 
\geq 
c_2
\big\|\refs{v}\big\|^2_{H^1_k(B_2)} - c_6 \big\|\refs{v}\big\|^2_{L^2(B_2)}
\quad\forall \refs{v}\in H^1(B_2),
\end{equation*}
and for all $k\geq \frac12$.
\end{corollary}

\bpf[Proof of Lemma \ref{lem:continuity}]
By the explicit formula \eqref{eq:DPi}, 
\begin{gather}\label{eq:DPhi_inverse}
  \Dervxt\Phibf(\rC;\xrf)^{-1} = \frac{ \Dervxt\Phibf(\rC;\xrf)^{\sharp}}{\det\Dervxt\Phibf(r;\xrf)}
  \;,\quad \Dervxt\Phibf(\rC;\xrf)^{\sharp} :=
  \begin{bmatrix}
    1+\dfrac{\chi(\rr)}{\rr}\rC(\Bsr) & - \frac{\chi(\rr)}{\rr}
    \dfrac{d \rC}{d
      \Bs}(\Bsr) \\
    \mathbf{0} & \left( 1+\chi'(\rr)\rC(\Bsr) \right)\,\VI_{d-1}
  \end{bmatrix},
\end{gather}
so that $\refs{\VA}(\rC;\refs{\Bx})$ defined by \eqref{eq:At}
can be written as 
\begin{equation}\label{eq:At_new}
\refs{\VA}(\rC;\refs{\Bx}) 
=   
\frac{1}{\det\Dervxt\Phibf(r;\xrf)}\,  \Dervxt\Phibf(\rC;\xrf)^{\sharp}\big(\Dervxt\Phibf(\rC;\xrf)^{\sharp}\big)^\top,
\;\; \xrf\in B_2
\;.
\end{equation}
From the expression \eqref{eq:DPhi_inverse} for $\Dervxt\Phibf( \rC;\refs{\Bx})^{\sharp}$, 
the bounds on $\chi$ in \eqref{chi} imply that there exists $C>0$ (depending only on $d$) such that
\begin{gather}
  \label{eq:DPhib_new}
  \big\|\Dervxt\Phibf( \rC;\refs{\Bx})^{\sharp}
  \big\|_{2} \leq C\big(1+\N{ \rC}_{C^{1}(\bbS^{d-1})}\big)
  \quad\text{ for all } \rC\in \rCS,\;  \xrf\in B_2\;;
\end{gather}
a similar bound holds for 
$\|(\Dervxt\Phibf( \rC;\refs{\Bx})^{\sharp})^\top\|_{2}$. 

We next observe that 
the lower bound \eqref{eq:bdd} on the modulus of the determinant
holds even for complex-valued $\rC\in\rCS$, 
i.e., 
\begin{equation}\label{eq:bdd2}
 |\det\Dervxt\Phibf(\rC;\xrf)| \geq 3^{-d} \;\text{ for all } \rC\in\rCS \;,
\;\;\xrf\in B_2\;.
\end{equation} 
This is because 
$\left|1+\beta \rC(\Bsr)\right| \geq \left|1+\beta \Re \rC(\Bsr)\right|\geq \frac13$ 
if $0\leq \beta \leq 2$ and $-\frac13\leq \Re \rC(\Bsr) \leq \frac13$, 
regardless of the imaginary part of $\rC(\Bsr)$. 
This completes the proof of the bound on $\sup_{\refs{\Bx}}\big\|
\refs{\VA}(\rC;\refs{\Bx})\big\|_2$. 

To prove the bound for
$\refs{n}(\rC;\refs{\Bx})$ we simply appeal to its definition \eqref{eq:nt} and the
expression \eqref{eq:DPd} for $\det\Dervxt\Phibf(\rC;\xrf)$ .
The proof of the bound on
$\sup_{\refs{\Bx}}\big\| \refs{\nablabf}\refs{\VA}(\rC;\refs{\Bx})\big\|_2 $ in Part (ii)
is similar, but more involved because of the differentiation (and thus requiring one more
derivative of $\rC$ -- contrast \eqref{eq:continuity_cond1} and
\eqref{eq:continuity_cond2}). We omit the details.
\epf

We are now in a position to prove Theorem \ref{thm:existence}. 
A central issue here is that for complex-valued $\rC\in\rCS$ 
Theorem \ref{thm:MS} is not immediately applicable.
The \emph{crucial idea} of the proof is 
to rewrite the variational problem \eqref{eq:uhatzdef} 
as
\begin{align*}
  &\int_{B_2}\Big(\wn^{-2}\refs{\VA}(\Re \rC;\refs{\Bx})
  \refs{\nablabf} \refs{u}(\rC;\refs{\Bx}) \cdot \overline{\refs{\nablabf}
    \refs{v}}(\refs{\Bx}) -\refs{n}(\Re \rC;\refs{\Bx})\refs{u}(\rC;\refs{\Bx})
  \overline{\refs{v}}(\refs{\Bx})\Big) \rd \refs{\Bx}-
  \int\nolimits_{\partial B_{2}} \wn^{-1}\dtn(\rst{\refs{u}(\rC;\cdot)}{\partial
    B_{2}})\overline{\refs{v}}\,\mathrm{d}S\\
  &= \widehat{L}(\refs{v})+
  \int_{B_2}\bigg[\wn^{-2}\Big(\refs{\VA}(\Re
  \rC;\refs{\Bx})-\refs{\VA}(\rC;\refs{\Bx})\Big)\refs{\nablabf}
  \refs{u}(\rC;\refs{\Bx}) \cdot\overline{ \refs{\nablabf}
    \refs{v}(\refs{\Bx}) }
  -\Big(\refs{n}(\Re \rC;\refs{\Bx})-\refs{n}(\rC;\refs{\Bx})\Big)
  \refs{u}(\rC;\refs{\Bx})\overline{\refs{v}}(\refs{\Bx})\bigg] \rd \refs{\Bx}\;.
\end{align*}
We then transform back to the physical domain via
$\refs{\Bx}\mapsto \Bx:= \Phibf(\Re \rC;\refs{\Bx})$ 
and obtain that 
$u(\rC;\cdot):=\refs{u}(\rC,\Phibf^{-1}(\Re \rC,\Bx))$ 
solves the variational problem: 

find $u(\rC;\cdot) \in H^1(B_2)$ such that, for all $v \in H^1(B_2)$, 
\begin{multline}
  \label{eq:uhatmapped}
  \int_{B_2} \Big( \wn^{-2}\nablabf u(\rC;\Bx)\cdot \nablabf v(\Bx) - n(\Re \rC; \Bx) u(\rC;\Bx)
  \,\overline{v}(\Bx)\Big)\rd \Bx - \int_{\partial B_2}
  \wn^{-1}\dtn(u(\rC;\cdot)|_{\partial B_2})\overline{v} \, \rd S \\ =L(v) + \widetilde{L}(\rC;v),
\end{multline}
where
\begin{align}
  \label{eq:Ltilde}
  \widetilde{L}(\rC;v) & :=\int_{B_2} \bigg[ \wn^{-2}\widetilde{\VA}(\rC;\Bx) \nablabf
  u(\rC;\Bx)
  \cdot \nablabf v(\Bx) - \widetilde{n}(\rC;\Bx) u(\rC;\Bx) \,\overline{v}(\Bx)\bigg] \rd \Bx\;,\\
  \label{eq:widetildeA}
  \widetilde{\VA}(\rC;\Bx) & :=\left. \frac{\Big( \refs{\VA}(\Re \rC; \refs{\Bx})
      -  \refs{\VA}(\rC; \refs{\Bx})\Big)}
    {  \det \Dervxt\Phibf(\Re \rC;\refs{\Bx}) }\Dervxt\Phibf(\Re \rC;\refs{\Bx})
    \,\Dervxt\Phibf(\Re \rC;\refs{\Bx})^{\top} \right|_{\refs{\Bx}\mapsto \Bx := \Phibf(\Re \rC;\refs{\Bx})}\;,\\
\label{eq:widetilden}
  \widetilde{n}(\rC;\Bx) & :=\left. \frac{\Big(
      \det \Dervxt\Phibf(\Re \rC;\refs{\Bx})
      - \det \Dervxt\Phibf(\rC;\refs{\Bx})\Big)
      n_{0}(\refs{\Bx})
    }{ 
      \det \Dervxt\Phibf(\Re \rC;\refs{\Bx}) } \right|_{\refs{\Bx}\mapsto \Bx := \Phibf(\Re \rC;\refs{\Bx})
  }\;.
\end{align}
The key point is that \eqref{eq:uhatmapped} is a variational formulation
corresponding to a Helmholtz problem with real coefficients, 
and thus we can apply Theorem~\ref{thm:MS}.

\begin{lemma}
  \label{lem:difference}
  Given $d>0$, there exist $C_{\rm diff,1}, C_{\rm diff,2}>0$ such that, for all $\rC\in\rCS$,
  \begin{subequations}
    \label{eq:2_differences}
    \begin{align}
      \label{eq:2_differences1}
      \N{\widetilde{n}(\rC;\cdot)}_{L^{\infty}(B_{2})}
      & \leq C_{\rm diff,1}\N{\Im \rC}_{C^{0}(\bbS^{d-1})}\big(1+\N{\Im \rC}_{C^{0}(\bbS^{d-1})}\big)^{d-1},\\
      \label{eq:2_differences2}
      \big\|\widetilde{\VA}(\rC,\cdot)\big\|_{L^{\infty}(B_{2})} & \leq C_{\rm diff,2}
      \N{\Im \rC}_{C^{1}(\bbS^{d-1})}\big(1+\N{\rC}_{C^{1}(\bbS^{d-1})}\big)^{{2d-2}}. 
    \end{align}
  \end{subequations}
\end{lemma}

\bpf[Proof of Theorem \ref{thm:existence} assuming Lemma \ref{lem:difference}]
(i) Given $C_{\rm Re}>0$, let $c_1, c_2>0$ be as in Lemma \ref{lem:strong_elliptic}.
    Assume that 
$\N{\Re \rC}_{C^{1}(\bbS^{d-1})}\leq C_{\Re}$ and  $\|\Im \rC\|_{C^1(\mathbb{S}^{d-1})}\leq c_1$ 
(in the course of the proof we will restrict $\|\Im \rC\|_{C^1(\mathbb{S}^{d-1})}$ further).
In particular, $\N{\rC}_{C^{1}(\bbS^{d-1})}\leq C_{\Re} + c_1=:C_{0}$.

By Corollary \ref{cor:sesqui}, the sesquilinear form $\refs{a}(\cdot,\cdot)$ 
is continuous and satisfies a G\aa rding inequality, with constants independent of $\rC$. 
Therefore, by Fredholm theory, if, under the assumption of existence, one has a bound on the solution
of the Helmholtz transmission problem in terms of the data $\refs{L}$, then the solution
exists and is unique (for more details of this method of arguing applied to the Helmholtz equation, 
see, e.g., \cite[Lemma 3.5]{GPS19}). 

It is therefore sufficient to assume that $\refs{u}(\rC;\cdot)$ exists and show that the
bound \eqref{eq:holo_thm2} holds under the conditions on $\N{\Re \rC}_{C^{1}(\bbS^{d-1})}$
and $\N{\Im \rC}_{C^{1}(\bbS^{d-1})}$ that are implied by $\rC \in \Ca(k;C_{\Re},C_{\Im})$,
with the constant $c_1$ in Lemma~\ref{lem:strong_elliptic} given by $c_1 = C_{\Im} / k$.

By the definition of $\widetilde{L}$ \eqref{eq:Ltilde}, 
the bounds \eqref{eq:2_differences}, 
the assumption $\N{\rC}_{C^{1}(\bbS^{d-1})}\leq C_{0}$, 
and the Cauchy-Schwarz inequality
\begin{align*}
 |\widetilde{L}(v)|
  & \leq   \wn^{-2} \big\|\widetilde{\VA}(\rC,\cdot)\big\|_{L^{\infty}(B_{2})} \N{\nablabf
    u(\rC;\cdot)}_{L^2(B_2)} \N{\nablabf v}_{L^2(B_2)}  + 
  \N{\widetilde{n}(\rC;\cdot)}_{L^{\infty}(B_{2})} \N{u(\rC;\cdot)}_{L^2(B_2)}
  \N{v}_{L^2(B_2)} \\
  & \leq C_{\rm diff}\N{\Im \rC}_{C^{1}(\bbS^{d-1})}
  \Big( \wn^{-2}\N{\nablabf u(\rC;\cdot)}_{L^2(B_2)} \N{\nablabf v}_{L^2(B_2)} 
  +\N{u(\rC;\cdot)}_{L^2(B_2)} \N{v}_{L^2(B_2)} \Big),
\end{align*}
with $C_{\rm diff}>0$ depending only on $C_0,$ $C_{2}$, and on the constants in
Lemma~\ref{lem:difference}.  
Thus, by the definitions of $\|\cdot\|_{(H^1_\wn (B_2))^*}$
and $\|\cdot\|_{H^1_\wn (B_2)}$ in \eqref{eq:weighted_norms} and the Cauchy--Schwarz inequality,
\begin{align*}
  \big\|\widetilde{L}\big\|_{(H^1_\wn (B_2))^*} 
  &
  \leq  C_{\rm diff} \N{\Im \rC}_{C^{1}(\bbS^{d-1})} \N{u(\rC;\cdot)}_{H^1_\wn (B_2)}.
\end{align*}
Applying the a priori bound \eqref{eq:MS4}, 
we obtain that the solution $u$ of \eqref{eq:uhatmapped} satisfies
\begin{equation*}
  \N{u(\rC;\cdot)}_{H^1_\wn (B_2)} 
  \leq 
  \Csolt(k, n_i) \Big( C_{\rm diff} \N{\Im \rC}_{C^{1}(\bbS^{d-1})}
  \N{u(\rC;\cdot)}_{H^1_\wn (B_2)} + \N{L}_{(H^1_\wn (B_2))^*}\Big).
\end{equation*}
Now, if 
\begin{equation}\label{eq:holo_thm5}
   \Csolt(\wn, n_i) C_{\rm diff} \N{\Im \rC}_{C^{1}(\bbS^{d-1})} \leq 1/2,
\end{equation}
then 
\begin{equation}\label{eq:holo_thm2a}
  \N{u(\rC;\cdot)}_{H^1_\wn (B_2)} \leq 2\Csolt(k, n_i)
  \N{L}_{(H^1_\wn (B_2))^*}\;.
\end{equation}
If $\wn \geq 1/2$, then the explicit expressions \eqref{eq:MS2} and
\eqref{eq:MS5} for $\Csolo$ and $\Csolt$, respectively, imply that
\begin{equation*}
  \Csolt(\wn, n_i) \leq 
  \frac{\wn}{n_i} \left( 1+\frac{{8}}{\sqrt{n_i}}\sqrt{ 1 + \frac{1}{n_i} \left(1+\frac{d-1}{2}\right)^2 } \right).
\end{equation*}
We now convert the bound \eqref{eq:holo_thm2a} on $u$ in terms of $L$ into the bound \eqref{eq:holo_thm2} on $\refs{u}$ in terms of $\refs{L}$. Indeed, 
\eqref{eq:holo_thm2} follows (for a suitable $C_1$ depending on $d$, $n_i$, and $C_0$) by using 
the expression for $ \Csolt(\wn, 2, n_i)$ along with the norm equivalence from Lemma \ref{lem:normtrf}
combined with the bound $\N{\rC}_{C^{1}(\bbS^{d-1})}\leq C_{0}$ (and observing that the norm equivalence in
$\|\cdot\|_{H^1_\wn(B_2)}$ implies a similar equivalence in
$\|\cdot\|_{(H^1_\wn(B_2))^*}$).
We now let 
\begin{equation*}
C_{\Im} := \left(
  \frac{2}{n_i} \left( 1+\frac{8}{\sqrt{n_i}}\sqrt{ 1 + \frac{1}{n_i} \left(1+\frac{d-1}{2}\right)^2 } \right) C_{\rm diff}
\right)^{-1}
\end{equation*}
and $k_0:= \min\{C_{\Im}/c_1, 1/2\}$, 
so that the condition on $\Im \rC$ in \eqref{eq:holo_thm1} implies that $\Im \rC$ 
satisfies both \eqref{eq:holo_thm5} and $\|\Im \rC\|_{C^1(\mathbb{S}^{d-1})}\leq c_1$.

(ii) The assumption that $\refs{L}(\refs{v})= \int_{B_2}f \overline{\refs{v}}$ for $f\in L^2(B_2)$ implies that the variational problem \eqref{eq:uhatzdef} is \emph{equivalent} to the (variational formulation of the) 
boundary-value problem
\begin{align*}
k^{-2} \refs{\nablabf}\cdot\big(\refs{\VA}(\rC;\refs{\Bx}) \refs{\nablabf}
      \refs{u}(\rC;\refs{\Bx})\big)&= 
      -\refs{n}(\rC;\refs{\Bx})
      \refs{u}(\rC;\refs{\Bx})- f(\refs{\Bx}) \quad\text{ for } \refs{\Bx} \in B_2,\\
k^{-1} \partial_n \refs{u} &= \dtn\big(\refs{u}|_{\partial B_2}\big) \quad\text{ on } \partial B_2.
\end{align*}
Then, by the $H^2$ regularity result of \cite[Theorem 6.1]{LaSpWu:22}, 
there exists $C>0$ (depending on $\sup_{\refs{\Bx}}\|\refs{\VA}(\rC;\refs{\Bx})\|_2$, $\sup_{\refs{\Bx}}\|\refs{\nablabf}\refs{\VA}(\rC;\refs{\Bx})\|_2$, and the constant $c_2$ in \eqref{eq:strong_elliptic}) 
such that, for all $k\geq 1/2$, say,
\begin{equation*}
\big\|\refs{u}\big\|_{H^2_k(B_2)} \leq C \big( \big\| \refs{u}\big\|_{H^1_k(B_2)} + \big\| f\big\|_{L^2(B_2)}\big).
\end{equation*}
The bound \eqref{eq:holo_thm2_H2} then follows from the bound \eqref{eq:holo_thm2} 
from Part (i) and Lemmas \ref{lem:strong_elliptic} and \ref{lem:continuity} (with the assumption $\|\rC\|_{C^2(\mathbb{S}^{d-1})}\leq C$ required to apply Part (ii) of Lemma \ref{lem:continuity}).
\epf

It therefore remains to prove Lemma \ref{lem:difference}. 
To proceed, we record the following two elementary results.
\begin{lemma}
  \label{lem:elementary}
  Let $V$ be a normed vector space, $D\subset V$ some subset. 
  If the matrix-valued functions
  $\VF_j:D\to \bbR^{m\times m} $, $j=1,\ldots, n$, $m,n\in\bbN$, are such
  that there exist $C_1,C_{2}>0$ such that, for all $n\in \bbN$ and 
  for all $j=1,\ldots, n$,
  \begin{equation}\label{eq:elementary2}
    \N{\VF_j(\Vy)- \VF_j(\Vz)} \leq C_1 \N{\Vy-\Vz}_{V} \;\forall \Vy,\Vz\in D
    \quad\text{ and }\quad \N{\VF_j(\Vy)}\leq C_2\;\forall \Vy\in D
  \end{equation}
  then 
  \begin{equation}\label{eq:elementary1}
    \left\|{ \prod_{j=1}^n\VF_j(\Vy) -\prod_{j=1}^n \VF_j(\Vz)} \right\|\leq  nC_1C_{2}^{n-1} \N{\Vy-\Vz}_{V},
  \end{equation}
  for any sub-multiplicative matrix norm $\left\|\cdot\right\|$. 
\end{lemma}

\bpf
Defining a product to be $\VI$ in case the lower index exceeds the upper, we have 
\begin{align*}
  \prod_{j=1}^n\VF_j(\Vy) -\prod_{j=1}^n \VF_j(\Vz) & = \sum\limits_{k=1}^{n}\left(
    \prod\limits_{j=1}^{n-k+1}\VF_{j}(\Vy)\cdot\prod\limits_{j=n-k+2}^{n}\VF_{j}(\Vz) -
    \prod\limits_{j=1}^{n-k}\VF_{j}(\Vy)\cdot\prod\limits_{j=n-k+1}^{n}\VF_{j}(\Vz)\right) \\
  & =
  \sum\limits_{k=1}^{n}
  \left(
    \left(\prod\limits_{j=1}^{n-k}\VF_{j}(\Vy)\right)(\VF_{n-k+1}(\Vy)-\VF_{n-k+1}(\Vz))
    \left(\prod\limits_{j=n-k+2}^{n}\VF_{j}(\Vz)\right)
  \right)
\end{align*}
(where the right-hand side of the first equality is a telescoping sum). 
The result then follows from sub-multiplicativity and the triangle inequality. 
\epf

\begin{lemma}
  \label{lem:quotest}
Given $z_{1},z_{2}\in\bbC$ with $|z_{1}|,|z_{2}|\geq
  L_{0}\esc{>} 0$, for any two elements $\Vx_{1},\Vx_{2}\in V$ of a normed vector space
  $V$, 
  \begin{gather}
    \label{eq:quotest}
    \left\| \frac{\Vx_{1}}{z_{1}} - \frac{\Vx_{2}}{z_{2}}\right\|_{V} \leq L_{0}^{-2}
    \Big(\N{\Vx_{1}}_{V} \left|z_{2}-z_{1}\right| + |z_{1}|
      \N{\Vx_{2}-\Vx_{1}}_{V}\Big) \;.
  \end{gather}
\end{lemma}

\bpf
The estimate is an immediate consequence of the triangle inequality and 
\begin{gather*}
  \frac{\Vx_{1}}{z_{1}} - \frac{\Vx_{2}}{z_{2}} =
  \frac{z_{2}\Vx_{1}-z_{1}\Vx_{2}}{z_{1}z_{2}} =
  \frac{(z_{2}-z_{1})\Vx_{1}+z_{1}(\Vx_{1}-\Vx_{2})}{z_{1}z_{2}} .
\end{gather*}
\epf

\bpf[Proof of Lemma \ref{lem:difference}] We first prove the bound on $\widetilde{n}$ in
\eqref{eq:2_differences}.  By the bound \eqref{eq:bdd2}, it is sufficient to bound the
numerator in \eqref{eq:widetilden}.  We do this for fixed $\xrf=\refs{\rho}\Bsr\in B_{2}$
using the formula \eqref{eq:DPd} for $\det\Dervxt\Phibf(\rC;\xrf)$ and
Lemma~\ref{lem:elementary} with $V=C^{0}(\bbS^{d-1},\bbC)$, $n=d$, $m=1$,
$f_{1}(\rC;\refs{\Bx})=1+\chi'(\refs{\rho})\rC(\Bsr)$,
$f_{\ell}(\rC;\refs{\Bx})=1+\frac{\chi(\refs{\rho})}{\refs{\rho}}\rC(\Bsr)$, 
and
$\ell=2,\ldots,d$. 
With \eqref{chi}, we find the crude bounds
\begin{gather}
  \label{eq:fjb}
  \begin{aligned}
    |f_{\ell}(\rC;\refs{\Bx})| & \leq{1 + 2 \N{\rC}_{C^{0}(\bbS^{d-1})}}, \\
    |f_{\ell}(\rC;\refs{\Bx})-f_{\ell}(\rC';\refs{\Bx})| & \leq 2 \N{\rC-\rC'}_{C^{0}(\bbS^{d-1})} 
  \end{aligned}\quad\text{ for all } \rC\in\rCS,\; \ell=1,\ldots,d\;,
\end{gather}
and for all $\xrf=\refs{\rho}\Bsr\in B_{2}$, $\Bsr\in\bbS^{d-1}$.
Since the bounds are uniform in $\xrf$ we conclude that
\begin{equation}\label{eq:temp1}
  \sup_{\refs{\Bx}}\big| \det \Dervxt\Phibf(\Re \rC;\refs{\Bx})- \det
  \Dervxt\Phibf(\rC;\refs{\Bx})\big|\leq 2d
  \big(1 + 2 \N{\rC}_{C^{0}(\bbS^{d-1})})\big)^{d-1}
  \,\N{\Im
    \rC}_{C^{0}(\bbS^{d-1})}\;.
\end{equation}
Combining this bound with \eqref{eq:bdd2} and using the fact that the
real part of $\rC$ is bounded by $\frac{1}{3}$, we obtain \eqref{eq:2_differences1} with
$C_{\rm diff,1}>0$ depending only on $d$. 

We now prove the bound \eqref{eq:2_differences2} on $\widetilde{\VA}(\rC,\Bx)$ \eqref{eq:widetildeA}. As above we
fix $\xrf=\refs{\rho}\Bsr\in B_{2}$. Thanks to the lower bound for
$ |\det\Dervxt\Phibf(\rC;\xrf)|$ \eqref{eq:bdd2}, it remains to bound (i) 
$\N{\Dervxt\Phibf(\Re \rC;\refs{\Bx})}_{2}$
 and (ii) 
$ \| \refs{\VA}(\Re \rC; \refs{\Bx}) - \refs{\VA}(\rC; \refs{\Bx})\|_2$. 

Regarding (i), in an almost-identical way to how we obtained \eqref{eq:DPhib_new},
the explicit formula \eqref{eq:DPhi} and
the bounds from \eqref{chi} imply that there exists $C>0$ (depending only on $d$) such that
\begin{gather*}
  \N{\Dervxt\Phibf( \Re\rC;\refs{\Bx})
  }_{2} \leq C\big(1+\N{\Re \rC}_{C^{1}(\bbS^{d-1})}\big)
  \quad\text{ for all } \rC\in \rCS,
\end{gather*}

Regarding (ii), to bound $ \refs{\VA}(\Re \rC; \refs{\Bx}) - \refs{\VA}(\rC; \refs{\Bx})$
we use the expression \eqref{eq:At_new} for $\refs{\VA}(\rC;\refs{\Bx})$ and
Lemma~\ref{lem:elementary} with $m=d$, $n=3$, $V=C^{1}(\bbS^{d-1},\bbC)$ and the
terms $\VF_{j}$ given by the factors in \eqref{eq:At_new}; i.e.
\begin{gather}
  \label{eq:fj_def}
  \VF_{1}(\rC) := \big(\det \Dervxt\Phibf(\rC;\refs{\Bx})\big)^{-1}\VI\;,
  \quad  
  \VF_{2}(\rC) := \Dervxt\Phibf(\rC;\refs{\Bx})^{\sharp}\;,\quad\text{ and } \quad
  \VF_{3}(\rC) :=\big(\Dervxt\Phibf(\rC;\refs{\Bx})^{\sharp}\big)^\top\;.
\end{gather}
For all $\rC \in \rCS$ and $\refs{\Bx}\in B_2$, $\| \VF_1 (\rC)\|_2 \leq 3^d $ by
\eqref{eq:bdd2}, and
\begin{equation*}
\| \VF_1 (\Re \rC) - \VF_1 (\rC) \|_2 \leq 9^d 2 d \big(1+ 2\|\rC\|_{C^0(\mathbb{S}^{d-1})}\big)^{d-1} \| \Im \rC\|_{C^0(\mathbb{S}^{d-1})}.
\end{equation*}
by \eqref{eq:temp1}.
Furthermore, 
$\| \VF_2(\rC)\|_2\leq C( 1 + \|\rC\|_{C^1(\mathbb{S}^{d-1})})$ by \eqref{eq:DPhib_new}.
Now, by the second equation in \eqref{eq:DPhi_inverse},
\begin{gather*}
  \Dervxt\Phibf(\Re \rC;\refs{\Bx})^{\sharp}- \Dervxt\Phibf(\rC;\refs{\Bx})^{\sharp} =-\ri
  \begin{bmatrix}
    \dfrac{\chi(\refs{\rho})}{\refs{\rho}}\Im \rC(\Bsr) & -
    \dfrac{\chi(\refs{\rho})}{\refs{\rho}} \Im \dfrac{d \rC}{d\Bs}(\Bsr) \\
    \mathbf{0} & \chi'(\refs{\Bs}) \Im \rC(\Bsr) \VI_{d-1}\;,
  \end{bmatrix},
\end{gather*}
which implies that 
\begin{gather*}
  \|\VF_2(\Re\rC) -\VF_2 (\rC)\|_2 
  = 
  \N{ \Dervxt\Phibf(\Re \rC;\refs{\Bx})^{\sharp} - \Dervxt\Phibf(\rC;\refs{\Bx})^{\sharp}}_{2} 
  \leq C \N{\Im \rC}_{C^{1}(\bbS^{d-1})}\;,
\end{gather*}
with a constant $C$ depending on $d$ alone. 
Finally, $\VF_3$ satisfies exactly the
same bounds as $\VF_2$; thus we can apply Lemma \ref{lem:elementary} with
\begin{equation*}
C_1 = C (1+\N{\rC}_{C^{0}(\bbS^{d-1})})^{d-1}\N{\Im \rC}_{C^{1}(\bbS^{d-1})}
\quad\text{ and } \quad
C_2=1+\N{ \rC}_{C^{1}(\bbS^{d-1})}\;,
\end{equation*}
to obtain \eqref{eq:2_differences2}.
\epf

\section{\rev{Recap of $\wn$-explicit Finite-Element Error Bounds}}
\label{sec:ExpBdsFEM}


In \S\ref{sec:FreDepHol} we proved existence, uniqueness, and $k$-explicit bounds on
  the solution $\refs{u}(\rC;\cdot)$ to the variational problem \eqref{eq:uhatzdef}; 
  see Theorem~\ref{thm:existence}.
In this section we \esc{recall results from \cite{GaLaSp:21a, GS3} about computing approximations to $\refs{u}(\rC;\cdot)$ using the $h$-version of the FEM (i.e., where the polynomial degree is fixed and accuracy is increased by decreasing $h$) with the radiation condition approximated by a radial perfectly-matched layer (PML).}
    \rev{When applied to our set up, \cite{GaLaSp:21a, GS3} establish FEM error bounds for each $\mathfrak{r}\in\mathcal{A}$ \eqref{eq:Anset}, with the constants in principle depending on $\mathfrak{r}$. 
In this section we go through \cs{some of} the arguments in \cite{GaLaSp:21a, GS3} 
to justify that the constants in these error bounds can be taken to be independent of $\mathfrak{r}\in\mathcal{A}$.}
  
%

\subsection{Definition of radial PML truncation.}
The solution $\refs{u}(\rC;\refs{\Bx})$ to the variational problem \eqref{eq:uhatzdef} 
is approximated by the solution of the following problem.
Let $2< R_1 < R_{\rm tr}<\infty$ and let $\Omega_{\rm tr} \supset B(0, R_{\rm tr})$ 
be a bounded Lipschitz domain, and 
let $\refs{u}_{\rm PML}\in H^1_0(\Omega_{\rm tr})$ be 
the solution of the variational problem 
\begin{align}\label{eq:PMLvf}
  \boxed{
\int_{\Omega_{\rm tr}}
\bigg( \wn^{-2}
\refs{\VA}_{\rm PML}(\rC;\refs{\Bx}) 
\refs{\nablabf} \refs{u}_{\rm PML} (\rC;\refs{\Bx}) 
\cdot 
\overline{\refs{\nablabf} \refs{v}(\rC;\refs{\Bx})} - \refs{n}_{\rm PML}(\rC;\refs{\Bx}) 
\refs{u}_{\rm PML}(\rC;\refs{\Bx})\overline{\refs{v}(\rC;\refs{\Bx})}\bigg)\rd \refs{\Bx} 
= 
\refs{L}(\refs{v}) }
\end{align}
for all $\refs{v} \in H^1_0(\Omega_{\rm tr})$, 
where $\refs{\VA}_{\rm PML}$ and $\refs{n}_{\rm PML}$ 
are defined in terms of $\refs{\VA}, \refs{n}$, 
and functions $\alpha, \beta$ by 
\beq\label{eq:defAPML}
\refs{\VA}_{\rm PML}(\rC;\refs{\Bx}) 
:= 
\begin{cases}
\refs{\VA}(\rC;\refs{\Bx}) & \text{ for } |\refs{\Bx}| \leq R_1 
\\
\VH\VD\VH^T &\text{ for } |\refs{\Bx}|> R_1,
\end{cases}
\quad\text{ and }\quad
\refs{n}_{\rm PML}(\rC;\refs{\Bx}) 
:= 
\begin{cases}
\refs{n}(\rC;\refs{\Bx}) & \text{ for } |\refs{\Bx}| \leq R_1 
\\
\alpha(\refs{\rho}) \beta(\refs{\rho})^{d-1} &\text{ for } |\refs{\Bx}|> R_1,
\end{cases}
\eeq
where, in polar/spherical coordinates 
(with, as in \S\ref{sec:param}, $|\refs{\Bx}|= \refs{\rho}$),
\begin{align*}
  \VD & :=
  \left(
    \begin{array}{cc}
      \beta(\refs{\rho})\alpha(\refs{\rho})^{-1} &0 \\
      0 & \alpha(\refs{\rho}) \beta(\refs{\rho})^{-1}
    \end{array}
  \right)\;, & \VH & :=
\left(
\begin{array}{cc}
\cos \theta & - \sin\theta \\
\sin \theta & \cos\theta
\end{array}
\right) & \text{ for } d=2\;,
\intertext{and}
\VD & :=
\left(
\begin{array}{ccc}
\beta(\refs{\rho})^2\alpha(\refs{\rho})^{-1} &0 &0\\
0 & \alpha(\refs{\rho}) &0 \\
0 & 0 &\alpha(\refs{\rho})
\end{array}
\right) 
\;, &
\VH & :=
\left(
\begin{array}{ccc}
\sin \theta \cos\phi & \cos \theta \cos\phi & - \sin \phi \\
\sin \theta \sin\phi & \cos \theta \sin\phi & \cos \phi \\
\cos \theta & - \sin \theta & 0 
\end{array}
\right)  & 
\text{ for } d=3.
\end{align*}
The functions $\alpha(\refs{\rho})$ and $\beta(\refs{\rho})$ 
are defined as follows:~given a radial function $\widetilde{\sigma}$ and $R_2>R_1$ 
such that 
\begin{subequations}\label{eq_sigma_prop}
\begin{align}
&\widetilde{\sigma}(\refs{\rho})
 = 0 
\text{ for } 
 \refs{\rho}\leq R_1, \\
&\widetilde{\sigma}(\refs{\rho}) \text{ is increasing}  \text{ for } R_1 \leq \refs{\rho}\leq R_2, \text{ and }\\
&\widetilde{\sigma}(\refs{\rho}) =\sigma_0>0  \text{ for } \refs{\rho}\geq R_2,
\end{align}
\end{subequations}
let 
\beq\label{eq_sigma}
\sigma(\refs{\rho}) := \big( \refs{\rho} \widetilde{\sigma}(\refs{\rho})\big)', \quad \alpha(\refs{\rho}) := 1 + \ri \sigma(\refs{\rho}), \quad \text{ and }\quad \beta(\refs{\rho}) := 1 + \ri \widetilde{\sigma}(\refs{\rho}).
\eeq
We note that $\VD=\VI_{d-1}$ and $n=1$ when $\refs{\rho}=R_1$ and thus $\refs{\VA}_{\rm PML}$ and $\refs{n}_{\rm PML}$ are continuous at $r=R_1$. 
We also note that $R_{\rm tr}$ can be $<R_2$, i.e., 
we allow truncation before $\widetilde{\sigma}$ reaches $\sigma_0$. 

\begin{remark}[The definition of the PML scaling function]
  \label{rem:pmlscf}
  Here we have followed, e.g., \cite[\S2]{BrPa:07}, \cite[\S1.2]{GaLaSp:21a} and, starting
  from $\widetilde{\sigma}$, defined $\sigma$ in terms of
  $\widetilde{\sigma}$. Alternatively, one can start from a non-decreasing function
  $\sigma$ and define $\widetilde{\sigma}$ such that the first equation in
  \eqref{eq_sigma} holds; see, e.g., \cite[\S3]{CoMo:98a}, \cite[\S2]{LaSo:98},
  \cite[\S4]{HoScZs:03}, \cite[\S2]{ChGaNiTo:22}. The notation $\alpha$ and $\beta$ is
  also used by \cite{LaSo:98, LiWu:19}.
\end{remark}

\subsection{Properties of the sesquilinear form and solution operator of the PML problem}

\begin{lemma}[Sign property of $\Re (\refs{\VA}_{\rm PML})$]\label{lem:signPML}
Given $C_{\Re}>0$, let $\cA, C_{\Im},k_0$ be as in Corollary \ref{cor:A}.
Given $\widetilde{\sigma}$ as in \eqref{eq_sigma_prop}, 
there exists $C>0$ such that, 
for all $\rC \in \cA$ as defined in \eqref{eq:Anset}, 
for all $\refs{\Bx} \in \mathbb{R}^d$, and 
for all $\xi \in \mathbb{R}^d$, 
\beq\label{eq:signPML}
\Re \big( \refs{\VA}_{\rm PML}(\rC; \refs{\Bx}) \xi, \xi\big)_2 \geq C \|\xi\|^2_2. 
\eeq
\end{lemma}

\bpf
\esc{The sign property \eqref{eq:signPML} with the constant $C$ not uniform in $\rC\in \mathcal{A}$ is given in, e.g., \cite[Lemma 2.3]{GLSW22} (where $\widetilde{\sigma}(r)= f_\theta(r)/r$ in the notation of \cite[Lemma 2.3]{GLSW22}). 
Since $\rC$ does not enter the definition of the PML scaling function $\widetilde{\sigma}$, the fact that $C$ in \eqref{eq:signPML} can be taken to be independent of 
$\rC\in \mathcal{A}$ follows from Lemma \ref{lem:strong_elliptic}.}
In using Lemma \ref{lem:strong_elliptic}, we use the fact that, \esc{by design,} the constants $C_{\Im},k_0$ from Corollary \ref{cor:A} are such that if $\rC \in \mathcal{A}$ defined by \eqref{eq:Anset}, then the assumptions of Lemma  \ref{lem:strong_elliptic} are satisfied.
\epf

Lemma \ref{lem:signPML} and Corollary \ref{cor:sesqui} then imply the following analogue of Corollary \ref{cor:sesqui}.

\begin{corollary}\label{cor:PMLGarding}
The sesquilinear form in the variational problem \eqref{eq:PMLvf} is continuous and satisfies a G\aa rding inequality, with both the continuity constant and the constants in the G\aa rding inequality \emph{independent of $\rC$} for $\rC \in \cA$.
\end{corollary} 

\begin{theorem}[PML solution operator inherits behaviour of non-truncated solution operator]\label{thm:inherit}
  Given $n_i<1,$ $d=2$ or $3$, and $C_{\Re}>0$, let $\cA$ and $C_{\Im}$ be as in \eqref{eq:Anset}. 
  Then, given $\widetilde{\sigma}\in C^3(0,\infty)$ as in \eqref{eq_sigma_prop},
  there exists $C, k_1 \geq \frac12$ such that, for all $\rC \in \cA$, for any
  $\refs{L}\in (H^1_\wn (B_2))^*$, and for all $k\geq k_1$, the solution
  $\refs{u}_{\rm PML}(\rC;\cdot)$ of \eqref{eq:PMLvf} exists, 
  is unique, and satisfies the bound
  \begin{equation}
    \label{eq:inherit1}
    \N{\refs{u}_{\rm PML}(\rC;\cdot)}_{H^1_\wn (\Omega_{\rm tr})} 
    \leq C \wn  \big\|\refs{L}\big\|_{(H^1_\wn(\Omega_{\rm tr}))^*}\;.
  \end{equation}
\end{theorem}

\bpf
\esc{The result with $C$ in \eqref{eq:inherit1} not uniform in $\rC\in \mathcal{A}$ follows from \cite[Theorem 1.6]{GaLaSp:21a}.
}
Indeed, \cite[Theorem 1.6]{GaLaSp:21a} proves that the solution operator of the PML problem 
is bounded by the solution operator of the original (i.e., non-truncated) problem 
for all scattering problems \esc{belonging to} the black-box scattering framework of \cite{SjZw:91}. 

\esc{We now justify that $C$ in \eqref{eq:inherit1} can be taken to be independent of $\rC\in \mathcal{A}$. An apparent difficulty is that }
the variational problem \eqref{eq:uhatzdef} does not fit in the black-box framework as
described in, e.g., \cite[Chapter 4]{DyZw:19}, since the PDE is not formally self-adjoint
when $\Im \rC\neq 0$ (because the coefficients are complex). Nevertheless, the results
from the black-box framework required to prove \cite[Theorem 1.6]{GaLaSp:21a} still
hold. Indeed, \cite[Theorem 1.6]{GaLaSp:21a} follows from \cite[Lemma 3.3]{GaLaSp:21a},
and the two ingredients of this lemma that depend on the contents of the black box are (i)
that the PML problem is Fredholm of index zero, and (ii) agreement away from the scaling
region of the solution operators of the complex-scaled and unscaled Helmholtz problems
(see, e.g., \cite[Theorem 4.37]{DyZw:19}).

\esc{Regarding (i):~for the variational problem \eqref{eq:PMLvf}, this property holds since the sesquilinear form is continuous and satisfies a G\aa rding inequality (by Corollary \ref{cor:PMLGarding}) and the solution is unique for each $\rC \in \mathcal{A}$ by the a priori bound of \cite[Theorem 1.6]{GaLaSp:21a} (it does not matter that the constant a priori depends on $\rC\in \mathcal{A}$ to establish this uniqueness)}.

Regarding
(ii):~inspecting the proof of \cite[Theorem 4.37]{DyZw:19}, we see that the only place
where this proof uses that the contents of the black-box are self adjoint is in finding a
complex $k$ where the (unscaled) problem can be shown to have a unique solution and
satisfies the natural bound on the solution operator (in terms of $|k|$-dependence); see
\cite[Lemma 4.3]{DyZw:19} \esc{(the existence of such a complex $k$ then allows one to use analytic Fredholm theory; see, e.g., \cite[Theorem C.8]{DyZw:19})}. In our case, if $k=\ri \lambda$, then one can prove directly
from the variational formulation \esc{(i.e., by integration by parts)}, using Lemma \ref{lem:signPML} and the definition of
$\refs{n}$ \eqref{eq:nt}, that the problem \eqref{eq:uhatzdef} with $\rC \in \cA$ has a
unique solution and satisfies the bound \eqref{eq:holo_thm2} with $k$ replaced by
$\lambda$.  \epf

\begin{remark}\label{rmk:PMLHol}
As the transformations $\refs{\Bx}\mapsto \Phibf(\rC;\refs{\Bx})$
leave the PML-zone invariant, 
$\rC\mapsto\refs{u}_{\rm PML}(\rC;\cdot)$ is still analytic on $\Ca$.
\end{remark}

\subsection{The accuracy of PML truncation for $k$ large.}


\begin{theorem}[Radial PMLs are exponentially accurate for $k$ large]\label{thm:GLS}
Given $n_i<1$, $d=2$ or $3$, and $C_{\Re}>0$, let $\cA$ and $C_{\Im}$ be as in Corollary \ref{cor:A}.
Given $\widetilde{\sigma}\in C^3(0,\infty)$ as in \eqref{eq_sigma_prop}, and $\epsilon>0$, there exist $C_{\rm PML, 1}, C_{\rm PML, 2}, k_1>0$ such that  
the following is true for all 
$R_{\rm tr}>(1 + \epsilon)R_1$, $\rC \in \cA$, and $k\geq k_1$.

Given $\refs{f}\in L^2(B_2)$, let $\refs{u}$ be the solution of the variational problem \eqref{eq:uhatzdef} with $\refs{L}(\refs{v}) = \int_{B_2}\refs{f}\overline{\refs{v}}\rd \refs{\Bx}$, and let $\refs{u}_{\rm PML}(\rC;\cdot)$ be the solution of the variational problem \eqref{eq:PMLvf} with the same $\refs{L}$. Then $\refs{u}_{\rm PML}(\rC;\cdot)$ exists, is unique, and satisfies 
\beq\label{eq:GLS}
\|\refs{u}(\rC;\cdot)-\refs{u}_{\rm PML}(\rC;\cdot)\|_{H^1_k(B(0,R_1))} \leq C_{\rm PML, 1} \exp \Big( - C_{\rm PML, 2} k\big(R_{\rm tr}-(1+\epsilon) R_1\big)
\Big) 
\big\|\refs{f}\big\|_{L^2(B_2)}.
\eeq
\end{theorem}

\bpf
\esc{The result with the constants $C_{\rm PML, 1}, C_{\rm PML, 2}, k_1$ not uniform in $\rC\in \mathcal{A}$ follows from \cite[Theorem 1.4]{GaLaSp:21a}.
The same arguments in the proof of Theorem \ref{thm:inherit} above show that $C_{\rm PML, 1}, C_{\rm PML, 2}, k_1$ can be taken independent of $\rC\in \mathcal{A}$.
Indeed, like \cite[Theorem 1.4]{GaLaSp:21a}, \cite[Theorem 1.6]{GaLaSp:21a} uses 
}
agreement away from the scaling region of the solution operators of the complex-scaled and unscaled Helmholtz problems \cite[Theorem 4.37]{DyZw:19}; this result holds for our problem as explained in the proof of Theorem \ref{thm:inherit}.
\epf

Theorem \ref{thm:GLS} shows that the error in approximating $\refs{u}(\rC;\cdot)$ by $\refs{u}_{\rm PML}(\rC;\cdot)$ 
decreases exponentially in the wavenumber $k$ and the PML width $R_{\rm tr}-R_1$, 
uniformly for $\rC\in \cA$.

\subsection{The accuracy of the $h$-FEM approximation of the PML solution}\footnote{\cs{Here, and throughout, 
the notion ``$h$-FEM'' refers to the situation where p is fixed and accuracy is increased by varying $h$.}} 
\label{sec:hFEAcc}
We consider subspaces $(V_h)_{h>0}$ of $H^1(\Omega_{\rm tr})$ 
satisfying the following assumption.
\begin{assumption}\label{ass:spaces}
  For some ``polynomial degree'' $p\in \mathbb{N}$, the subspaces $(V_{h})_{h>0}$
  satisfy the following: there exists a constant $C>0$ such that, for all $h>0$ and
  $0\leq \ell\leq p$, given $v \in H^1_0(\Omega_{\rm tr})\cap H^{\ell+1}(\Omega_{\rm tr})$
  there exists $\mathcal{I}_{h,p}v \in V_{h,p}$ such that \beq\label{eq:pp_approx} |v-
  \mathcal{I}_{h,p}v|_{H^j(\Omega_{\rm tr})} \leq C h^{\ell + 1-j} \|v\|_{H^{\ell
      +1}(\Omega_{\rm tr})}.  \eeq
\end{assumption}
Assumption \ref{ass:spaces} holds when $(V_h)_{h>0}$ consists of 
functions that are continuous in $\Omega_{\rm tr}$ and piecewise polynomials of total degree $p$ 
on a shape-regular family of simplicial triangulations of $\Omega_{\rm tr}$,
indexed by the meshwidth $h$; see, e.g.,
\cite[Theorem 17.1]{Ci:91}, \cite[Proposition 3.3.17]{BrSc:08}.

Let $\refs{a}_{\rm PML}(\rC;\cdot,\cdot)$ 
denote the sesquilinear form on the left-hand side of \eqref{eq:PMLvf}.  
The sequence of Galerkin solutions of the PML problem is defined by
\begin{equation}\label{eq:FEM}
  \boxed{\text{find} \quad\refs{u}_{\rm PML}(\rC;\cdot)_h \in V_h\,\, \text{ such that }\,\, 
  \refs{a}_{\rm PML}(\rC;\refs{u}_{\rm PML}(\rC;\cdot)_h ,\refs{v}_h)
  = 
  \refs{L}(\refs{v}_h) \quad\text{ for all } \refs{v}_h \in V_h\;.}
\end{equation}

Our result about $k$-explicit convergence of the $h$-FEM (Theorem \ref{thm:FEM1} below)
requires that, for
$\gamma\geq p-1$,
$\refs{\VA}_{\rm PML}(\rC;\cdot) \in C^{\gamma,1}$ and $\refs{n}_{\rm PML}(\rC;\refs{\Bx})\in C^{\gamma-1,1}$  with the corresponding norms bounded independently of $\rC$ (the reason
for this is explained in the discussion preceding Lemma \ref{lem:shift} and the proof of Theorem \ref{thm:FEM1}).  
For simplicity, \emph{we assume that $\boxed{p = \gamma +1}$}, i.e., given the regularity of
  $\refs{\VA}_{\rm PML}(\rC;\cdot)$  and $\refs{n}_{\rm PML}(\rC;\refs{\Bx})$, we take the minimal polynomial degree to obtain the
  best-possible FEM convergence result for that regularity.
\esc{Since the regularity of   $\refs{\VA}_{\rm PML}(\rC;\cdot)$  and $\refs{n}_{\rm PML}(\rC;\refs{\Bx})$ depend on the regularity of the radial displacement function (via \eqref{eq:coeffs}), this means that \emph{we are selecting the minimal polynomial degree to obtain the
  best-possible FEM convergence result for the given regularity of the transmission interface.}}

By the definitions of $\refs{\VA}_{\rm PML}(\rC;\cdot)$ 
in terms of
$\refs{\VA}(\rC;\cdot)$ \eqref{eq:At}, 
for $\refs{\VA}_{\rm PML}(\rC;\cdot) \in C^{p-1,1}$ we need
 $\rC$-uniform control of the $C^{p,1}$ norm of $r(\rC;\cdot)$. Similarly,
by the definition of $\refs{n}_{\rm PML}(\rC;\refs{\Bx})$ in terms of $\refs{n}(r;\refs{\Bx})$ \eqref{eq:nt}, 
for $\refs{n}_{\rm PML}(\rC;\refs{\Bx})\in C^{p-2,1}$ 
 we need  $\rC$-uniform control of the $C^{p-1,1}$ norm of $r(\rC;\cdot)$.
We therefore consider the following 
subsets of $\cA$ \eqref{eq:Anset}: for $p=1,2,3,...$, define
\footnote{For simplicity we make the constant on the right-hand side of the bound on the
  $C^{p,1}$ norms in \eqref{eq:Hh_gamma} $C_{\Re}$, but in principle this could be a
  different constant.
  }
  \begin{gather}
    \label{eq:Hh_gamma}
    \cA_p  = \cA_{p}(k;C_{\Re},C_{\Im}) 
    :=
    \big\{ \rC \in \cA(k;C_{\Re},C_{\Im}) \, :\,    
 \N{\Re \rC}_{C^{p, 1}(\bbS^{d-1})}\leq C_{\Re} \big\}, \; p\in \mathbb{N}.
\end{gather}
\begin{theorem}[$k$-explicit quasioptimality of $h$-FEM, uniform for $\rC \in \cA_{p}$]
\label{thm:FEM1}
Given $n_i<1$, $d=2$ or $3$, $C_{\Re}>0$, $p \in \mathbb{N}$,  
suppose that 
\begin{itemize}
\item both the PML scaling function
  $\widetilde{\sigma}$ and $\partial \Omega_{\rm tr}$ are $C^{p,1}$ regular, 
\item $k_1$ is as in Theorem \ref{thm:inherit},
\item $(V_h)_{h>0}$ is as in Assumption~\ref{ass:spaces}.
\end{itemize} 
Then there exists $C_{\rm FEM, j}, j=1,2,3, 4$ 
such that, 
for all $\rC \in \cA_p$ (as defined in \eqref{eq:Hh_gamma}), 
for all $\refs{L}\in (H^1_\wn (\Omega_{\rm tr}))^*$, and 
for all $k\geq k_1$, if $h$ is such that
\begin{gather}
  \label{eq:FEMthreshold}
{(hk)^{2p}} k \leq C_{\rm FEM, 1},
\end{gather}
then the solution $\refs{u}_{\rm PML}(\rC;\cdot)_h$ of \eqref{eq:FEM} exists, is unique,
and satisfies
\begin{gather}
  \label{eq:FEMresult1} \big\| \refs{u}_{\rm PML}(\rC;\cdot)
-\refs{u}_{\rm PML}(\rC;\cdot)_h \big\|_{H^1_k(\Omega_{\rm tr})} \leq C_{\rm FEM, 2}
{\Big(1 + (hk)^p k\Big)}
\min_{\refs{v}_h \in V_h} \big\| \refs{u}_{\rm PML}(\rC;\cdot) - \refs{v}_h
\big\|_{H^1_k(\Omega_{\rm tr})}
\end{gather}
and
\begin{gather}
  \label{eq:FEMresult2}
  \big\|
\refs{u}_{\rm PML}(\rC;\cdot) -\refs{u}_{\rm PML}(\rC;\cdot)_h \big\|_{L^2(\Omega_{\rm
    tr})} \leq C_{\rm FEM,3}{\Big( hk + (hk)^p k\Big)}
\min_{\refs{v}_h \in V_h} \big\| \refs{u}_{\rm PML}(\rC;\cdot) - \refs{v}_h
\big\|_{H^1_k(\Omega_{\rm tr})}.
 \end{gather} 
 Suppose, in addition, that 
 \begin{itemize}
 \item 
   $\refs{L}(\refs{v}) = \int_{\Omega_{\rm tr}}\refs{f}\overline{\refs{v}}\rd \refs{\Bx}$ with
   $\refs{f}\in {H^p}(\Omega_{\rm tr})$ that satisfies for some $C_{f}>0$
   \begin{gather}
     \label{eq:koscillatory}
     |\refs{f}|_{H^{\ell}(\Omega_{\rm tr})} \leq C_{f} k^\ell \|\refs{f}\|_{L^2(\Omega_{\rm tr})}
     \,\,
     \text{ for all } \,\,\ell=0,\ldots,p-1 \,\,\text{ and } \,\, k\geq k_1. 
   \end{gather}
\end{itemize}
Then there is a constant $C_{\rm FEM, 4}>0$ (depending on $C_{f}$)
such that, for all $\rC\in \cA_p$, 
\begin{gather}\label{eq:FEMresult3} 
\big\| \refs{u}_{\rm PML}(\rC;\cdot) -\refs{u}_{\rm PML}(\rC;\cdot)_h \big\|_{H^1_k(\Omega_{\rm tr})} 
\leq 
C_{\rm FEM, 4} \Big( 1 + (hk)^p k\Big)(hk)^p  
               \big\| \refs{u}_{\rm PML}(\rC;\cdot)\big\|_{H^1_k(\Omega_{\rm tr})}
\end{gather} 
and 
\begin{gather}\label{eq:FEMresult4} 
\big\| \refs{u}_{\rm PML}(\rC;\cdot) -\refs{u}_{\rm PML}(\rC;\cdot)_h \big\|_{L^2(\Omega_{\rm tr})} 
\leq
C_{\rm FEM, 4} \Big( hk + (hk)^p k\Big)(hk)^p  \big\| \refs{u}_{\rm PML}(\rC;\cdot)\big\|_{H^1_k(\Omega_{\rm tr})}.
\end{gather} 
\end{theorem}
    
\begin{remark}[Interpreting Theorem \ref{thm:FEM1} and the mesh threshold \eqref{eq:FEMthreshold}]  \label{rem:FEM1}
Theorem \ref{thm:FEM1} implies that
\begin{enumerate}
\item If $(hk)^p k$ is sufficiently small, then the Galerkin solutions are quasioptimal,
  with constant in independent of $k$, as $k\to\infty$ (by \eqref{eq:FEMresult1}).
\item If $(hk)^{2p}k$ is sufficiently small and the data is ``$k$-oscillatory'' in the
  sense of \eqref{eq:koscillatory}, then the relative $H^1_k$ error is controllably small,
  uniformly in $k$, as $k\to \infty$ (by \eqref{eq:FEMresult3}).
\end{enumerate}
These thresholds are observed empirically to be sharp (going back to the work of Ihlenburg
and Babu\v{s}ka in 1-d \cite{IhBa:95a, IhBa:97}); see the discussion in \cite[\S1.3]{GS3}
and the references therein.

In the limit $h\to 0$ for fixed $k$, the bounds \eqref{eq:FEMresult3} and \eqref{eq:FEMresult4} recover the well-known
results that, for $f\in H^{p-1}$, the Galerkin error in $H^1$ is $O(h^p)$
\eqref{eq:FEMresult3}, and the Galerkin error in $L^2$ is $O(h^{p+1})$
\eqref{eq:FEMresult4}.
\end{remark} 
  
\esc{The result of Theorem \ref{thm:FEM1} with the constants $C_{\rm FEM, j}, j=1,2,3, 4$, depending on $\rC \in \mathcal{A}$ is a direct consequence of the main result of \cite{GS3}. The assumptions in \cite{GS3} are that (i) the sesquilinear form is continuous and satisfies a G\aa rding inequality, and (ii) the highest-order terms in the underlying PDE satisfy the natural elliptic-regularity shift, and the constants in the FEM-error bounds in \cite{GS3} then only depend on the constants in (i) and (ii). 
To prove Theorem \ref{thm:FEM1} therefore, we need to show that the properties (i) and (ii) hold with constants uniform in $\rC\in \mathcal{A}_p$. For (i), this property follows immediately from Corollary \ref{cor:PMLGarding}. We therefore now focus on (ii). 
}
  
\begin{lemma}[Elliptic regularity shift property, uniform for $\rC \in \cA_p$]
\label{lem:shift}
  Suppose that $\widetilde{\sigma}$ and $\partial \Omega_{\rm tr}$ are both $C^{p,1}$, 
  $p\in \mathbb{N}$.
  Given $d=2$ or $3$ and $C_{\Re}>0$, let $C_{\Im}$ be as in Corollary \ref{cor:A}.  Then
  there exists $C>0$ such that for all
  $\rC \in \cA_p$ and $j=1,\ldots,p+1$, 
given $\refs{f}\in H^{j-2}(\Omega_{\rm tr})$,
    the solution of
  \beqs - \refs{\nablabf} \cdot\big(\refs{\VA}_{\rm
    PML}(\rC;\refs{\Bx}) \refs{\nablabf}\refs{w}(\rC;\refs{\Bx})\big) =
  \refs{f}(\refs{\Bx}) \quad \text{ for }\refs{\Bx} \in \Omega_{\rm tr} \quad\text{ and
  }\quad \refs{w}(\refs{\Bx}) = 0 \quad\text{ for } \refs{\Bx}\in\partial\Omega_{\rm tr}
  \eeqs
  satisfies
  \beq
\label{eq:shift0}
| \refs{w} |_{H^j(\Omega_{\rm tr})} \leq C \big\|\refs{f}\big\|_{H^{j-2}(\Omega_{\rm tr})}.
\eeq
Furthermore, the analogous bound hold when  $\refs{\VA}_{\rm
    PML}$ is replaced by either $\overline{(\refs{\VA}_{\rm
    PML})^T}$ or $\Re (\refs{\VA}_{\rm
    PML}):= (\refs{\VA}_{\rm
    PML}+\overline{(\refs{\VA}_{\rm
    PML})^T})/2$. 
\end{lemma}

\bpf We first claim that it is sufficient to prove that 
\beq
\label{eq:shift1}
| \refs{w}
|_{H^{j}(\Omega_{\rm tr})} \leq C' \big(\N{\refs{w}}_{H^1(\Omega_{\rm tr})} +
\big\|\refs{f}\big\|_{H^{j-2}(\Omega_{\rm tr})}\big)
\eeq 
for some $C'$ independent of
$\rC$. Indeed, Lemma \ref{lem:signPML}, Part (i) of Lemma \ref{lem:continuity} and the
Lax--Milgram theorem imply that
$\|\refs{w}\|_{H^1(\Omega_{\rm tr})} \leq C'' \|\refs{f}\|_{L^2(\Omega_{\rm tr})}$ with $C''$
independent of $\rC$, and combining this with \eqref{eq:shift1} yields \eqref{eq:shift0}.

By \esc{the standard elliptic-regularity shift result (see}, e.g., \cite[Theorem 4.18]{Mc:00}) the bound \eqref{eq:shift1} holds if 
(i) \eqref{eq:signPML} holds,
(ii) $\partial \Omega_{\rm tr}$ is $C^{p,1}$, and 
  (iii) $\refs{\VA}_{\rm PML}\in C^{p-1,1}(\Omega_{\rm tr})$ with norm bounded independently of $\rC$.
The property (i) holds by Lemma \ref{lem:signPML} and the property (ii) holds by assumption.
For the property (iii), the fact that $\widetilde{\sigma}\in C^{p,1}$
 implies that $\refs{\VA}_{\rm PML}$ 
(defined by \eqref{eq:defAPML}) restricted to the PML region $|\refs{\Bx}|\geq R_1$ is $C^{p-1,1}$. 
Arguing as in Lemma \ref{lem:continuity}, we have that if $\rC\in \mathcal{A}_p$, then 
$\refs{\VA}_{\rm PML}$ restricted to $|\refs{\Bx}|\leq R_1$ is in $C^{p,1}$ with norm bounded independently of $\rC$.
\epf  
  
  \esc{With Corollary \ref{cor:PMLGarding} and Lemma \ref{lem:shift} in hand, we can now prove Theorem \ref{thm:FEM1}.}  
  
\bpf[Proof of Theorem \ref{thm:FEM1}] 
We obtain the bounds \eqref{eq:FEMresult1},
\eqref{eq:FEMresult2}, and \eqref{eq:FEMresult3} under the mesh threshold
\eqref{eq:FEMthreshold} by showing that the PML problem \eqref{eq:PMLvf} and its Galerkin
discretisation \eqref{eq:FEM} fit into the class of problems considered by \cite{GS3}.
Analytic dependence of the Galerkin solution on $\rC$ then follows by noting that, by
\eqref{eq:FEMresult1} and the triangle inequality, $\refs{u}_{\rm PML}(\rC;\cdot)_h$
satisfies the same bound in terms of the data as $\refs{u}_{\rm PML}(\rC;\cdot)$ (i.e.,
the bound \eqref{eq:inherit1}) and then arguing as in the proof of Corollary \ref{cor:A}
using the analytic implicit function theorem.

The paper \cite{GS3} proves the bounds \eqref{eq:FEMresult1}-\eqref{eq:FEMresult4} 
for Helmholtz problems whose sesquilinear forms are continuous, 
satisfy a G\aa rding inequality, and satisfy elliptic regularity. 
We now show how Corollary \ref{cor:PMLGarding} and Lemma \ref{lem:shift} imply that
the variational problem \eqref{eq:PMLvf} fits into this framework, 
with constants independent of $\rC$ for $\rC \in \cA_{p}$.

Corollary \ref{cor:PMLGarding} immediately shows that \cite[Equations 1.6 and 1.7]{GS3} are satisfied (i.e., the sesquilinear form is continuous and satisfies a G\aa rding inequality).
For the elliptic regularity assumptions of \cite[Assumptions 1.2 and 1.6]{GS3}, we need that there exists $C>0$ such that, for $j=1,\ldots,p+1$ and $\rC \in \cA_{p}$,
\beq\label{eq:needSunday1}
\N{\refs{w}}_{H^j_k(\Omega_{\rm tr})} \leq C \Big( \N{\refs{w}}_{L^2} + \big\| k^{-2} \nabla \cdot (\refs{\VA}_{\rm
    PML}
 \nabla \refs{w} ) + \refs{n}_{\rm PML} \refs{w} \big\|_{H^{j-2}_k(\Omega_{\rm tr})} \Big),
\eeq
as well as the analogous bound with $\refs{\VA}_{\rm
    PML}$ replaced by either $\overline{(\refs{\VA}_{\rm
    PML})^T}$ or $\Re (\refs{\VA}_{\rm
    PML}):= (\refs{\VA}_{\rm
    PML}+\overline{(\refs{\VA}_{\rm
    PML})^T})/2$. In \eqref{eq:needSunday1} $H^j_k$ is the usual $H^j$ norm, but with each derivative weighted by $k^{-1}$, as in \eqref{eq:weighted_norms}, \eqref{eq:weighted_H2}.
By multiplying \eqref{eq:shift0} by $k^{-j}$, we obtain that there exists $C>0$ such that, for $j=1,\ldots,p+1$ and $\rC \in \cA_{p}$,
\beqs
\N{\refs{w}}_{H^j_k(\Omega_{\rm tr})} \leq C \big\| k^{-2} \nabla \cdot (\refs{\VA}_{\rm
    PML} \nabla \refs{w} )  \big\|_{H^{j-2}_k(\Omega_{\rm tr})}.
\eeqs
Now, since $\refs{n}_{\rm PML}  \in C^{p-1}$ with norm bounded independently of $\rC \in \cA_{p}$, 
by \esc{a classic result about the Sobolev norm of a product (see}, e.g., \cite[Theorem 1.4.1.1, page 21]{Gr:85}) 
there exists $C>0$ such that 
\beq\label{eq:nboundmult}
\|  \refs{n}_{\rm PML}\refs{w}\|_{H^{j-2}_k(\Omega_{\rm tr})} \leq C \|\refs{w}\|_{H^{j-2}_k(\Omega_{\rm tr})}  \quad\text{ for } j=1,\ldots,p+1\, \text{ and } \, \rC \in \cA_{p}.
\eeq
 Therefore, for $j=1,\ldots,p+1$ and $\rC \in \cA_{p}$,
\beqs
\N{\refs{w}}_{H^j_k(\Omega_{\rm tr})} \leq C \Big(  \N{\refs{w}}_{H^{j-2}_k(\Omega_{\rm tr})}+\big\| k^{-2} \nabla \cdot (\refs{\VA}_{\rm
    PML} \nabla \refs{w} ) + \refs{n}_{\rm PML}  \refs{w}  \big\|_{H^{j-2}_k(\Omega_{\rm tr})}  \Big),
\eeqs
and the required bound \eqref{eq:needSunday1} follows by induction. 

Assumption \ref{ass:spaces} is \cite[Assumption 4.8]{GS3}, and then \eqref{eq:FEMresult1}-\eqref{eq:FEMresult3} follow from \cite[Equations 4.16, 4.17, and 4.19]{GS3}. The bound \eqref{eq:FEMresult4} is not stated explicitly in \cite{GS3}, but follows from the displayed equation before \cite[Remark 2.3]{GS3} (one repeats the arguments that obtain \eqref{eq:FEMresult3} from \eqref{eq:FEMresult1}, but now one starts from \eqref{eq:FEMresult2}).
\epf  

\section{Quantity of Interest (QoI): Far-Field Pattern}
\label{sec:faf}

\subsection{Definition of and expressions for the far-field pattern}

If $v$ is a solution of the Helmholtz equation $(-k^{-2}\Delta -1)v=0$ outside $B_{R_0}$
for some $R_0>0$ and $v$ satisfies the Sommerfeld radiation condition \eqref{eq:src}
(i.e., $v$ is \emph{outgoing}), then the \emph{far-field pattern} of $v$,
$v_\infty:\mathbb{S}^{d-1}\to \mathbb{C}$, is defined by
\beq\label{eq:ff_def}
v_\infty (\unitx) = \lim_{ \rho \to \infty} \Big(
\rho^{(d-1)/2} \exp(- \ri k \rho)\, v(\rho\unitx) \Big)\;,\quad \unitx \in \bbS^{d-1}
\eeq
(this limit exists and is a smooth function of $\unitx$ by, e.g., \cite[Corollary 3.7]{CoKr:83}).

The quantity of interest for our scattering problem \eqref{eq:htp} is the far field
pattern $u_{\infty}:\bbS^{d-1}\to\bbC$ of the scattered wave $\uscat:=u-\uinc$.

We now give an expression for the far-field pattern of an outgoing Helmholtz solution from 
\cite[Theorem 2.2]{Mo:95} (written in a slightly-more general way here).

\begin{theorem}[Expression for far-field pattern as integral over subset of domain]\label{thm:Monk}
  Suppose that, for some $R_0>0$, $v\in C^2(\Rea^d\setminus B_{R_0})$ satisfies the
  Helmholtz equation $(-k^{-2}\Delta - 1)v=0$ in $\Rea^d\setminus \overline{B_{R_0}}$ and
  the Sommerfeld radiation condition \eqref{eq:src}.  
  Let $\psi \in C^2(\Rea^d;[0,1])$ be such that $\psi \equiv 0$ 
  in a neighbourhood of $B_{R_0}$ and $\psi \equiv 1$ on
  $(B_{R_1})^c$ for some $R_1>R_0$.  
  
  Then 
  \beq\label{eq:Monk}
  \boxed{v_\infty(\unitx) = C(d,k)
  \int_{\supp \nabla\psi} v(\By) \Big( \Delta \psi(\By) - 2 \ri k\,
  \unitx\cdot\nabla\psi(\By) \Big) \exp\big(-\ri k \By \cdot \unitx\big) \, {\rm
    d} \By}\;, 
\eeq 
where \cs{$\unitx \in \bbS^{d-1}$ and}
\beq \label{eq:Cdk} 
      C(2,k) =\frac{1}{\sqrt{k} } \frac{\re^{\ri \pi/4}}{2\sqrt{2\pi}} \quad\text{ and }\quad C(3,k) = \frac{1}{4\pi}.  \eeq
\end{theorem}

Since the proof is relatively short, we give it here.

\bpf[Proof of Theorem \ref{thm:Monk}]
Let 
\beqs
G(\Bx,\By):= \frac{\ri}{4} H_0^{(1)}(k |\Bx-\By|) \,\,\text{ for } d=2, \quad \text{ and
}\quad G(\Bx,\By) := \frac{\exp(\ri k |\Bx-\By|)}{4\pi |\Bx-\By|} \,\,\text{ for } d=3,
\eeqs
and recall that $(\Delta_{\By} + k^2) G(\Bx,\By) = - \delta (\Bx-\By)$.
Green's integral representation states that if $w\in C^2(\overline{B_R})$ then, for $\Bx\in B_R$,
\beqs
w(\Bx) = \int_{\partial B_R}\left( \frac{\partial w}{\partial \nu}(\By) G(\Bx,\By) - w(\By) \frac{\partial G(\Bx,\By)}{\partial \nu(\By)}\right) \rd S(\By) - \int_{B_R} G(\Bx,\By) (\Delta + k^2)w(\By) \rd \By,
\eeqs
where $\nu(\By):= \By/|\By|$; see, e.g., \cite[Theorem 3.1]{CoKr:83}. Let $w= v\psi$ and observe that  $w\in C^2(\overline{B_R})$
since $\psi\equiv 0$ on a neighbourhood of $B_{R_0}$.
With this choice of $w$, as $R\to \infty$, the integral over $\partial B_R$ tends to zero, since both $v$ and $G(\Bx,\cdot)$ (for $\Bx$ fixed) satisfy the Sommerfeld radiation condition (see, e.g., \cite[Last equation in the proof of Theorem 3.3]{CoKr:83}). Now,
\beqs
(\Delta +k^2) w = (\Delta +k^2)(\psi v) = 2 \nabla\psi \cdot\nabla v + v \Delta \psi + \psi (\Delta +k^2)v 
= 2 \nabla\psi \cdot\nabla v + v \Delta \psi,
\eeqs
since $(\Delta +k^2)v=0$ on $\supp \psi$;
therefore, if $|\Bx|> R_1$, then
\beqs
v(\Bx) = -\int_{\supp \nabla\psi} 
 G(\Bx,\By) \Big(2 \nabla\psi \cdot\nabla v + v \Delta \psi\Big)(\By)
\rd \By.
\eeqs
Letting $|\Bx|\to \infty$ and using \eqref{eq:ff_def} and (when $d=2$) the large-argument asymptotics of $H_0^{(1)}(\cdot)$ (see, e.g., \cite[Equation 10.17.5]{Di:22}), we find that
\beqs
v_\infty(\unitx) = -C(d,k) \int_{\supp \nabla\psi}
\Big(2 \nabla\psi \cdot\nabla v + v \Delta \psi\Big)(\By)
\exp\big(-\ri k \By \cdot \unitx\big)
\, {\rm d} \By.
\eeqs
The result \eqref{eq:Monk} then follows by integrating by parts (using the divergence theorem) the term involving $\nabla\psi \cdot\nabla v$ (moving the derivative from $v$ onto $\psi$).
\epf

\subsection{Formulating the solution of the plane-wave scattering problem as the outgoing solution of a Helmholtz problem with $L^2$ data}
\label{sec:planewave}

We now recall how to formulate the solution of the plane-wave scattering problem of
\eqref{eq:htp} as the outgoing solution of a Helmholtz problem with source
$f\in L^2(B_2)$, so that it can be approximated by PML truncation (with the error then given
by Theorem \ref{thm:GLS}).

One option is to solve for the scattering field $\uscat:=u- \uinc$, which satisfies the
Sommerfeld radiation condition \eqref{eq:src} (by \eqref{eq:htp2}) and
\beq\label{eq:rhs_compare1} (-k^{-2}\Delta - n)\uscat = - (1-n)\uinc; \eeq since the
right-hand side of this PDE is compactly supported in $B_2$, PML truncation can be used to
approximate $\uscat$ (with the error then controlled by Theorem \ref{thm:GLS}).

A second option is described in the following lemma.  Although this second option is more
complicated than the first, the second option has the advantage that, when $\uinc$ is an incident
plane wave, the $L^{2}$-norm of the right-hand side of the PDE behaves like $O(k^{-1})$ for
$k\to\infty$ (see \eqref{eq:pw_f} below), whereas the right-hand side of the PDE
\eqref{eq:rhs_compare1} is uniformly bounded with respect to $k$.  
Recall 
 that 
slower growth in $k$ of the right-hand side implies slower growth of the solution (by \eqref{eq:inherit1}) and thus 
stronger bounds on the finite-element solution (by \eqref{eq:FEMresult3}).

\begin{lemma}[Transmission solution formulated as an outgoing Helmholtz solution with $L^2$ data]\label{lem:L2}
Given $\uinc$, 
let $u$ be the solution of the Helmholtz transmission problem \eqref{eq:htp} 
with $n$ given by \eqref{eq:n} and $D$ as described in \S\ref{sec:pi}.
Given $\eta>0$, let $\varphi \in C^{2}_{\rm comp}(\Rea^d;[0,1])$ be such that 
\beq\label{eq:psi_pw}
\varphi \equiv 1\,\,\text{ on } B_{2-\eta} \quad\text{ and } \quad \varphi \equiv 0\,\,\text{ on }\,\,\Rea^d \setminus B_{2-\eta/2}.
\eeq
Let $\ualt$ be the outgoing solution to 
\beq\label{eq:pw_f}
(-k^{-2}\Delta - n) \ualt
=
-k^{-2}\big(2 \nabla \varphi\cdot \nabla \uinc + \uinc \Delta \varphi\big)=: f^{\rm alt}.
\eeq
Then 
\beq\label{eq_tilde_u}
\ualt =\varphi \uinc  + (u-\uinc) 
= u- (1-\varphi)\uinc,
\eeq
and thus  $\ualt\equiv u$ on $B_{2-\eta}$ and $\ualt = \uscat$ on
  $\mathbb{R}^{d}\setminus B_{2-\eta/2}$.
\end{lemma}

\bpf
Given $\varphi$ as above, the function $u- (1-\varphi)\uinc$ satisfies both the Sommerfeld radiation condition \eqref{eq:src} and 
\begin{align}\nonumber
(-k^{-2}\Delta - n)\big(u- (1-\varphi)\uinc\big)
= -(-k^{-2}\Delta -n) (1-\varphi) \uinc&=-(-k^{-2}\Delta-1) (1-\varphi) \uinc\\
&= -k^{-2}\big(2 \nabla \varphi\cdot \nabla \uinc + \uinc \Delta \varphi\big),
\nonumber
\end{align}
where we have used that $n\equiv 1$ on $\supp(1-\varphi)$ in the second equality. By the uniqueness of outgoing solutions of the transmission problem, \eqref{eq_tilde_u} holds.
\epf



We now combine Theorem \ref{thm:Monk} and Lemma \ref{lem:L2} 
to give an expression for the far-field pattern of the scattered wave in the transmission problem \eqref{eq:htp} in terms of $\ualt$.

\begin{corollary}[Far-field pattern in terms of $\ualt$ in the nominal domain]\label{cor:Monk}
Let $\lambda>0$ be as in \S\ref{sec:param}, \eqref{chi}, and choose $\eta>0$ such that $\eta<\lambda$. 
Let further
$\psi\in C^2(\mathbb{R}^d;[0,1])$ be such that $\psi\equiv 0$ in a neighbourhood of $B_{2-\lambda}$ 
and $\psi\equiv1$ in a neighbourhood of $\mathbb{R}^d\setminus B_{2-\eta}$. 
Let $\varphi \in C^2(\mathbb{R}^d;[0,1])$ be as in \eqref{eq:psi_pw}, and 
let $\ualt$ be the outgoing solution of \eqref{eq:pw_f}. Then 
\begin{align}\nonumber
\refs{u}^{\rm scat}_\infty(\rC;\unitxh)  
  &= C(d,k) \int\limits_{\supp \nabla \psi}
  \Big(\refs{u}^{\rm alt}(\rC;\refs{\By}) - \refs{u}^{\rm inc}(\refs{\By})\Big)
  \Big( \Delta \psi(\refs{\By}) - 2 \ri k\,
  \unitxh\cdot\nabla\psi(\refs{\By}) \Big) \cdot \\
&  \hspace{8cm}
  \exp\big(-\ri k \refs{\By} \cdot \unitxh\big) \, {\rm d} \refs{\By}\;,
 \quad \unitxh\in\bbS^{d-1}\;.
\label{eq:faft}
\end{align}
Furthermore, given $C_{\Re}>2$, 
let {$\cA = \cA(k;C_{\Re},C_1)$ for $k > \frac12$ 
be as in \eqref{eq:Anset} and Corollary \ref{cor:A}. 
}

Then the map $\Ca\to \Ltwo[\bbS^{d-1}]:\rC\mapsto \refs{u}^{\rm scat}_{\infty}$ is
\emph{holomorphic}, 
and there exist $C', C''>0$ such that, for all $k\geq \frac12$,
\begin{gather}
  \label{eq:ffbd}
  \sup\limits_{\rC\in\Ca}
  \N{\refs{u}^{\rm scat}_{\infty}(\rC;\cdot)}_{L^\infty(\bbS^{d-1})} 
  \leq  
  C' C(d,k)  k \N{ \refs{u}^{\rm inc}}_{H^1_k(B_2)} \leq C'' C(d,k) k.
\end{gather}
The constant $C(d,k)$ is as in \eqref{eq:Cdk}; 
    in particular, as $k\to\infty$, $C(d,k) = O(k^{(d-3)/2})$.
\end{corollary}

In Corollary \ref{cor:Monk} we choose $\psi$ such that $\refs{\Bx}\mapsto
\Phibf(r;\refs{\Bx})$ is the identity on $\supp \psi$ (by \eqref{chi}); 
this choice is not necessary, i.e., an analogous expression to \eqref{eq:faft} holds 
if $\refs{\Bx}\mapsto
\Phibf(r;\refs{\Bx})$ is not the identity on $\supp \psi$, 
but this analogous expression is slightly more complicated than \eqref{eq:faft}.

\bpf[Proof of Corollary \ref{cor:Monk}]
By Theorem \ref{thm:Monk}, $\uscat_\infty$ is given by the right-hand side of \eqref{eq:Monk} with $v$ replaced by $u-\uinc$.
The definition of $\psi$ implies that $\supp \nabla\psi \subset B_{2-\eta}$, and Lemma \ref{lem:L2} implies that $u\equiv \ualt$ on $B_{2-\eta}$.
Therefore, $\uscat_\infty$ is given by the right-hand side of \eqref{eq:Monk} with $v$ replaced by $\ualt-\uinc$.

We now map the integral in \eqref{eq:Monk} back to the nominal domain using $\refs{\Bx}\mapsto
\Phibf(r;\refs{\Bx})$; by \eqref{chi}. This transformation is the identity on 
$\Rea^d\setminus B_{2-\lambda}$, and thus on $\supp\nabla \psi$, and the expression \eqref{eq:faft} follows.

By Corollary \ref{cor:A}, 
the map $\rC\mapsto \refs{u}^{\rm alt}(\rC;\cdot)$ is holomorphic for $\rC\in \mathcal{A}$.
Since the mapping $\refs{u}^{\rm alt} \mapsto
\refs{u}_{\infty}^{\rm scat}$ is linear and 
all other terms entering the integrand in \eqref{eq:faft} are independent of $\rC$, 
the map $\Ca\to \Ltwo[\bbS^{d-1}]$, $\rC\mapsto \refs{u}_{\infty}$, is holomorphic. 
Finally, by \eqref{eq:holo_thm2}, 
the definition \eqref{eq:weighted_norms} of $\|\cdot\|_{(H_k^1(B_2))^*}$, 
and by \eqref{eq:pw_f}, 
for $k\geq k_0$ it holds
\beq\label{eq:falt}
\sup_{\rC\in \mathcal{A}}\big\|\refs{u}^{\rm alt}(\rC;\cdot)\big\|_{H^1_k(B_2)} 
\leq 
C_2 k \big\| \refs{f}^{\rm alt} \big\|_{L^2(B_2)} 
\leq 
C' \big\| \refs{u}^{\rm inc} \big\|_{H^1_k(B_2)}.
\eeq
The bound \eqref{eq:ffbd} then follows 
by combining \eqref{eq:falt} and \eqref{eq:faft}.
\epf

\subsection{Accuracy of the PML $h$-Galerkin FEM approximation of the far-field pattern}    
\label{sec:Accff}
We now use the results of \S\ref{sec:ExpBdsFEM} 
to bound the error in the far-field pattern when PML truncation and the $h$-FEM are 
used to compute an approximation to $\refs{u}^{\rm alt}$, 
and this approximation used in the expression \eqref{eq:faft}.

As above, 
let $\refs{u}^{\rm alt}(\rC;\cdot)$ 
be the outgoing solution of \eqref{eq:pw_f} transformed to the nominal domain. 
Let $\refs{u}^{\rm alt}_{\rm PML}(\rC;\cdot)$ 
be the solution of the PML variational formulation \eqref{eq:PMLvf} 
with the corresponding right-hand side. 
Let $\refs{u}^{\rm alt}_{\rm PML}(\rC;\cdot)_h$ 
be the corresponding solution of the Galerkin equations \eqref{eq:FEM}, 
and let, \cs{for $\unitx \in \bbS^{d-1}$,} 
\begin{gather}
  \boxed{
\refs{u}^{\rm scat}_{\infty, \rm PML}(\rC;\unitxh)_h  
:= 
C(d,k) \int\limits_{\supp \nabla \psi} 
\begin{aligned}[t]
  \Big(\refs{u}^{\rm alt}_{\rm PML}(\rC;\refs{\By})_h - \refs{u}^{\rm
    inc}(\refs{\By})\Big) \Big( \Delta \psi(\refs{\By}) - 2 \ri k\,
  \unitxh\cdot\nabla\psi(\refs{\By}) \Big)\cdot \\  \exp\big(-\ri k \refs{\By}
  \cdot \unitxh\big) \, {\rm d} \refs{\By}
\end{aligned}}
\label{eq:faft2}
\end{gather}
(i.e., \eqref{eq:faft} with $\refs{u}^{\rm alt}(\rC;\refs{\By})$ 
 replaced by $\refs{u}^{\rm alt}_{\rm PML}(\rC;\refs{\By})_h$) 
denote the corresponding approximate far-field.

\begin{corollary}[Error in the computed far-field pattern]
\label{cor:FEMfar-field}
  Under the assumptions of Theorem~\ref{thm:FEM1} 
  and with $k_{1}$ as in Theorem~\ref{thm:inherit},
  there exists $C_{\rm FF,1}$ 
  such that 
  for all $k\geq k_1$, 
  for all $\rC \in \cA_p$, and 
  for all $h,k$, and $p$ satisfying
  \eqref{eq:FEMthreshold}, 
\begin{align}\nonumber
&\N{\refs{u}^{\rm scat}_\infty(\rC;\cdot)
-
\refs{u}^{\rm scat}_{\infty, \rm PML}(\rC;\cdot)_h }_{L^\infty(\mathbb{S}^{d-1})} 
\\
&\qquad
\leq  \,C_{\rm FF,1}  C(d,k) 
\Big[
 \Big( hk + (hk)^p k\Big)(hk)^p k
 +
  \exp \Big( - C_{\rm PML, 2} k\big(R_{\rm tr}-(1+\epsilon) R_1\big)\Big)\Big] 
\N{ \refs{u}^{\rm inc}}_{H^1_k(B_2)}, 
  \label{eq:FEMfar-field}
\end{align}
where $C(d,k)$ is given for $d=2,3$ by \eqref{eq:Cdk}.
\end{corollary}  

Since $\| \refs{u}^{\rm inc}\|_{H^1_k(B_2)}\leq C'$, the bound \eqref{eq:FEMfar-field}
shows that \emph{the error in the far-field pattern \eqref{eq:FEMfar-field} is controlled
  uniformly in $k$ provided that $(hk)^p k$ is sufficiently small.}  Since
$C(2,k) = O(k^{-1/2})$, when $d=2$ the error in the far-field pattern decreases with
increasing $k$ subject to this mesh threshold.

\bpf[Proof of Corollary \ref{cor:FEMfar-field}]
By \eqref{eq:faft} and \eqref{eq:faft2},
\begin{align}\nonumber
& \refs{u}^{\rm scat}_{\infty} (\rC;\unitxh)-\refs{u}^{\rm scat}_{\infty, \rm PML}(\rC;\unitxh)_h 
\\
&= C(d,k) \int_{\supp \nabla \psi} 
\Big( 
\refs{u}^{\rm alt}(\rC;\refs{\By})
-\refs{u}^{\rm alt}_{\rm PML}(\rC;\refs{\By})_h 
\Big)
\Big( \Delta \psi(\refs{\By}) - 2 \ri k\,\unitxh\cdot\nabla\psi(\refs{\By}) \Big) 
\exp\big(-\ri k \refs{\By} \cdot \unitxh\big) \, {\rm d} \refs{\By}.
\nonumber
\end{align}
By the Cauchy--Schwarz inequality and the triangle inequality, for some $C'>0$, 
\begin{multline}
\N{  \refs{u}^{\rm scat}_\infty(\rC;\cdot)
-\refs{u}^{\rm scat}_{\infty, \rm PML}(\rC;\cdot)_h
}_{L^\infty(\mathbb{S}^{d-1})} 
\leq  k C(d,k) C' \big\|  \refs{u}^{\rm alt}(\rC;\cdot)-\refs{u}_{\rm PML}^{\rm alt}(\rC;\cdot)_h\big\|_{L^2(B_2)}\\
\leq   k C(d,k) C'\Big[
\big\|\refs{u}^{\rm alt}(\rC;\cdot)-\refs{u}^{\rm alt}_{\rm PML}(\rC;\cdot) \big\|_{L^2(B_2)}
+
 \big\|  \refs{u}^{\rm alt}_{\rm PML}(\rC;\cdot)-\refs{u}^{\rm alt}_{\rm PML}(\rC;\cdot)_h\big\|_{L^2(B_2)} 
 \Big]\;.
 \label{eq:temp_final1}
\end{multline}
To bound the second term in parentheses on the right-hand side of \eqref{eq:temp_final1}, 
we use the 
FE error bound \eqref{eq:FEMresult4} and then the bound \eqref{eq:inherit1} 
applied to  $\|\refs{u}^{\rm alt}_{\rm PML}(\rC;\cdot)\|_{H^1(\Omega_{\rm tr})}$ 
to obtain that
\begin{align}\nonumber
\big\|  \refs{u}^{\rm alt}_{\rm PML}(\rC;\cdot)-\refs{u}^{\rm alt}_{\rm PML}(\rC;\cdot)_h\big\|_{L^2(B_2)} 
&\leq \big\|  \refs{u}^{\rm alt}_{\rm PML}(\rC;\cdot)
-\refs{u}^{\rm alt}_{\rm PML}(\rC;\cdot)_h
\big\|_{L^2(\Omega_{\rm tr})} \\
&\leq 
C_{\rm FEM, 4} \Big( hk + (hk)^p k\Big)(hk)^p  \big\| \refs{u}_{\rm
  PML}(\rC;\cdot)\big\|_{H^1_k(\Omega_{\rm tr})}
  \nonumber
  \\
&  \leq 
C_{\rm FEM, 4}  
 \Big( hk + (hk)^p k\Big)(hk)^p 
  C k \big\|\refs{f}^{\rm alt}\big\|_{L^2(B_2)},\label{eq:temp_final2}
\end{align}
where we have used that $\refs{f}^{\rm alt}$ defined by \eqref{eq:pw_f} 
satisfies the $k$-oscillatory property in 
Theorem \ref{thm:FEM1} (to apply \eqref{eq:FEMresult4}) 
and is supported in $B_2$ (so that the final norm in \eqref{eq:temp_final2} is over $B_2$).
To bound the first term in the parentheses 
on the right-hand side of \eqref{eq:temp_final1},
we use the PML error bound \eqref{eq:GLS}. 
Combining this with \eqref{eq:temp_final1} and \eqref{eq:temp_final2}, 
we obtain that, for some $C''>0$, 
\begin{align}\nonumber
&\N{  \refs{u}^{\rm scat}_\infty(\rC;\cdot)-\refs{u}^{\rm scat}_{\infty, \rm PML}(\rC;\cdot)_h}_{L^\infty(\mathbb{S}^{d-1})} 
\\
&\qquad\leq   k C(d,k) C''\Big[
\Big(hk+(hk)^pk\Big)(hk)^p k
+
 \exp \Big( - C_{\rm PML, 2} k\big(R_{\rm tr}-(1+\epsilon) R_1\big)
\Big) 
\Big]\| \refs{f}^{\rm alt} \|_{L^2(B_2)} 
.\nonumber
\end{align}
The result \eqref{eq:FEMfar-field} then follows by recalling from \eqref{eq:falt} that 
$\| \refs{f}^{\rm alt} \|_{L^2(B_2)} \leq C k^{-1}\| \refs{u}^{\rm inc}\|_{H^1_k(B_2)}$. 
\epf

\begin{remark}
  \label{rem:esme}
  The following two steps in the proof of Corollary \ref{cor:FEMfar-field} might appear
  over-simplistic:
  \begin{enumerate}
  \item using bounds involving the $L^2(B_2)$ norm of the data $f^{\rm alt}$, while for
    plane-wave scattering $f^{\rm alt} \in H^s(B_2)$ for all $s>0$ (see \eqref{eq:pw_f}),
    and
  \item estimating the integral in the expression \eqref{eq:faft} for the far-field
    pattern using the Cauchy--Schwarz inequality, instead of using, say, a duality
    argument.
  \end{enumerate}
  \esc{In \S\ref{sec:appendix1} we describe how, given the current state-of-the-art FEM convergence theory, we cannot do better than arguing as in Points 1 and 2 above.}
\end{remark}

\section{$k$- and $h$-Explicit, Parametric Holomorphy of $\rhffph$}
\label{sec:ParHolIntQuad}

We now study the parametric holomorphy of the PML-Galerkin approximated far-field pattern
$\rhffph$ defined in \eqref{eq:faft2}.

To this end, we \rev{further constrain the generic shapes that were introduced}
in Section~\ref{sec:pi}, with displacement functions $r\in\Cr$ satisfying Assumption~\ref{ass:rbd}.
\rev{We adopt a \emph{affine-parametric} representation of the scatterers' shape.}
\rev{Probabilistic models of} shape uncertainty will subseuently 
    be introduced by placing a (probability) measure on 
\cs{the affine parameter sequences which occur in the radial displacement functions, yielding}
$r=r(\omega)$, $\omega\in \Omega$, 
with $\Omega$ the set of elementary events in a probability space.  
\cs{Specifically,} 
we assume affine-parametric dependence of $r$ 
on a sequence $\Vy:=\left(y_{j}\right)_{j \in \bbN}$ of parameters.  
As is customary in computational UQ for PDEs (see, e.g., \cite{SCG11,KUS24}),
the measure will be constructed as product of probability measures on the co-ordinates $y_j$.  
This is to say that the dependence of $r$ on $\omega$ is expressed through a
sequence $\boldsymbol{Y} = \left( Y_{j} \right)_{j\in\bbN}$ of independent identically
distributed (i.i.d), $(-1,1)$-valued random variables $Y_{j}=Y_{j}(\omega)$.  
 
We thus introduce an \emph{affine uncertainty-parametrization of the scatterer
geometry with globally supported basis functions $r_j(\Bs)$, $j\in \bbN$, with
$\Bs \in \bbS^{d-1}$.  We focus on space dimensions $d=2,3$}.  Doing so leads to a
countably-parametric description of the ensemble of admissible shapes, with parameters
being the expansion coefficients.  

Specifically, we adopt here the Karhunen-Lo\`{e}ve-type shape expansion
\begin{gather}
  \label{eq:rKL}
  r(\boldsymbol{Y}(\omega);\Bs) 
  = 
  \frac{1}{k}\,
  \sum\limits_{j=1}^{\infty}\beta_{j}\,Y_{j}(\omega)\,r_{j}(\Bs)\;,\quad \Bs\in\bbS^{d-1}\;,
\end{gather}
with \emph{$k$-independent} weights $\beta_{j}>0$ and globally supported, 
real-valued expansion functions $r_{j}\in C^{\infty}(\bbS^{d-1})$.
We assume that these are normalized such that $\N{r_{j}}_{C^{0}(\bbS^{d-1})}=1$, for all
$j\in\bbN$. We confine the discussion to the following particular choices:
  \begin{subequations}
    \label{eq:rj}
    \begin{itemize}
    \item For $d=2$ we opt for
      \begin{gather}
        \label{eq:rj2d}
        r_{j}(s) :=
        \begin{cases}
          \sin(\tfrac{j}{2}\,s) & \text{for even } j\;, \\
          \cos(\tfrac{j-1}{2}\,s) & \text{otherwise } \;, \\
        \end{cases}\quad 0\leq s < 2\pi \;.
      \end{gather}
      Here ${\N{r_{j}}_{C^{p,1}(\bbS^{1})}\sim j^{p+1}}$ as $j\to\infty$.
    \item For $d=3$ the $r_{j}$ are rescaled real spherical harmonics:
      \begin{gather}
        \label{eq:rj3d}
        r_{j} := Y_{\ell,m},\;\; j=\ell^{2}+\ell+m+1,\;\; -\ell\leq m\leq \ell, \;\;
        \ell\in\bbN_{0}.
      \end{gather}
      Here $\N{r_{j}}_{C^{p,1}(\bbS^{2})} \sim j^{(p+1)/2}$ as $j\to\infty$.
    \end{itemize}
\end{subequations}

We suppose that in \eqref{eq:rKL},
\begin{itemize}
\item 
$Y_{j}=Y_{j}(\omega)$ i.i.d uniformly in $[-1,1]$: $Y_{j}\sim \mathcal{U}([-1,1])$, 
and 
\item 
the deterministic weight sequence $(\beta_j)_{j\in\bbN}\in \ell_1(\bbN)$ 
with
$\N{(\beta_{j})}_{\ell_{1}(\bbN)}\leq 1 \leq k/3$.
\end{itemize}
These assumptions ensure $r\in\Cr$, \emph{cf.} \eqref{eq:R}.

\begin{remark}
    \label{rem:scalk}
    The scaling with ${k^{-1}}$ and the assumption 
    $Y_j \sim \mathcal{U}([-1,1])$ in \eqref{eq:rKL} 
    limits the shape variations in \eqref{eq:rKL} 
    to size $O(k^{-1})$ for $k\to\infty$, 
    that is, to a size proportional to the
    wavelength when $\beta_j$ is bounded independent of $k$, which is 
    a stronger requirement than the scaling 
    $\N{(\beta_{j})}_{\ell_{1}(\bbN)}\leq \frac{k}{3}$ stipulated above.
    In
    Section~\ref{ss:nonres} we found this to be necessary for the validity of polynomial
    surrogate modeling, which we have in mind throughout this work.
\end{remark}

\begin{remark}[Shape parametrization. Radial expansion functions]
  \label{rmk:RadExp}

\cs{In applications where $D$ is, for example, 
    an ``imperfect sphere'' \cite[Chapter~1]{SCA16}
    one may assume that the shape variations of $D$ are invariant under rotations,}
from which we conclude $\operatorname{Cov} r(\Bs,\Bs') = g(\Bs\cdot\Bs')$ for some
covariance function $g : \cintv{-1,1}\to \bbR$.  Then Karhunen-Lo\`{e}ve expansion of
$r=r(\omega)$ will yield exactly the radial spherical harmonic expansion functions $r_j$
of \eqref{eq:rj}.
\end{remark}

  \begin{remark}
    \label{rem:ocrj}
    All that follows can be adapted also to uncertainty
    parametrization with locally supported function systems, such as splines or wavelet
    functions; see \S\ref{sec:Concl} for further discussion.
  \end{remark}

For UQ one replaces the random variables $Y_{j}$ with deterministic parameters,
and places a product probability measure on the set of parameters.
In light
of \eqref{eq:rKL} the radial displacement function becomes a deterministic, affine-linear
function on the parameter set
\begin{equation}
  \label{eq:prms}
  \prms := \cintv{-1,1}^{\bbN} 
  = 
  \left\{\Vy = \left(y_{j}\right)_{j \in \bbN}:\;
    -1 \leq y_{j}\leq 1, \,\text{ for all } j\in\bbN\right\} \subset \ell^{\infty}(\bbN)\;, 
\end{equation}
via
\begin{equation}
  \label{eq:rd}
  \boxed{r(\Vy;\Bs) := \frac{1}{k} \sum\limits_{j=1}^{\infty}\beta_j\,y_j \,r_j (\Bs)\;,\quad
  \Bs\in\bbS^{d-1}\;,\quad 
  \Vy = (y_j)_{j\in\bbN} \in \prms} \;.
\end{equation}
Via this representation also the solution $\wh{u}=\wh{u}(r;\cdot)$ of \eqref{eq:vft},
the far-field pattern $\rhffp = \rhffparg[r]<\cdot>$ defined in \eqref{eq:faft},
the PML-truncated finite-element Galerkin solution $\rhuhpml[r]$, 
and the corresponding approximate far-field pattern $\rhffph[r]$ 
(assuming exact evaluation of the integral in \eqref{eq:faft2})
all can be regarded as 
deterministic functions of the parameters $\Vy \in \prms$. 
As such they will be tagged with a $\breve{\mbox{$\ $}}$\;, for instance,
\begin{gather}
  \label{eq:bru}
  \breve{u}: \prms\to H^1(B_2)\quad,\quad
  \Vy \mapsto \breve{u}(\Vy;\cdot) 
  := \wh{u}(r(\Vy);\cdot)\;.
\end{gather}
Since $r\mapsto\wh{u}(r;\cdot)$ 
could be extended into the complex domain, 
we can extend all
$\Vy\in\prms$-dependent functions 
to complex-valued parameter sequences $\Vz$
contained in (a suitable superset of)
$\cintv{-1,1}^{\bbN}+\iu \bbR^{\bbN} \subset \ell^{\infty}(\bbN,\bbC)$. 
Thanks to the affine dependence \eqref{eq:rd} of $r(\Vy;\cdot)$ on $\Vy$, 
the domain of analyticity of
the complex-parametric solution manifold
\begin{equation}\label{eq:phpb}
\Vz \mapsto {\breve{u}({\Vz} ;\cdot) := \wh{u}(\rC({\Vz};\cdot);\cdot)} 
\end{equation}
can immediately be read off Corollary~\ref{cor:A}. 
Here, we wrote $\rC(\Vz)$, ${\Vz\in \ell^{1}(\bbN,\bbC)}$, 
for functions defined by \eqref{eq:rd} 
with $z_{j}\in\bbC$ in place of the real-valued parameters $y_{j}\in [-1,1]$.

\begin{corollary}[Domain of analyticity of $\Vz \mapsto \breve{u}({\Vz} ;\cdot)$]
  \label{cor:H}
  The parametric mapping
  \begin{gather}
    {\Vz} \in \ell^{\infty}(\bbN,\bbC) \mapsto {\breve{u}({\Vz} ;\cdot) :=
    \wh{u}(\rC(\Vz);\cdot)}\in H^{1}(B_{2})\;,
  \end{gather}
  defined by combining
  \eqref{eq:uhatzdef} and \eqref{eq:rd} with \eqref{eq:phpb},
  is holomorphic on
  \begin{gather}
    \label{eq:Hh}
    \Ch = \Ch(C_{\Re},C_{\Im},k)
    := \left\{
      \begin{aligned}
        \Vzeta = \left( \zeta_{j} \right)_{j\in\bbN} \in \ell^{\infty}(\bbN,\bbC)
      \end{aligned}:\;
      \begin{aligned}
        & \sum\limits_{j=1}^{\infty}|\beta_{j}||\Re \zeta_{j}| \leq \frac{k}{3}, 
      \\ 
        & \sum\limits_{j=1}^{\infty}|\beta_{j}||\Re \zeta_{j}| \N{r_{j}}_{C^{1}(\bbS^{d-1})}
          \leq k C_{\Re}
        ,\\ & \sum\limits_{j=1}^{\infty}|\beta_{j}| |\Im \zeta_{j}| \N{r_{j}}_{C^{1}(\bbS^{d-1})}
            \leq C_{\Im}
      \end{aligned}
    \right\}\;,
  \end{gather}
  with the $\wn$-independent constants ${C_{\Re}},C_{\Im}>0$ as in Corollary~\ref{cor:A}.
\end{corollary}

A sufficient condition for the ``complexified'' radial displacement function $\rC(\Vz)$
from \eqref{eq:rd} to belong to class $\cA_{p}$ as defined in \eqref{eq:Hh_gamma}, that
is, a sufficient condition for the uniform wavenumber-explicit convergence estimates for
PML-based $h$-FEM in Theorem~\ref{thm:FEM1} and for the approximate far-field pattern from
Corollary~\ref{cor:FEMfar-field}, can be expressed in terms of 
complex-valued parameter sequences $\Vz$ in the set $\Ch_{p}$ for $p\in\bbN$, 
where
\begin{equation} \label{eq:Hp} 
  \Ch_{p}= \Ch_{p}(C_{\Re},C_{\Im},k)
  := 
  \left\{
    \Vz \in \Ch(C_{\Re},C_{\Im},k) : 
    \sum\limits_{j=1}^{\infty}|\beta_{j}||\Re z_{j}| \N{r_{j}}_{C^{p,1}(\bbS^{d-1})}\leq C_{\Re} k
  \right\}
\end{equation}
for $p\in\bbN$ \rev{(as introduced in the paragraph following \eqref{eq:FEM}). 
  Examples of
  admissible sequences $\left( \beta_{j} \right)_{j}$ for particular choices of the
  geometry variations $r_{j}$ will be presented in Section~\ref{sec:SmPMLFE}, see \eqref{eq:SumIbet}}.

From Corollary~\ref{cor:FEMfar-field}, \eqref{eq:ffbd}, 
and the triangle inequality, 
we conclude $k$-explicit, and $h$-uniform stability of the
complex-parametric, approximate far-field.
\medskip

\noindent\fbox{
  \begin{minipage}{0.975\textwidth}
  \begin{corollary}[Domain of analyticity of $\Vz\mapsto {\rbffph[\Vz]<\cdot>}$]
    \label{cor:ffpby}
Under the assumptions of Theorem~\ref{thm:FEM1} the approximate parametric
far-field pattern $\Vz\mapsto {\rbffph[\Vz]<\cdot>}$, obtained from $\rhffph$ in
\eqref{eq:faft2} via \eqref{eq:phpb}, is holomorphic on $\Ch_{p}$
    and satisfies  
    \begin{gather}
      \label{eq:ffpbybd}
      \exists C>0:\quad \NLtwo[\bbS^{d-1}]{\rbffph[\Vz]<\cdot>} 
             \leq C k \|\widehat{u}^{\mathrm{inc}}\|_{H^1_k(B_2)} \quad
      \begin{aligned}
        & \forall k\geq k_1 \;\text{($k_{1}$ as in Theorem \ref{thm:inherit})},\\
        & \forall \Vz \in \Ch_p ,\\
        & \forall h,p,
        \;\text{satisfying \eqref{eq:FEMthreshold}.}
      \end{aligned}
    \end{gather}
An analogous assertion is also valid for the exact far-field pattern
    $\Vz\mapsto \rbffparg[\Vz]<\cdot>$.
  \end{corollary}
\end{minipage}}
\medskip

Thanks to Remark~\ref{rmk:PMLHol} we can also conclude $k$- and $h$-uniform parametric holomorphy of
the PML/FEM-Galerkin solutions.

\begin{corollary}[Uniform parametric holomorphy of PML/FEM-Galerkin solutions]
  \label{cor:ffpby2}
Under the assumptions of Theorem~\ref{thm:FEM1}, for
  \begin{itemize}
  \item $\Vz\mapsto \breve{u}^{\rm alt}(\Vz;\cdot)$, the outgoing solution of \eqref{eq:pw_f}
    transformed to the nominal domain, obtained from $\refs{u}^{\rm alt}(\rC;\cdot)$ via
    the correspondence \eqref{eq:bru}, 
  \item
    $\Vz\mapsto \breve{u}^{\rm alt}_{\rm PML}(\Vz;\cdot)$
    corresponding similarly to the exact solution of the PML variational formulation
    \eqref{eq:PMLvf} with the corresponding right-hand side, and 
  \item the parameterized PML/FEM-Galerkin approximation
    $\Vz\mapsto \breve{u}^{\rm alt}_{\rm PML}(\Vz;\cdot)_h$ derived from the solution of
    \eqref{eq:FEM},
  \end{itemize}
  the following holds.
  
  These parametric solution families are holomorphic as an $H^1_k(\Omega)$-valued map in a
  \emph{$h$ and $k$-independent} set
  $\Ch_{p}(C_{\Re},C_{\Im},{k_1})\subset \ell^{\infty}(\bbN,\bbC)$ (with $k_{1}$ as in
    Theorem~\ref{thm:inherit}), and, uniformly on this set, satisfy the bound
  \eqref{eq:falt}, with $\breve{u}^{\rm alt}_{\rm PML}(\Vz;\cdot)$ and
  $\breve{u}^{\rm alt}_{\rm PML}(\Vz;\cdot)_h$ in place of
  $\refs{u}^{\rm alt}(\rC;\cdot)$, correspondingly.
\end{corollary}
\section{Computational Shape Uncertainty Quantification}
\label{sec:AplShUQ}
\newcommand{\pbar}{\bar{p}}

The $k$ explicit, uniform w.r. to $h$ parametric holomorphy of the PML-Galerkin solution
is the basis for the error analysis of efficient, deterministic computation of quantities
of interest (QoIs).  In a UQ-context these are, for example, the expectation
$\bbE(\breve{u}^{\rm scat}_\infty)$ and, possibly, higher order spatial correlation
functions of the corresponding random field
$\omega \mapsto \breve{u}^{\rm scat}_\infty(\omega):=\breve{u}^{\rm
  scat}_\infty(\VY(\omega);\cdot)$.  Such QoIs will be deterministic, smooth functions on
$\bbS^{d-1}$.  Since all random variables $Y_{j}(\omega)$ are uniformly distributed in
$[-1,1]$, we find with the parametric, deterministic shape representation $r(\Vy;\cdot)$
from \eqref{eq:rd} the following expressions for the mean 
of the far-field pattern corresponding to $\breve{u}$ in \eqref{eq:bru}, i.e.
\begin{equation}
  \label{eq:mv}
  \begin{array}{rl}
    \bbE(\breve{u}^{\rm scat}_\infty)(\unitx)
    & = \int\nolimits_{\prms}
    \breve{u}^{\rm scat}_\infty({\Vy} ; \unitx)\,\mathrm{d}\mu({\Vy})\;, \quad \unitx\in\bbS^{d-1} \;,
\end{array}
\end{equation}
for all far-field directions $\unitx = \nicefrac{\Bx}{|\Bx|}\in\bbS^{d-1}$.
Here, 
the measure $\mu$ denotes 
the countable product of the uniform probability measure on $[-1,1]$.
We add that in actual computations in \eqref{eq:mv} the 
far-field pattern $\breve{u}^{\rm scat}_\infty$ 
has to be replaced with the computable approximation $\rbffph[\cdot]<\cdot>$ 
in \eqref{eq:faft2}.

The formula \eqref{eq:mv} is a so-called ``ensemble average'' over all admissible
scatterer shapes, and involve ``infinite-dimensional integrals'', whose efficient
approximation by (possibly higher-order and deterministic) quadrature formulas is
addressed next.
\subsection{High-Dimensional Smolyak / Sparse-Grid Quadrature}
\label{sec:HDQad}
The QoI \eqref{eq:mv} 
being a countably-parametric, deterministic integral, 
we discuss two classes of numerical integration: 
first,
so-called \emph{Smolyak-Quadrature} (see \cite{ZS20_2485} and references there) 
and 
second, so-called \emph{Higher-Order QMC Quadrature}, in particular so-called
\emph{interlaced polynomial lattice rules}  (``IPL-QMC integration'', see \cite{JDTlGChS2016}).
In these references,
convergence rates of both numerical integration methods have been proved 
to be independent of the dimension of the domain of integration, under suitable
quantified holomorphy of the integrands.
In view of the preceding discussion on quantified, 
wavenumber-explicit holomorphy of 
the parametric solutions
$\{ \breve{\Vy} \mapsto \hat{u}(\breve{\Vy}) : \breve{\Vy} \in \prms \}$, 
we connect the parametric holomorphy analysis to \cite{JDTlGChS2016,ZS20_2485}.
We start by recalling a suitable concept of quantified parametric holomorphy,
which implies dimension-independent convergence rates of 
(i) Smolyak-quadratures \cite{Z18_2760,ZS20_2485}, 
(ii) sparse-grid interpolants \cite{ChCChS14,Z18_2760} 
and 
(iii) higher-order Quasi-Monte Carlo quadratures \cite{JDTlGChS2016}.

In order to quantify the holomorphic parameter dependence, 
we quantify the domain of parametric holomorphy of the (analytic continuation of the)
affine-parametric function
  $\{ \breve{\Vy}\in \prms \mapsto u_{\breve{\Vy}} : \breve{\Vy} \in \prms \}$. 
We still
work with the $\wn^{-1}$-scaling and the affine-parametric parametrizations
$\breve{\Vy} \mapsto \rC(\breve{\Vy})$ from \eqref{eq:rd} and with the globally supported
radial expansion functions $r_j$ from \eqref{eq:rj}.

We present quantified parametric holomorphy for a generic ``target'' Banach space $X$.
Observe that the QoI's in \eqref{eq:mv} take values in the separable, complex Banach
space $X \in \{ C^0(\bbS^{d-1};\bbC) , L^2(\bbS^{d-1};\bbC)\}$.
\begin{definition}[$(\Bb,\pbar,\varepsilon)$-Holomorphy]
\label{def:bpeHol}
Let $X$ be a complex Banach space with norm $\| \circ \|_X$.  For $\varepsilon>0$, a
sequence $\Bb = (b_j)_{j\geq 1} \in (0,\infty)^\bbN$ and some $p\in (0,1)$, 
the parametric map
$$
\prms \ni \breve{\Vy} \mapsto u_{\breve{\Vy}} \in X
$$
is called $(\Bb,\pbar, \varepsilon)$-holomorphic if the following conditions hold.
\begin{itemize}
\item[(i)] 
The map $\prms \ni \breve{\Vy} \mapsto u_{\breve{\Vy}}$ is uniformly bounded, i.e.
there exists a bound $M_0>0$ such that
$$ 
\sup_{\breve{\Vy} \in \prms} \| u_{\breve{\Vy}} \|_X \leq M_0\;,
$$
\item[(ii)]
there holds $\Bb\in \ell^{\pbar}(\bbN)$ and 
there exists a constant $C_\varepsilon > 0$ 
such that for any sequence 
$\Vrho = (\rho_j)_{j\geq 1} \in (1,\infty)^\bbN$
which is \emph{$(\Bb,\varepsilon)$-admissible}, 
i.e.
\beq\label{eq:bepsAdm}
\sum_{j\geq 1} (\rho_j-1)b_j \leq \varepsilon \;,
\eeq
the map $\prms \ni \breve{\Vy} \mapsto u_{\breve{\Vy}}$ 
admits a complex extension $\breve{\Vz} \mapsto u_{\breve{\Vz}}$
that is continuous w.r. to $\breve{\Vz}$ and 
holomorphic w.r. to each variable $\breve{z}_j$ 
on a cartesian product set of the form
$$
\Co_{\Vrho} := \bigtimes_{j\geq 1} \Co_{\rho_j} 
\;.
$$
Here,
for $j\geq 1$, 
$\Co_{\rho_j}\subset \bbC$ is some open set with 
$[-1,1] \subset \cD_{\rho_j} \subset \Co_{\rho_j}$
with strict inclusions where, for $\rho>1$, 
$\cD_\rho := \{ z\in \bbC:  |z| \leq \rho \} $ denotes
the closed disc in $\bbC$ of radius $\rho>1$.
\item[(iii)]
For each $(\Bb,\varepsilon)$-admissible polyradius $\Vrho$, 
the holomorphic extension 
$\{ \breve{\Vz} \mapsto u_{\breve{\Vz}} : \breve{\Vz} \in \Co_{\Vrho} \} \subset X$
of the parametric map 
$\breve{\Vy} \mapsto u_{\breve{\Vy}}$ 
is bounded on the polydisc $\cD_{\Vrho}=\bigtimes_{j\geq 1} \cD_{\rho_j} $ according to
\beq\label{eq:HolExtBd}
\sup_{\breve{\Vz} \in \cD_{\Vrho}} \| u_{\breve{\Vz}} \|_X \leq M_u \;.
\eeq
\end{itemize}
\end{definition}
The significance of $(\Bb,\pbar,\varepsilon)$-holomorphy lies in the fact that 
holomorphic maps between (complex) Banach spaces become, upon adopting
affine-parametric representation of their arguments w.r. to a suitable
representation system 
(such as, e.g. $\{ \psi_j \}_{j\geq 1}$ with $\psi_j \sim \beta_j r_j$ 
where $r_j$ is as in Remark~\ref{rmk:RadExp})
$(\Bb,\varepsilon)$-holomorphic maps in terms of the coefficient sequences
in the representation of inputs (see \cite[Lemma~3.3]{ZS20_2485}).
The summability exponent $\pbar\in (0,1)$ of the sequence $\bsb$ 
determines the convergence rate of suitable Smolyak quadratures 
in \eqref{eq:SmolErr} below.

\cs{
Next, we proceed to estimating the Smolyak Quadrature Error.
To this end,}
we recall the definition of a sparse-grid Smolyak quadrature: given a 
sequence $(\chi_{n;j})_{j=0}^n$, $n\in \bbN_0$ of $n$-tuples of pairwise
distinct points in $[-1,1]$, 
Smolyak quadratures are built on 
corresponding univariate interpolatory quadrature rules $(Q_n)_{n\geq 0}$ 
with nodes $\chi_{n;0},...,\chi_{n;n}\subset [-1,1]$ 
and corresponding weights $w_{n;j} > 0$ 
w.r. to the uniform (probability) measure $\frac{1}{2} \lambda^{1}$, 
i.e. 
$$
Q_n f = \sum_{j=0}^n w_{n;j} f(\chi_{n;j}) \;,\quad 
w_{n;j} 
:= 
\frac{1}{2} \int_{-1}^1 \prod_{i=0,i\ne j}^n \frac{y-\chi_{n;j}}{\chi_{n;i}-\chi_{n;j}} dy
\;.
$$
By construction, $Q_n$ is exact for univariate polynomials of degree $n$.
The weights $w_{n;j}$ can be negative, in general.
For the error bound \cite[Theorem 2.16]{ZS20_2485} to hold, 
\cs{
we assume stability of the univariate quadrature points
\cite[Eqn. (2.3)]{ZS20_2485}, i.e. 
in the sense that there exist a constant $\vartheta>0$}
such that
\begin{equation}\label{eq:Qstab}
\forall n\in \bbN_0: \quad 
\sup_{0\ne f\in C^0([-1,1])} \frac{|Q_nf|}{\| f \|_{C^0([-1,1])}} \leq (n+1)^\vartheta
\end{equation}
We denote in the following the array of univariate sampling points 
\begin{equation}\label{eq:bschi}
\bschi := \left( (\chi_{n;j})_{j=0}^n \right)_{n\in \bbN_0} 
\;.
\end{equation}
We adopt the convention that $\chi_{0;0} = 0$, i.e. 
$Q_0$ corresponds to the midpoint rule, and also set $Q_{-1}:=0$. 
We refer to \cite{CC20} and the references there for concrete constructions
of such points.

Multivariate anisotropic quadratures are 
built from the univariate hierarchy $(Q_n)_{n\geq 0}$ 
by tensorization. 
Let $\cF = \{\bsnu\in \bbN_0^\bbN: |\bsnu|<\infty\}$ denote
the set of finitely supported multiindices, and let 
$\Lambda\subset \cF$ be a downward closed\footnote{
We recall (e.g. \cite[Definition~1.1]{ChCChS14}) that
an index set $\Lambda\subset \cF$ is downward closed (``d.c.'' for short) 
if 
$\bsnu\in \Lambda$ and $\bsmu\leq \bsnu$ implies $\bsmu\in \Lambda$. 
Here, $\bsmu\leq \bsnu$ means $\mu_j\leq \nu_j$ for all $j$.
}
finite index set.
Then for $\bsnu\in \cF$, 
we define the multivariate tensor product quadrature 
$Q_\bsnu := \bigotimes_{j \in \bbN} Q_{\nu_j}$.
Then the \emph{Smolyak Quadrature} for 
a d.c. set $\Lambda \subset \cF$ is defined by 
\begin{equation}
\label{eq:Smolyak}
Q_\Lambda := \sum_{\bsnu\in \Lambda} \bigotimes_{j \in \bbN} (Q_{\nu_j} - Q_{\nu_j -1}) \;. 
\end{equation}
In particular, $Q_\Lambda$ admits the representation
$$
Q_\Lambda := \sum_{\bsnu\in \Lambda} \iota_{\Lambda,\bsnu} Q_\bsnu\;,
\quad 
\mbox{where}
\quad
\iota_{\Lambda,\bsnu} := \sum_{\bse\in \{0,1\}^\bbN: \bsnu+\bse\in \Lambda} (-1)^{|\bse|} 
\;.
$$
Based on this representation, 
one numerical evaluation of $Q_\Lambda$ requires accessing the 
integrand function in all points in the (finite) set
$$
\pts(\Lambda,\bschi) 
:= 
\left\{ (\chi_{\nu_j;\mu_j} : \bsnu\in \Lambda, \iota_{\Lambda,\bsnu}\ne 0, \bsmu\leq \bsnu \right\}
\subset \prms
\;.
$$
We refer to \cite[Section 2.2]{ZS20_2485} for details. 

The main result from \cite[Section 2.5]{ZS20_2485} on the convergence rate of 
$Q_\Lambda$ for suitable downward closed sets $\Lambda\subset \cF$ 
of ``active quadrature orders'' $\bsnu\in \Lambda$ is as follows.

\begin{theorem}\label{thm:Smolyak}
  Let $Z$ and $X$ be complex Banach spaces.  Denote for $r>0$ with
  $B^Z_r = \{ \varphi\in Z: \| \varphi\|_Z <r \}$ the open ball in $Z$ centered at the
  origin of radius $r$.  
  Assume that we are given a \emph{holomorphic} map
  $\Fu:B_r^Z \to X$, a real constant $\delta>0$, a sequence
  $(\psi_j)_{j\in \bbN}\subset Z$, $r>0$, and $\pbar\in (0,1)$.  
  Fix $\delta>0$ arbitrarily small.  
  Assume further that the following hold:
\begin{itemize}
\item[(i)] $\sum_{j \geq 1} \| \psi_j \|_Z < r$ and the sequence
  $\bsb = (\| \psi_j \|_Z) \in \ell_{\pbar}(\bbN)\subset \ell_1(\bbN)$,
\item[(ii)] $\Fu: B^Z_r \to X$ is holomorphic and bounded {by $M_u$},
\item[(iii)] the collection of univariate quadrature abscissae $\bschi$ satisfies
  \eqref{eq:Qstab}.
\end{itemize}
Define the countably parametric map $U: \prms \to X$ via the composition
\begin{equation}\label{eq:DefParInt}
\breve{\Vy} \mapsto U(\breve{\Vy}) := \Fu \left(\sum_{ j \geq 1 } \breve{y}_j \psi_j\right).
\end{equation}

Then, there holds
\begin{itemize}
\item[(i)] {[$\wn$-independent parametric holomorphy]}
The parametric maps $U$ in \eqref{eq:DefParInt} are $(\bsb, \bar{p} ,\varepsilon)$ holomorphic
in the sense of Definition~\ref{def:bpeHol}, with some $\varepsilon>0$ independent of $\wn$,
and
\item[(ii)] {[$\wn$-uniform convergence rate of Smolyak quadrature error]}
There exists a constant $C>0$ (depending on $\delta>0$, but independent of $\wn$) 
such that for every $\epsilon > 0$, there exists 
a finite, downward closed multiindex set $\Lambda_\epsilon \subset \cF$ 
with $|\Lambda_\epsilon| \to \infty$ as $\epsilon \to 0$ 
such that the following error bound holds:
\begin{equation}\label{eq:SmolErr}
\left\| 
\int_\prms U(\breve{\Vy}) d\mu(\breve{\Vy}) - Q_{\Lambda_\epsilon} U 
\right\|_X 
\leq 
C 
|\pts(\Lambda_\epsilon, \bschi)|^{-\frac{2}{\pbar}+1+\delta} 
{M_u}
\;.
\end{equation}
\end{itemize} 
\end{theorem}
This is \cite[Theorem~2.16 and Theorem~ 4.3, Eqn.(4.1)]{ZS20_2485}, with
the $(\bsb,\pbar,\eps)$-holomorphy 
\cs{following} from \cite[Lemma~3.3]{ZS20_2485}.
Inspection of the proofs in \cite{ZS20_2485} reveals, in particular,
that the constant $C>0$ in the quadrature error bound in the
statement of \cite[Theorem~ 4.3, Eqn.(4.1)]{ZS20_2485} scales linearly
in the integrand modulus $M_u$, whence the bound \eqref{eq:SmolErr}.

Here, as in \eqref{eq:mv}, 
the measure $\mu$ in \eqref{eq:SmolErr} 
is a countable product probability measure obtained from 
the univariate scaled Lebesgue measure, $\frac{1}{2} \lambda^1$.
For 
\emph{nested collections of univariate integration points} 
in the sense of \cite[Definition 2.1]{ZS20_2485}
it holds that 
$|\pts(\Lambda,\bschi)| = |\Lambda|$ (\cite[Lemma 2.1]{ZS20_2485}), 

The constant $C>0$ in \eqref{eq:SmolErr} depends on the sequence $\bsb$ in the statement
of the $(\bsb,\pbar,\eps)$ holomorphy of the parametric integrand function in \eqref{eq:DefParInt}.
%
\subsection{Combined Smolyak-PML/FEM error bounds}
\label{sec:SmPMLFE}
%
With the Smolyak-quadrature error bound Theorem~\ref{thm:Smolyak} in hand, we proceed to
estimate the error in the quadrature-FE approximation of the corresponding mean 
of the far-field.

The integral \eqref{eq:mv}, 
is approximated by a 
Smolyak type quadrature formula $Q_\Lambda[\cdot]$ as in \eqref{eq:Smolyak}.
We write, 
for any far-field direction $\unitx \in \bbS^{d-1}$,
and with the approximate far-field 
$\refs{u}^{\rm scat}_{\infty, \rm PML}(\rC;\unitxh)_h$ 
as in \eqref{eq:faft2},
\begin{multline}\label{eq:Bound}
    \bbE[\breve{u}^{\rm scat}_\infty(\cdot; \unitx)]
    -
    Q_\Lambda[ \rbffph[\cdot ]<\unitx> ] 
    \\ 
    = \underbrace{\bbE[\breve{u}^{\rm scat}_\infty (\cdot; \unitx) - \rbffph[\cdot ]<\unitx>)]}_{\esc{=:I}}
    + \underbrace{\bbE[\rbffph[\cdot ]<\unitx>] -  Q_\Lambda[ \rbffph[\cdot ]<\unitx> ] }_{\esc{=:II}} .
\end{multline}
We estimate the $L^2(\bbS^{d-1};\bbC)$-norms 
of the terms $I$ and $II$ in \eqref{eq:Bound} separately.

{\bf Term $I$} is the expected discretization error of the 
PML-Galerkin FE discretization.
For fixed direction $\unitx \in \bbS^{d-1}$,
estimated by the uniform 
(w.r.t. all admissible shapes $\breve{y}\in \cP$) 
error bound in the computed far-field pattern 
$\refs{u}^{\rm scat}_{\infty, \rm PML}(\rC;\cdot)_h$
which was bound in
Corollary~\ref{cor:FEMfar-field}, \eqref{eq:FEMfar-field}.
Due to $\mu(\cP) = 1$, 
\begin{equation}\label{eq:ffI}
\sup_{\unitx \in \bbS^{d-1}}
\left| 
\bbE( \breve{u}^{\rm scat}_\infty(\cdot; \unitx) - \rbffph[\cdot ]<\unitx> )
\right| 
\leq 
\sup_{\rC \in \cA_p} 
\N{\refs{u}^{\rm scat}_\infty(\rC;\cdot)
-\refs{u}^{\rm scat}_{\infty, \rm PML}(\rC;\cdot)_h 
}_{L^\infty(\mathbb{S}^{d-1})}
\;.
\end{equation}
We further majorize term $I$ by the bound \eqref{eq:FEMfar-field}, 
using the assumptions of Theorem~\ref{thm:FEM1}
and with $k_{1}$ as in Theorem~\ref{thm:inherit}, 
for all $h,k$, and $p$ satisfying \eqref{eq:FEMthreshold}, 
provided that $\rC(\Vy;\cdot) \in \cA_p$ uniformly w.r.t. $\Vy\in \cP$.

A sufficient condition for the 
regularity $\rC(\Vz;\cdot) \in \cA_p$ 
of the parametric geometry representation
\eqref{eq:rj2d}-\eqref{eq:rd}
is that that $\Vz\in  \Ch_{p}$, see \eqref{eq:Hp}.
With the radial spherical harmonics expansion functions $r_j$
as in Rem.~\ref{rmk:RadExp},
i.e. \eqref{eq:rj2d} for $d=2$ and \eqref{eq:rj3d} for $d=3$,
it holds that
\[
\| r_j \|_{C^{p+1}(\bbS^{d-1})} \simeq \| r_j \|_{C^{p,1}(\bbS^{d-1})} \simeq j^{(p+1)/(d-1)}\;,\;\; 
j=1,2,... 
\;.
\]
A sufficient condition for $\rC \in \cA_p$ 
is that the weight sequence 
$(\beta_j)_{j\geq 1} \in (0,\infty)^\bbN$
in \eqref{eq:rd} admits the bound
\begin{equation}\label{eq:SumIbet}
\exists C>0:\quad \beta_j  \leq C j^{-1-\epsilon-(p+1)/(d-1)} \;,\;\; j=1,2,3,... \;.
\end{equation}
For $\Vy\in \cP$, and every $p\geq 1$,
the geometry representation in \eqref{eq:rKL}, 
i.e. the sum
\[
\sum_{j\geq 1} \beta_jy_j r_j(s) 
\]
then converges in $C^{p,1}(\bbS^{d-1}) \sim C^{p+1}(\bbS^{d-1})$, 
uniformly w.r. to $\Vy\in \cP$,
and \eqref{eq:ffI} can be majorized by \eqref{eq:FEMfar-field}.

{\bf Term $II$}
is the quadrature error of the Smolyak quadrature $Q_\Lambda$ in \eqref{eq:Smolyak}
applied to the parametric integrand 
$\refs{u}^{\rm scat}_{\infty, \rm PML}(\rC;\cdot)_h$.
We estimate it with Theorem~\ref{thm:Smolyak}, 
\emph{assuming} \eqref{eq:SumIbet} on the weight sequence $\Bbeta$.
We introduce the notation $\bar{p}\in (0,1]$ for the summability
exponent of the weight sequence $\Bbeta$, and keep $p\in \bbN$ 
for the smoothness class in e.g. \eqref{eq:Hp}.

With $p\geq 1$, i.e.  $\Vz\in \Ch_p$ as defined in \eqref{eq:Hp}, $\bar{p}$-summability of
the weights $(\beta_j)_{j\geq 1} \in (0,\infty)^\bbN$ satisfying \eqref{eq:SumIbet} to
bound term $I$ will imply, thanks to Theorem~\ref{thm:Smolyak}, quadrature convergence rates immune
to the Curse of Dimensionality for Smolyak (or ``sparse grid'') quadrature, as described in
Section~\ref{sec:HDQad} and as analyzed in \cite{ZS20_2485}.  
It will also imply, via the parametric $(\bsb,\pbar,\eps)$-holomorphy shown in 
corresponding results for higher-order Quasi-Monte Carlo
(``HoQMC'' for short) deterministic integration, as shown in \cite[Theorem~3.1 and
Proposition~4.1]{JDTlGChS2016}.  We leverage these results, with $k$-explicit error bounds, by
verifying $(\bsb,\pbar,\eps)$-holomorphy (cf. Definition~\ref{def:bpeHol}) of the parametric
integrands $\Vz \mapsto \rbffph[\Vz ]<\cdot>$.

To verify the assumptions in Theorem~\ref{thm:Smolyak}, 
we remind that we work in the shape-parametrization \eqref{eq:rd} 
with the (globally supported, in $\bbS^{d-1}$) 
radial spherical harmonics expansion functions 
$r_j$ as in Rem.~\ref{rmk:RadExp}, 
i.e. \eqref{eq:rj2d} for $d=2$ and \eqref{eq:rj3d} for $d=3$.
We choose in Theorem~\ref{thm:Smolyak}, item (i), 
\begin{equation}\label{eq:ChoicXZ}
Z   = C^{2}(\bbS^{d-1};\bbR) \subset C^{1,1}(\bbS^{d-1};\bbR)
\quad \mbox{and} \quad
X   = C^{0}(\bbS^{d-1};\bbC) 
\;,
\end{equation}
with corresponding norms.  The $C^{1,1}$-regularity implied by the choice of $Z$ is
required in Theorem~\ref{thm:FEM1}. 

Based on the correspondence \eqref{eq:phpb}, we furthermore set $\psi_j = \beta_j r_j$ and
identify $\Fu$ in Theorem~\ref{thm:Smolyak}, item (ii), with
$\refs{u}^{\rm scat}_{\infty, \rm PML}(\rC;\cdot)_h$ in \eqref{eq:faft2} resulting in the
parametric integrand $U(\breve{\Vy})$ in \eqref{eq:DefParInt} being
$\rbffph[\breve{\Vy}]<\cdot>$.
 
Assumption \eqref{eq:SumIbet} 
on the weight sequence 
$\Bbeta = (\beta_j)_{j\geq 1}$ 
and $\psi_j = \beta_j r_j$ implies 
with \eqref{eq:rj2d}, \eqref{eq:rj3d} for $d=2,3$ with $p=2$ 
that 
(using $\| r_j \|_{C^{1,1}(\bbS^{d-1})} \sim \| r_j \|_{C^2(\bbS^{d-1})} \sim j^{3/(d-1)}$ as $j\to \infty$;
 the $C^{1,1}$-regularity is required by Theorem~\ref{thm:FEM1})
\begin{equation}\label{eq:Defbj}
b_j := \| \psi_j \|_Z \simeq j^{1/(d-1)} \beta_j \lesssim j^{-(1+\epsilon + p/(d-1))}.
\end{equation}
Then, 
$(\| \psi_j \|_Z)_{j\geq 1} \in \ell_{\bar{p}}(\bbN)$,
if
\begin{equation}\label{eq:pSumab}
\left( 1 + \frac{p}{d-1} \right)^{-1} \leq \pbar < 1 
\;,
\quad 
p=1,2,...
\end{equation}
%

The following result complements Cor.~\ref{cor:ffpby}.

\begin{proposition}\label{prop:bpeHol}
Suppose that the assumptions of Theorem~\ref{thm:FEM1} hold
and 
assume \eqref{eq:SumIbet} for the weight sequence $\Bbeta = (\beta_j)_{j\geq 1}$,
and the affine shape parametrization \eqref{eq:rd}
with the expansion coefficients 
$\psi_j = \beta_j r_j$ with \eqref{eq:rj2d}, \eqref{eq:rj3d} for $d=2,3$. 

Then, there holds
that
the parametric PML-Galerkin FE far-field approximations 
$$\{ \Vy\mapsto \rbffph[\Vy ]<\unitx>: \Vy\in \cP\}$$
as defined in \eqref{eq:faft2} via \eqref{eq:phpb} are, with the sequence
$\bsb = (b_j)_{j\geq 1}$ as defined in \eqref{eq:Defbj}, $(\bsb, \bar{p},\varepsilon)$
holomorphic, uniformly w.r.t. $h$, $\wn$, in the space $X = C^0(\bbS^{d-1};\bbC)$.
\end{proposition}
\begin{proof}
This follows from \cite[Lemma~3.3]{ZS20_2485}
with the summability exponent $\bar{p}$ bounded as in \eqref{eq:pSumab}.
\end{proof}

With Proposition~\ref{prop:bpeHol}, 
the Smolyak error bound \eqref{eq:SmolErr} then 
implies that with $p\geq 1$ as in \eqref{eq:SumIbet} 
\[
\begin{array}{rcl}
\displaystyle
\left\| 
\bbE[\rbffph[\cdot ]<\unitx>] - Q_\Lambda[ \rbffph[\cdot ]<\unitx> ] 
\right\|_{L^2(\bbS^{d-1};\bbC)} 
&\leq&  
C M_{\rbffph[\cdot]<\cdot>} 
|\pts(\Lambda_\epsilon, \bschi)|^{-\frac{2}{\bar{p}}+1+\delta} 
\\
&\leq& 
C M_{\rbffph [\cdot]<\cdot>}
|\pts(\Lambda_\epsilon, \bschi)|^{-1-\frac{2p}{d-1}+\delta} 
\end{array}
\]
By Corollary~\ref{cor:ffpby} and \eqref{eq:ffpbybd}, the modulus
$ M_{\rbffph [\cdot]<\cdot>} $ admits the $k$-explicit bound
\[
M_{\rbffph [\cdot]<\cdot>} 
\leq 
\sup_{\Vz \in \Ch} \NLtwo[\bbS^{d-1}]{\rbffph[\Vz]<\cdot>} 
\leq 
C k \|\widehat{u}^{\mathrm{inc}}\|_{H^1_k(B_2)}
\;.
\]
Combining this with the bound \eqref{eq:ffI} for term $I$, 
we arrive for scatterer-geometries of 
regularity $C^{p,1}(\mathbb{S}^{d-1})$ for some $p\geq 1$
at the error bound
\begin{multline}\nonumber
    \left\|  
      \bbE(\breve{u}^{\rm scat}_\infty(\cdot; \unitx))
      -
      Q_\Lambda[ \rbffph[\cdot ]<\unitx> ] 
    \right\|_{L^2(\bbS^{d-1};\bbC)} 
    \\
    \leq 
    C
    \sup_{\rC \in \cA_p} 
    \N{\refs{u}^{\rm scat}_\infty(\rC;\cdot)
      -\refs{u}^{\rm scat}_{\infty, \rm PML}(\rC;\cdot)_h 
    }_{L^\infty(\mathbb{S}^{d-1})}
    +
    C k \|\widehat{u}^{\mathrm{inc}}\|_{H^1_k(B_2)}
    |\pts(\Lambda_\epsilon, \bschi)|^{-1-\frac{2p}{d-1}+\delta}
\end{multline}
with the first term of the bound in turn majorized by \eqref{eq:FEMfar-field}.
Absorbing the ($\wn$-independent) constants yields
an $\wn$-explicit error bound for the 
\rev{combined FE-PML, Smolyak-quadrature approximated}
\emph{expected far-field pattern}.
\medskip

\rev{%
\noindent\fbox{%
  \begin{minipage}{0.975\linewidth}%
\begin{theorem}\label{thm:CombErrBd}
Suppose that the dimension of the scatterer is $d\in \{2,3\}$
and that the assumptions of Theorem~\ref{thm:FEM1} hold.
Assume further \eqref{eq:SumIbet} 
for the weight sequence $\Bbeta = (\beta_j)_{j\geq 1}$
ensuring that the scatterers' geometry is $C^{p,1}$-regular,
and the affine shape parametrization \eqref{eq:rd}
with the expansion coefficients 
$\psi_j = \beta_j r_j$ with \eqref{eq:rj2d}, \eqref{eq:rj3d} for $d=2,3$.

Then the approximate, expected far-field pattern obtained by 
Smolyak-quadrature $Q_\Lambda$ in \eqref{eq:Smolyak} 
applied to the parametric FE-PML approximations 
$\rbffph[\cdot ]<\unitx>$ satisfies the error bound
    \begin{multline}
    \nonumber
\left\|  
    \bbE(\breve{u}^{\rm scat}_\infty(\cdot; \unitx))
    -
    Q_\Lambda[ \rbffph[\cdot ]<\unitx> ] 
\right\|_{L^2(\bbS^{d-1};\bbC)} 
\\
\lesssim
\|\widehat{u}^{\mathrm{inc}}\|_{H^1_k(B_2)}
\left[
 \big( hk + (hk)^p k\big)(hk)^p k
 +
  e^{- C_{\rm PML, 2} k\big(R_{\rm tr}-(1+\epsilon) R_1\big)}
+
k
|\pts(\Lambda_\epsilon, \bschi)|^{-1-\frac{2p}{d-1}+\delta}
\right]
\;.
\end{multline}
Here, the constant hidden in $\lesssim$ is
bounded independent of $\wn$, but depends on
the regularity $p\geq 1$ of the geometry,
and on the space dimension $d$.
\end{theorem}
\end{minipage}
}}
\medskip

\begin{remark}[Higher Order Quasi-Monte Carlo - PML Galerkin FE error bound]
\label{rmk:HoQMCFE}
In Proposition~\ref{prop:bpeHol}, 
we observed that 
the parametric integrands $\Vy \mapsto \rbffph[\breve{\Vy}]<\cdot>$
are $(\bsb,\bar{p},\varepsilon)$-holomorphic, 
uniformly with respect to $h$,
as vector-valued integrand functions taking values in the 
(separable) Hilbert-space $X= L^2(\bbS^{d-1};\bbC)$.
For such integrand functions, certain deterministic 
so-called \emph{Quasi-Monte Carlo integration methods} 
have been shown to furnish likewise higher orders of convergence,
without the CoD. 
For general introduction and comprehensive presentation of this class
of high-dimensional quadrature methods, we refer to \cite{QMCBook}.

Specifically, based on Proposition~\ref{prop:bpeHol}, 
it was shown in \cite[Theorem~3.1, Proposition~4.1]{JDTlGChS2016}, 
that deterministic quadrature rules with $N$ points can be constructed
that approximate the integral \eqref{eq:mv} with rate $N^{-1/\bar{p}}$. 
The construction can be effected in $O(N\log N)$ many operations, 
subject to prior truncation to a suitable finite number $s$ of integration
variables.

Particular QMC quadratures covered by the weighted norm setting in 
\cite{JDTlGChS2016} allow for computable 
a-posteriori error estimation of the QMC integration error \cite{DLS22}.
\rev{For related, recent work on wavenumber-explicit error estimates of 
    QMC FE approximations in UQ, we refer to \cite{GKN25}.}
\end{remark}

\section{Conclusion}
\label{sec:Concl}
We have established shape-holomorphy of the approximate far-field pattern based on
domain transformation, finite-element Galerkin approximation, and PML truncation for a
class of time-harmonic acoustic scattering transmission problems. We have done this in the
special setting of star-shaped scatterers parameterized by their radial extension and for
piecewise homogeneous isotropic media, for which the refractive index outside the
scatterer is larger than that inside. In this situation we could accomplish the first
analysis that is \emph{fully explicit in the  wavenumber $\wn$}. We could also derive a
$\wn$-explicit estimate for the sparse grid (Smolyak) quadrature error in the case of an
affine parameterization of the radial displacement function that determines the shape of
the scatterer.

In order to keep the focus of the article, several issues and possible extensions have not
been addressed:
\begin{itemize}
\item In the present affine-parametric representation \eqref{eq:rKL}-\eqref{eq:rd},
  a basis $\{ r_{j}(\Bs) \}_{j\geq 1}$ of $L^2(\bbS^d)$ 
  with global support in $\bbS^d$ was employed. 
  Locally supported basis functions $r_{j}$ like (spline) wavelets for
  representing the radial displacement function are equally possible and 
  offer advantages in terms of parsimonious representation of local features.
  The quantified holomorphy and quadrature convergence analysis for such 
  representation systems can be performed along the present lines. 
  We refer to \cite{GHS24} for details and computational comparisons, albeit 
  in a $\wn$-implicit setting.
\item The investigation of multi-level algorithms for high-dimensional quadrature
  \cite{DiGaLeSc:17,Z18_2760}. The mathematical justification of these algorithms in the
  present setting requires, however, uniform with respect to $h$ and $\wn$ holomorphic dependence
  of the parametric solution $\refs{u}(\rC;\cdot)$ in norms which are stronger than
  $H^1_k(B_2)$ in \eqref{eq:holo_thm2}.  In Theorem~\ref{thm:existence}, part (ii), we
  already provided such bounds in \eqref{eq:holo_thm2_H2}.
\item 
 The most general case of large wavenumber-independent ``$O(1)$'' shape variations, \emph{cf.}
  the discussion in Section~\ref{ss:nonres}. In this case we expect a $\wn$-dependent,
  pre-asymptotic phase of the convergence of standard Smolyak quadrature. In this phase we may
  observe a dimension-dependent rate up to parameter dimension $O(\wn^{d-1})$. 
  Smolyak constructions based
  on \emph{wavenumber-dependent} (Filon-type) univariate quadrature may provide a partial
  remedy here, 
  as proposed in \cite{wu2022filonclenshawcurtissmolyak}.
\item The construction of efficient sparse polynomial surrogate shape-to-far-field maps
  based on interpolation. 
\item The extension of the approximation results of Section~\ref{sec:AplShUQ} to the
  variance of the far-field pattern. This will confront the difficulty that the arising
  integral will fail to feature an analytic integrand, since we have to integrate the
  squared modulus of a complex-valued function. 
\item \esc{Finally, we note that we expect analogous shape holomorphy results to hold for the transmission problem for the time-harmonic Maxwell equations. Indeed, the Maxwell analogues of the Helmholtz bounds in \cite{MS19} (used in \S\ref{sec:MS}) appear in \cite{CFMS23}, and the Maxwell analogue of the Helmholtz $h$-FEM convergence results of \cite{GS3} (used in \S\ref{sec:ExpBdsFEM}) appear in \cite{CFGS24}.}
\end{itemize}

\begin{appendix}

  \section{Comparison of \esc{Corollary \ref{cor:FEMfar-field}}
    with the results of \cite{EsMe:14}}
\label{sec:appendix1}
 
Remark \ref{rem:esme} described how the following two steps in the proof of Corollary \ref{cor:FEMfar-field} might appear
  over-simplistic:
  \begin{enumerate}
  \item using bounds involving the $L^2(B_2)$ norm of the data $f^{\rm alt}$, while for
    plane-wave scattering $f^{\rm alt} \in H^s(B_2)$ for all $s>0$ (see \eqref{eq:pw_f}),
    and
  \item estimating the integral in the expression \eqref{eq:faft} for the far-field
    pattern using the Cauchy--Schwarz inequality, instead of using, say, a duality
    argument.
  \end{enumerate}

  \rev{Conversely,} the paper \cite{EsMe:14} both
  \begin{enumerate}
  \item bounds the Galerkin error in terms of $\|f\|_{H^s}$ for arbitrary $s>0$, and
  \item bounds $L(u-u_h)$, where $L(u) = \int u \overline{z}$ for $z \in H^{s'}$, using a
    duality argument.
  \end{enumerate}
%

Nevertheless, for the particular case of data $f$ coming from a plane-wave (i.e.,
$f^{\rm alt}$), and the functional $L(\cdot)$ being the far-field pattern expressed as
\eqref{eq:faft}, the results of \cite{EsMe:14} 
do not give a better result than Corollary \ref{cor:FEMfar-field}. 

We now briefly justify this statement; strictly speaking, the results in \cite{EsMe:14} cover the $hp$-FEM applied to the
  constant-coefficient Helmholtz equation, but in principle they can be extended to the
  $h$-FEM (for arbitrary $p$) for the variable-coefficient Helmholtz equation using the
  ideas of \cite{ChNi:20} and \cite{GS3}. In this discussion we are interested in these results applied to the 
FE error $\refs{u}^{\rm alt}_{\rm PML}(\rC;\cdot) - \refs{u}^{\rm alt}_{\rm PML}(\rC;\cdot)_h$ 
(as in  Corollary \ref{cor:FEMfar-field}), 
but to lighten notation we just talk about $u-u_h$. 
Furthermore, in this discussion 
we assume that the norm of the solution operator scales as $\sim k$ (as in Theorem \ref{thm:MS}).

A standard duality argument 
(see, e.g., \cite[Proposition 2.1]{EsMe:14}) 
shows that 
\beq\label{eq:EM1}
\N{u-u_h}_{L^2} \leq C 
 \N{u-u_h}_{H^1_k}
  \sup_{0\neq g \in L^2}
 \min_{v_h \in V_h}\frac{\N{\mathcal{S}^* g - v_h}_{H^1_k}}{\N{g}_{L^2}},
\eeq
and \cite[Lemma 2.3]{EsMe:14} uses similar ideas to prove that 
\beq\label{eq:EM2}
\left|\int (u-u_h)z \right| \leq C \N{u-u_h}_{H^1_k} \min_{v_h \in V_h} \N{\mathcal{S}^* z - v_h}_{H^1_k},
\eeq
where $\mathcal{S}^* :L^2 \to H^1_k$ is the Helmholtz adjoint solution operator.
{The work \cite{EsMe:14} focusses on} these bounds when the Galerkin solution is quasi-optimal, 
{in which case} the $H^1_k$ errors are bounded, uniformly in $k$, by the best approximation error. 
Recall from Theorem \ref{thm:FEM1} and Remark \ref{rem:FEM1} 
that a {sufficient (and empirically necessary)} 
condition for quasioptimality is that $(hk)^p k$ {be} sufficiently small.

If $z\in L^2$, then
\beq\label{eq:EM3}
\min_{v_h \in V_h}\N{\mathcal{S}^* z - v_h}_{H^1_k} \leq C \Big( hk + (hk)^pk \Big) \N{z}_{L^2}
\eeq
by \cite{MeSa:10, MeSa:11}, \cite[Lemma 2.13]{ChNi:20}, \cite[Theorem 1.7]{GS3}. Furthermore,
if $z\in H^{p-1}$ is $k$-oscillatory, in that it satisfies the bound \eqref{eq:koscillatory} (with $\refs{f}$ replaced by $z$), then
\beq\label{eq:EM3a}
\min_{v_h \in V_h}\N{\mathcal{S}^* z - v_h}_{H^1_k} \leq C (hk)^pk \N{z}_{L^2}
\eeq
by \cite[Part (ii) of Theorem 4.4]{EsMe:14}. The bounds \eqref{eq:EM3} and \eqref{eq:EM3a} also hold with $\mathcal{S}^*$ replaced by $\mathcal{S}$ since $\mathcal{S}^* f = \overline{\mathcal{S}\overline{f}}$.


Therefore, if $u$ is the solution of the Helmholtz equation with $k$-oscillatory (in the sense of \eqref{eq:koscillatory})  right-hand side $f\in H^{p-1}$ and $h$ is such that $(hk)^p k\leq C$ for sufficiently-small $C>0$, then the combination of \eqref{eq:EM1}, quasioptimality, \eqref{eq:EM3}, and \eqref{eq:EM3a} implies that
\beq\label{eq:EM4}
\N{u-u_h}_{L^2} \leq C \Big( hk + (hk)^p k \Big)(hk)^p k \N{f}_{L^2}.
\eeq
Furthermore, if $z\in H^{p-1}$ is $k$-oscillatory, 
as in the case of the far-field pattern \eqref{eq:faft}, 
then the combination of \eqref{eq:EM2}, quasioptimality, and \eqref{eq:EM3a} implies that 
\beq\label{eq:EM5}
\left|\int (u-u_h)z \right| \leq C (hk)^p k \N{f}_{L^2} (hk)^p k \N{z}_{L^2}.
\eeq
 The bound \eqref{eq:EM4} is the same as \eqref{eq:temp_final2}, and the bound \eqref{eq:EM5} is no better than \eqref{eq:EM4}
when $\|z\|_{L^2}\sim 1$ 
and when $(hk)^p k$ is fixed (i.e., when one chooses the least restrictive condition on $h$ allowed by the theory as $k\to\infty$). 
That is, in this setting of data coming from a plane-wave {with the quantity of interest} 
being the far-field pattern expressed as {the linear functional} \eqref{eq:faft}, 
the results of \cite{EsMe:14} indeed do not give a better result
than Corollary \ref{cor:FEMfar-field}. 
(Note that the bounds on the error in the functional in \cite[Part (ii) of Corollary 4.5
and Part (ii) of Corollary 4.6]{EsMe:14} assume that $\|z\|_{H^{s'}}$ is independent of
$k$, and thus these results have better $k$-dependence than stated above.)
\end{appendix}

\section*{Acknowledgements}

EAS was supported by EPSRC grant EP/R005591/1.
This work was initiated when EAS was a Nachdiplom lecturer at ETH in Autumn 2021, supported by the Forschungsinstitut f\"ur Mathematik (FIM).


\begin{thebibliography}{10}

\bibitem{AC16}
{\scshape G.~S. Alberti and Y.~Capdeboscq}, {\em Lectures on elliptic methods
  for hybrid inverse problems}, vol.~25 of Cours sp\'ecialis\'es,
  Soci{\'e}t{\'e} Math{\'e}matique de France, Paris, 2018.

\bibitem{AJSZ20_2734}
{\scshape R.~Aylwin, C.~Jerez-Hanckes, C.~Schwab, and J.~Zech}, {\em Domain
  uncertainty quantification in computational electromagnetics}, SIAM/ASA J.
  Uncertainty Quantification, 8 (2020), pp.~301--341.

\bibitem{BNP20}
{\scshape F.~Bonizzoni, F.~Nobile, I.~Perugia, and D.~Pradovera}, {\em
  Least-squares {P}ad\'{e} approximation of parametric and stochastic
  {H}elmholtz maps}, Adv. Comput. Math., 46 (2020), pp.~Paper No. 46, 28.

\bibitem{BrPa:07}
{\scshape J.~Bramble and J.~Pasciak}, {\em {Analysis of a finite PML
  approximation for the three dimensional time-harmonic Maxwell and acoustic
  scattering problems}}, Mathematics of Computation, 76 (2007), pp.~597--614.

\bibitem{BrSc:08}
{\scshape S.~C. Brenner and L.~R. Scott}, {\em The Mathematical Theory of
  Finite Element Methods}, vol.~15 of Texts in Applied Mathematics, Springer,
  3rd~ed., 2008.

\bibitem{Cap12}
{\scshape Y.~Capdeboscq}, {\em On the scattered field generated by a ball
  inhomogeneity of constant index}, Asymptot. Anal., 77 (2012), pp.~197--246.

\bibitem{CLP12}
{\scshape Y.~Capdeboscq, G.~Leadbetter, and A.~Parker}, {\em On the scattered
  field generated by a ball inhomogeneity of constant index in dimension
  three}, in Multi-scale and high-contrast {PDE}: from modelling, to
  mathematical analysis, to inversion, vol.~577 of Contemp. Math., Amer. Math.
  Soc., Providence, RI, 2012, pp.~61--80.

\bibitem{CaPoVo:99}
{\scshape F.~Cardoso, G.~Popov, and G.~Vodev}, {\em Distribution of resonances
  and local energy decay in the transmission problem {II}}, Mathematical
  Research Letters, 6 (1999), pp.~377--396.

\bibitem{CaPoVo:01}
{\scshape F.~Cardoso, G.~Popov, and G.~Vodev}, {\em Asymptotics of the number
  of resonances in the transmission problem}, Communications in Partial
  Differential Equations, 26 (2001), pp.~1811--1859.

\bibitem{CNT15a}
{\scshape J.~E. Castrill\'{o}n-Cand\'{a}s, F.~Nobile, and R.~F. Tempone}, {\em
  Analytic regularity and collocation approximation for elliptic {PDE}s with
  random domain deformations}, Comput. Math. Appl., 71 (2016), pp.~1173--1197.

\bibitem{CM08}
{\scshape S.~N. Chandler-Wilde and P.~Monk}, {\em {Wave-number-explicit bounds
  in time-harmonic scattering}}, SIAM J. Math. Anal., 39 (2008),
  pp.~1428--1455.

\esc{
\bibitem{CFGS24}
{\scshape  T.~Chaumont-Frelet, J.~Galkowski, E.A.~Spence}, {\em Sharp error bounds for edge-element discretisations of the high-frequency Maxwell equations},
  arXiv preprint 2408.04507 [math.NA], 2024.
}

\bibitem{ChGaNiTo:22}
{\scshape T.~Chaumont-Frelet, S.~Gallistl, D.and~Nicaise, and J.~Tomezyk}, {\em
  Wavenumber explicit convergence analysis for finite element discretizations
  of time-harmonic wave propagation problems with perfectly matched layers},
  Communications in Mathematical Sciences, 20 (2022), pp.~1--52.

\bibitem{ChNi:20}
{\scshape T.~Chaumont-Frelet and S.~Nicaise}, {\em Wavenumber explicit
  convergence analysis for finite element discretizations of general wave
  propagation problems}, IMA Journal of Numerical Analysis, 40 (2020),
  pp.~1503--1543.
  
\esc{  
\bibitem{CFMS23}
{\scshape  T.~Chaumont-Frelet, A.~Moiola, E.A.~Spence}, {\em Explicit bounds for the high-frequency time-harmonic Maxwell equations in heterogeneous media}, J. Math. Pure. Appl., 179, (2023) pp.~183-218
}

\bibitem{ChCChS14}
{\scshape A.~Chkifa, A.~Cohen, and C.~Schwab}, {\em High-dimensional adaptive
  sparse polynomial interpolation and applications to parametric {PDE}s},
  Found. Comput. Math., 14 (2014), pp.~601--633.

\bibitem{CCS15}
\leavevmode\vrule height 2pt depth -1.6pt width 23pt, {\em Breaking the curse
  of dimensionality in sparse polynomial approximation of parametric {PDE}s},
  J. Math. Pures Appl. (9), 103 (2015), pp.~400--428.

\bibitem{CC20}
{\scshape M.~A. Chkifa}, {\em Leja, {F}ej\'{e}r-{L}eja and {$\mathfrak
  R$}-{L}eja sequences for {R}ichardson iteration}, Numer. Algorithms, 84
  (2020), pp.~1481--1505.

\bibitem{Ci:91}
{\scshape P.~G. Ciarlet}, {\em Basic error estimates for elliptic problems}, in
  Handbook of numerical analysis, {V}ol.\ {II}, Handb. Numer. Anal., II,
  North-Holland, Amsterdam, 1991, pp.~17--351.

\bibitem{CSZ18}
{\scshape A.~Cohen, C.~Schwab, and J.~Zech}, {\em Shape holomorphy of the
  stationary {N}avier-{S}tokes equations}, SIAM J. Math. Anal., 50 (2018),
  pp.~1720--1752.

\bibitem{CoMo:98a}
{\scshape F.~Collino and P.~Monk}, {\em The perfectly matched layer in
  curvilinear coordinates}, SIAM Journal on Scientific Computing, 19 (1998),
  pp.~2061--2090.

\bibitem{CoKr:83}
{\scshape D.~Colton and R.~Kress}, {\em Integral Equation Methods in Scattering
  Theory}, Wiley, New York, 1983.

\bibitem{COK13}
{\scshape D.~Colton and R.~Kress}, {\em Inverse Acoustic and Electromagnetic
  Scattering Theory}, vol.~93 of Applied Mathematical Sciences, Springer,
  Heidelberg, 2nd~ed., 2013.

  \esc{
\bibitem{DLM22}
{\scshape M.~Dalla~Riva, P.~Luzzini, and P.~Musolino}, {\em Multi-parameter
  analysis of the obstacle scattering problem}, Inverse Problems, 38 (2022),
  Paper No. 055004.}
  
\bibitem{DEI85}
{\scshape K.~Deimling}, {\em Nonlinear functional analysis}, Springer-Verlag,
  Berlin, 1985.

\bibitem{DiGaLeSc:17}
{\scshape J.~Dick, R.~N. Gantner, Q.~T. Le~Gia, and C.~Schwab}, {\em
  {Multilevel higher-order quasi-Monte Carlo Bayesian estimation}},
  Mathematical Models and Methods in Applied Sciences, 27 (2017), pp.~953--995.

\bibitem{QMCBook}
{\scshape J.~Dick, P.~Kritzer, and F.~Pillichshammer}, {\em Lattice
  rules---numerical integration, approximation, and discrepancy}, vol.~58 of
  Springer Series in Computational Mathematics, Springer, Cham, [2022]
  \copyright 2022.

\bibitem{JDTlGChS2016}
{\scshape J.~Dick, Q.~T. Le~Gia, and C.~Schwab}, {\em Higher order
  quasi-{M}onte {C}arlo integration for holomorphic, parametric operator
  equations}, SIAM/ASA J. Uncertain. Quantif., 4 (2016), pp.~48--79.

\bibitem{DLS22}
{\scshape J.~Dick, M.~Longo, and C.~Schwab}, {\em Extrapolated polynomial
  lattice rule integration in computational uncertainty quantification},
  SIAM/ASA J. Uncertain. Quantif., 10 (2022), pp.~651--686.

\bibitem{DyZw:19}
{\scshape S.~Dyatlov and M.~Zworski}, {\em Mathematical theory of scattering
  resonances}, American Mathematical Society, 2019.

\bibitem{EsMe:14}
{\scshape S.~Esterhazy and J.~M. Melenk}, {\em {An analysis of discretizations
  of the Helmholtz equation in L2 and in negative norms}}, Computers \&
  Mathematics with Applications, 67 (2014), pp.~830--853.

\bibitem{Ga:18}
{\scshape J.~Galkowski}, {\em The quantum sabine law for resonances in
  transmission problems}, Pure and Applied Analysis, 1 (2018), pp.~27--100.

\bibitem{GaLaSp:21a}
{\scshape J.~Galkowski, D.~Lafontaine, and E.~A. Spence}, {\em
  Perfectly-matched-layer truncation is exponentially accurate at high
  frequency}, SIAM J. Math. Anal., 55 (2023), pp.~3344--3394.

\bibitem{GLSW22}
{\scshape J.~Galkowski, D.~Lafontaine, E.~A. Spence, and J.~Wunsch}, {\em {The
  $hp$-FEM applied to the Helmholtz equation with PML truncation does not
  suffer from the pollution effect}}, Comm. Math. Sci., 22 (2024), pp.~1761--1816.

\bibitem{GS3}
{\scshape J.~Galkowski and E.~A. Spence}, {\em {Sharp preasymptotic error
  bounds for the Helmholtz $h$-FEM}}, SIAM J. Numer. Anal., 63 (2025), pp.~1-22.

\bibitem{GKS20}
{\scshape M.~Ganesh, F.~Y. Kuo, and I.~H. Sloan}, {\em Quasi-{M}onte {C}arlo
  finite element analysis for wave propagation in heterogeneous random media},
  SIAM/ASA J. Uncertain. Quantif., 9 (2021), pp.~106--134.

\bibitem{GHS24}
{\scshape W.~Gerrit~van Harten and L.~Scarabosio}, {\em Exploiting locality in
  sparse polynomial approximation of parametric elliptic pdes and application
  to parameterized domains}, ESAIM: Mathematical Modelling and Numerical
  Analysis,  (2024).

\bibitem{GKN25}
\rev{{\scshape I.~G. Graham, F.~Y. Kuo, D.~Nuyens, I.~H. Sloan, and E.~A. Spence},
  {\em Quasi-Monte Carlo methods for uncertainty quantification of wave
  propagation and scattering problems modelled by the Helmholtz equation},
  arXiv preprint 2502.12451 [math.NA], 2025.}
  
\bibitem{GPS19}
{\scshape I.~G. Graham, O.~R. Pembery, and E.~A. Spence}, {\em {The Helmholtz
  equation in heterogeneous media: a priori bounds, well-posedness, and
  resonances}}, Journal of Differential Equations, 266 (2019), pp.~2869--2923.

\bibitem{Gr:85}
{\scshape P.~Grisvard}, {\em Elliptic problems in nonsmooth domains}, Pitman,
  Boston, 1985.

\bibitem{HPS15}
{\scshape H.~Harbrecht, M.~Peters, and M.~Siebenmorgen}, {\em Analysis of the
  domain mapping method for elliptic diffusion problems on random domains},
  Numerische Mathematik,  (2016), pp.~1--34.

\bibitem{HS21_2779}
{\scshape F.~Henriquez and C.~Schwab}, {\em {Shape Holomorphy of the Calder\'on
  Projector for the Laplacean in ${\mathbb R}^2$}}, Journ. Int. Equns. Operator
  Theory, 93 (2021).

\bibitem{HSS15}
{\scshape R.~Hiptmair, L.~Scarabosio, C.~Schillings, and C.~Schwab}, {\em Large
  deformation shape uncertainty quantification in acoustic scattering}, Adv.
  Comput. Math., 44 (2018), pp.~1475--1518.

\bibitem{HoScZs:03}
{\scshape T.~Hohage, F.~Schmidt, and L.~Zschiedrich}, {\em {Solving
  time-harmonic scattering problems based on the pole condition II: convergence
  of the PML method}}, SIAM Journal on Mathematical Analysis, 35 (2003),
  pp.~547--560.

\bibitem{IhBa:95a}
{\scshape F.~Ihlenburg and I.~Babu{\v{s}}ka}, {\em {Finite element solution of
  the Helmholtz equation with high wave number Part I: The h-version of the
  FEM}}, Comput. Math. Appl., 30 (1995), pp.~9--37.

\bibitem{IhBa:97}
{\scshape F.~Ihlenburg and I.~Babuska}, {\em {Finite element solution of the
  Helmholtz equation with high wave number part II: the $hp$ version of the
  FEM}}, SIAM J. Numer. Anal., 34 (1997), pp.~315--358.

\bibitem{JerezChSZech2017}
{\scshape C.~Jerez-Hanckes, C.~Schwab, and J.~Zech}, {\em Electromagnetic wave
  scattering by random surfaces: shape holomorphy}, Math. Models Methods Appl.
  Sci., 27 (2017), pp.~2229--2259.

\bibitem{KUS24}
\rev{{\scshape S.~Kuijpers and L.~Scarabosio}, {\em Wavenumber-explicit
  well-posedness of Bayesian shape inversion in acoustic scattering}, arXiv preprint
  2410.23100 [math.AP], 2024.}
  
\bibitem{LaSpWu:21}
{\scshape D.~Lafontaine, E.~A. Spence, and J.~Wunsch}, {\em For most
  frequencies, strong trapping has a weak effect in frequency-domain
  scattering}, Communications on Pure and Applied Mathematics, 74 (2021),
  pp.~2025--2063.

\bibitem{LaSpWu:22}
\leavevmode\vrule height 2pt depth -1.6pt width 23pt, {\em {A sharp
  relative-error bound for the Helmholtz $h$-FEM at high frequency}}, Numer.
  Math., 150 (2022), pp.~137--178.

\bibitem{LaSo:98}
{\scshape M.~Lassas and E.~Somersalo}, {\em {On the existence and convergence
  of the solution of PML equations}}, Computing, 60 (1998), pp.~229--241.

\bibitem{LiWu:19}
{\scshape Y.~Li and H.~Wu}, {\em {FEM and CIP-FEM for Helmholtz Equation with
  High Wave Number and Perfectly Matched Layer Truncation}}, SIAM Journal on
  Numerical Analysis, 57 (2019), pp.~96--126.

\bibitem{Mc:00}
{\scshape W.~C.~H. McLean}, {\em {Strongly elliptic systems and boundary
  integral equations}}, Cambridge University Press, 2000.

\bibitem{MeSa:10}
{\scshape J.~M. Melenk and S.~Sauter}, {\em Convergence analysis for finite
  element discretizations of the {H}elmholtz equation with
  {D}irichlet-to-{N}eumann boundary conditions}, Math. Comp, 79 (2010),
  pp.~1871--1914.

\bibitem{MeSa:11}
\leavevmode\vrule height 2pt depth -1.6pt width 23pt, {\em Wavenumber explicit
  convergence analysis for {G}alerkin discretizations of the {H}elmholtz
  equation}, SIAM J. Numer. Anal., 49 (2011), pp.~1210--1243.

\bibitem{MS19}
{\scshape A.~Moiola and E.~A. Spence}, {\em Acoustic transmission problems:
  wavenumber-explicit bounds and resonance-free regions}, Mathematical Models
  and Methods in Applied Sciences, 29 (2019), pp.~317--354.

\bibitem{Mo:95}
{\scshape P.~Monk}, {\em The near field to far field transformation},
  COMPEL-The international journal for computation and mathematics in
  electrical and electronic engineering, 14 (1995), pp.~41--56.

\bibitem{MUJ86}
{\scshape J.~Mujica}, {\em Complex analysis in {B}anach spaces}, vol.~120 of
  North-Holland Mathematics Studies, North-Holland Publishing Co., Amsterdam,
  1986.

\bibitem{Di:22}
{\scshape {NIST}}, {\em {Digital Library of Mathematical Functions}}.
\newblock Digital Library of Mathematical Functions, \url{http://dlmf.nist.gov/}, 2022.

\bibitem{PoVo:99}
{\scshape G.~Popov and G.~Vodev}, {\em Resonances near the real axis for
  transparent obstacles}, Communications in Mathematical Physics, 207 (1999),
  pp.~411--438.

\bibitem{SCA16}
{\scshape L.~Scarabosio}, {\em Shape uncertainty quantification for scattering
  transmission problems}, ETH Dissertation no. 23574, ETH Z\"urich, 2016.

\bibitem{SCG11}
{\scshape C.~Schwab and C.~J. Gittelson}, {\em Sparse tensor discretizations of
  high-dimensional parametric and stochastic {PDE}s}, Acta Numerica, 20 (2011),
  pp.~291--467.

\bibitem{SjZw:91}
{\scshape J.~Sj\"{o}strand and M.~Zworski}, {\em Complex scaling and the
  distribution of scattering poles}, J. Amer. Math. Soc., 4 (1991),
  pp.~729--769.

\bibitem{SpWu:22}
{\scshape E.~A. Spence and J.~Wunsch}, {\em Wavenumber-explicit parametric
  holomorphy of {H}elmholtz solutions in the context of uncertainty
  quantification}, SIAM/ASA J. Uncertain. Quantif., 11 (2023), pp.~567--590.

\bibitem{St:00}
{\scshape P.~Stefanov}, {\em Resonances near the real axis imply existence of
  quasimodes}, Comptes Rendus de l'Acad{\'e}mie des Sciences-Series
  I-Mathematics, 330 (2000), pp.~105--108.

\bibitem{wu2022filonclenshawcurtissmolyak}
{\scshape Z.~Wu, I.~G. Graham, D.~Ma, and Z.~Zhang}, {\em {A
  Filon-Clenshaw-Curtis-Smolyak rule for multi-dimensional oscillatory
  integrals with application to a UQ problem for the Helmholtz equation}},
  2022.

\bibitem{XIT06}
{\scshape D.~Xiu and D.~M. Tartakovsky}, {\em Numerical methods for
  differential equations in random domains}, SIAM J. Sci. Comput., 28 (2006),
  pp.~1167--1185 (electronic).

\bibitem{Z18_2760}
{\scshape J.~Zech}, {\em Sparse-Grid Approximation of High-Dimensional
  Parametric PDEs}, ETH Dissertation 25683, ETH Z{\"u}rich, 2018.

\bibitem{ZS20_2485}
{\scshape J.~Zech and C.~Schwab}, {\em {Convergence rates of high dimensional
  Smolyak quadrature}}, M2AN, 54 (2020), pp.~1259--1307.

\end{thebibliography}

\end{document}